\documentclass[10pt]{article}

\usepackage{amsthm}
\usepackage{thmtools}
\usepackage[colorlinks=true,linkcolor=blue]{hyperref}
\usepackage{enumitem}
\usepackage{mathrsfs}
\usepackage{bbm}
\usepackage[margin=1.6in]{geometry}
\usepackage{marvosym}
\usepackage{graphicx}

\usepackage{tikz}
\usepackage{url}

\newcommand\spiral{%
  \begin{tikzpicture}[y=.009ex,x=.009ex,yscale=-1]
    \path[fill] (74.2500,379.5469) -- (62.4375,373.3594) .. controls
      (83.4375,365.2031) and (99.6563,356.1563) .. (111.0938,346.2188) ..
controls
      (122.5313,336.2813) and (131.0625,324.9375) .. (136.6875,312.1875) ..
controls
      (142.3125,299.4375) and (145.1250,284.4375) .. (145.1250,267.1875) ..
controls
      (145.1250,256.6875) and (141.9375,247.5000) .. (135.5625,239.6250) ..
controls
      (129.1875,231.7500) and (121.7813,226.3125) .. (113.3438,223.3125) ..
controls
      (104.9063,220.3125) and (95.6250,218.8125) .. (85.5000,218.8125) ..
controls
      (75.3750,218.8125) and (66.0937,220.3125) .. (57.6562,223.3125) ..
controls
      (49.2187,226.3125) and (41.8125,231.7500) .. (35.4375,239.6250) ..
controls
      (29.0625,247.5000) and (25.8750,256.6875) .. (25.8750,267.1875) ..
controls
      (25.8750,277.6875) and (27.3750,286.9687) .. (30.3750,295.0312) ..
controls
      (33.3750,303.0937) and (38.1563,308.9063) .. (44.7188,312.4688) ..
controls
      (51.2813,316.0313) and (59.2500,317.8125) .. (68.6250,317.8125) ..
controls
      (73.6875,317.8125) and (77.5782,316.7813) .. (80.2969,314.7188) ..
controls
      (83.0156,312.6563) and (85.2187,309.7500) .. (86.9062,306.0000) ..
controls
      (88.5937,302.2500) and (89.4375,297.3750) .. (89.4375,291.3750) ..
controls
      (89.0625,288.3750) and (88.4063,285.6563) .. (87.4688,283.2188) ..
controls
      (86.5313,280.7813) and (85.0313,278.9063) .. (82.9688,277.5938) ..
controls
      (80.9063,276.2813) and (78.3750,275.6250) .. (75.3750,275.6250) ..
controls
      (72.9375,275.6250) and (70.5469,277.1250) .. (68.2031,280.1250) ..
controls
      (65.8594,283.1250) and (63.7500,288.4219) .. (61.8750,296.0156) --
      (57.6562,295.8750) .. controls (55.5937,295.8750) and (54.0937,295.4063)
..
      (53.1562,294.4688) .. controls (52.2187,293.5313) and (51.7500,292.5000)
..
      (51.7500,291.3750) .. controls (51.7500,286.5000) and (52.6875,281.7187)
..
      (54.5625,277.0312) .. controls (56.4375,272.3437) and (59.2500,268.5000)
..
      (63.0000,265.5000) .. controls (66.7500,262.5000) and (70.8750,261.0000)
..
      (75.3750,261.0000) .. controls (80.6250,261.0000) and (85.4063,262.1250)
..
      (89.7188,264.3750) .. controls (94.0313,266.6250) and (97.4063,270.0937)
..
      (99.8438,274.7812) .. controls (102.2813,279.4687) and (103.5000,285.0000)
..
      (103.5000,291.3750) .. controls (103.5000,299.6250) and
(102.3750,306.5625) ..
      (100.1250,312.1875) .. controls (97.8750,317.8125) and (94.0313,322.5937)
..
      (88.5938,326.5312) .. controls (83.1563,330.4687) and (76.5000,332.4375)
..
      (68.6250,332.4375) .. controls (57.3750,332.4375) and (47.0625,329.8125)
..
      (37.6875,324.5625) .. controls (28.3125,319.3125) and (21.5625,311.5313)
..
      (17.4375,301.2188) .. controls (13.3125,290.9063) and (11.2500,279.5625)
..
      (11.2500,267.1875) .. controls (11.2500,254.0625) and (14.2500,241.9687)
..
      (20.2500,230.9062) .. controls (26.2500,219.8437) and (35.1563,211.3125)
..
      (46.9688,205.3125) .. controls (58.7813,199.3125) and (71.6250,196.3125)
..
      (85.5000,196.3125) .. controls (99.3750,196.3125) and (112.2187,199.3125)
..
      (124.0312,205.3125) .. controls (135.8437,211.3125) and
(144.7500,219.8437) ..
      (150.7500,230.9062) .. controls (156.7500,241.9687) and
(159.7500,254.0625) ..
      (159.7500,267.1875) .. controls (159.7500,287.4375) and
(156.5625,304.7813) ..
      (150.1875,319.2188) .. controls (143.8125,333.6563) and
(134.0625,346.1250) ..
      (120.9375,356.6250) .. controls (107.8125,367.1250) and (92.2500,374.7656)
..
      (74.2500,379.5469) -- cycle;
  \end{tikzpicture}%
}

\usepackage{amsmath}
\usepackage{amsxtra}
\usepackage{amssymb}
\usepackage{lineno}
\usepackage{turnstile}
\usepackage{verbatim}
\usepackage{enumitem}
\usepackage[retainorgcmds]{IEEEtrantools}
\usepackage{accents}
\usepackage[all]{xy}

\declaretheoremstyle[bodyfont=\sl]{slanted}

\swapnumbers
\declaretheorem[name=Definition,style=definition,qed=$\dashv$,
numberwithin=section]{dfn}
\declaretheorem[name=Definition,style=definition,numbered=no,qed=$\dashv$]{dfn*}
\declaretheorem[name=Definition,style=definition,numbered=no]{dfnnoqed*}

\declaretheorem[name=Theorem,style=slanted,numbered=no]{tm*}

\declaretheorem[name=Proposition,style=slanted,sibling=dfn]{prop}

\declaretheorem[name=Corollary,style=slanted,numbered=no]{cor*}
\declaretheorem[name=Remark,style=definition,sibling=dfn]{rem}

\swapnumbers
\declaretheoremstyle[headfont=\scshape]{claimstyle}

\declaretheorem[name=Claim,style=claimstyle,numbered=no]{clm*}

\declaretheorem[name=Subclaim,style=claimstyle,numbered=no]{sclm*}

\declaretheorem[name=Subsubclaim,style=claimstyle,numberwithin=sclmtwo]{ssclmtwo
}
\declaretheorem[name=Subsubclaim,style=claimstyle,numberwithin=sclmthree]{
ssclmthree}
\declaretheorem[name=Subsubclaim,style=claimstyle,numberwithin=sclmfour]{
ssclmfour}
\declaretheorem[name=Subsubclaim,style=claimstyle,numberwithin=sclmfive]{
ssclmfive}
\declaretheorem[name=Subsubclaim,style=claimstyle,numberwithin=sclmsix]{ssclmsix
}
\declaretheorem[name=Subsubclaim,style=claimstyle,numberwithin=sclmseven]{
ssclmseven}
\declaretheorem[name=Subsubclaim,style=claimstyle,numberwithin=sclmeight]{
ssclmeight}
\declaretheorem[name=Subsubclaim,style=claimstyle,numberwithin=sclmnine]{
ssclmnine}
\declaretheorem[name=Subsubclaim,style=claimstyle,numberwithin=sclmten]{ssclmten
}
\declaretheorem[name=Subsubclaim,style=claimstyle,numbered=no]{ssclm*}

\declaretheoremstyle[headfont=\scshape]{casestyle}

\declaretheorem[name=Case,style=casestyle]{case}

\declaretheorem[name=Subcase,style=casestyle,numberwithin=case]{scase}

\title{$\star$-translation for Varsovian models}
\author{Farmer Schlutzenberg\footnote{afirstname dot alastname at tuwien dot ac dot at, afirstname dot alastname at gmail dot com,
\url{\myurl}}
}

\newcommand{\passive}{\mathrm{pv}}
\newcommand{\moon}{{\text{\tiny\spiral}}}
\newcommand{\bh}{\moon}
\renewcommand{\Moon}{{\text{\footnotesize\spiral}}}
\newcommand{\BH}{\Moon}
\newcommand{\Coll}{\mathrm{Coll}}

\newcommand{\BB}{\mathbb B}
\newcommand{\sub}{\subseteq}

\newcommand{\om}{\omega}
\newcommand{\pow}{\mathcal{P}}
\newcommand{\OR}{\mathrm{OR}}
\newcommand{\Hull}{\mathrm{Hull}}
\newcommand{\cut}{\backslash}
\newcommand{\Tt}{\mathcal{T}}
\newcommand{\Ss}{\mathcal{S}}
\newcommand{\Uu}{\mathcal{U}}
\newcommand{\Vv}{\mathcal{V}}

\newcommand{\rg}{\mathrm{rg}}
\newcommand{\ins}{\trianglelefteq}

\newcommand{\pins}{\triangleleft}

\newcommand{\crit}{\mathrm{cr}}
\newcommand{\rest}{\!\upharpoonright\!}
\newcommand{\com}{\circ}

\newcommand{\lh}{\mathrm{lh}}
\newcommand{\Ult}{\mathrm{Ult}}
\newcommand{\sats}{\models}
\newcommand{\J}{\mathcal{J}}
\newcommand{\HC}{\mathrm{HC}}
\newcommand{\ZFC}{\mathrm{ZFC}}
\newcommand{\es}{\mathbb{E}}

\newcommand{\pred}{\mathrm{pred}}
\newcommand{\id}{\mathrm{id}}
\newcommand{\sq}{\mathrm{sq}}

\newcommand{\conc}{\ \widehat{\ }\ }

\newcommand{\rSigma}{\mathrm{r}\Sigma}

\newcommand{\rPi}{\mathrm{r}\Pi}

\newcommand{\rDelta}{\mathrm{r}\Delta}
\DeclareMathOperator{\Th}{Th}
\newcommand{\card}{\mathrm{card}}
\newcommand{\bfrSigma}{\undertilde{\rSigma}}

\newcommand{\psub}{\subsetneq}

\newcommand{\cHull}{\mathrm{cHull}}

\newcommand{\lpole}{\left\lfloor}
\newcommand{\rpole}{\right\rfloor}
\newcommand{\univ}[1]{\lpole #1\rpole}
\newcommand{\tu}{\textup}
\newcommand{\eqdef}{=_{\mathrm{def}}}
\newcommand{\pvec}{\vec{p}}
\newcommand{\myurl}{https://sites.google.com/site/schlutzenberg/home-1}

\begin{document}

\maketitle

\begin{abstract} The technique of \emph{$\star$-translation} is important in arguments calibrating the strengths of determinacy theories against large cardinals, for example in \cite{dmatm} and \cite{closson}. It has also been used in analysing the internal theory of mice, for example in \cite{odle_v2}, \cite{vm2}, \cite{vmom}.

We give a detailed development of $\star$-translation,
slightly strengthening the large cardinal level in the development of \cite{closson}.
We then develop a variant of $\star$-translation for Varsovian models, in which certain extenders overlapping the Woodin cardinal are total, and used to directly induce extenders on the sequence of the $\star$-translation.\footnote{This is a draft of the final paper. In this draft, the detailed discussion of this variant of $\star$-translation is not yet included. It will be added in due course.}
This variant was used in \cite{vm2} and \cite{vmom}, where it was outlined but not developed in detail.
We use the material to
verify the star-translation hypothesis of \cite{odle_v2} and deduce some self-iterability facts.
\end{abstract}

\tableofcontents

\section{Introduction}\label{sec:intro}

Any proper class $1$-small mouse can compute its own iteration strategy restricted to its universe, and hence satisfies ``I am fully iterable''. On the other hand, proper class mice with Woodin cardinals satisfy ``I am not fully iterable''.
But mice can typically compute significant \emph{fragments} of their own iteration strategy, irrespective of their large cardinals.

In \cite{sile}, Schindler and Steel established a general result along these lines for tame mice $M$ -- in which the extenders $E\in\es_+^M$ do not overlap local Woodin cardinals. A key method employed there was that of \emph{P-construction} , which is a kind of $L[\es]$-contruction, but which uses \emph{all} extenders on the sequence (including partial ones) in a certain interval, and which can only be formed under special circumstances. A typical such situation is when have a limit length $\om$-maximal tree $\Tt$ on some $N\pins M$ such that $\rho_\om^N=\om$, and $\Tt\in\HC^M$, and letting $\delta=\delta(\Tt)$,
then $\delta$ is a strong cutpoint of $M$,
$\OR^N<\delta$, $\Tt$ is definable from parameters over $M|\delta$, and $M|\delta$ is generic for the $\delta$-generator extender algebra of $M(\Tt)$ at $\delta$ (a class forcing of $M(\Tt)$). In this situation, we may want to compute, inside $M$, the correct branch $b$ through $\Tt$ (according to the unique strategy for $N$). If $M$ is tame, then we can, and
in fact, the Q-structure $Q=Q(\Tt,b)$ (which determines $b$) is given by the P-construction of $M$ over $M(\Tt)$. That is, we define $Q$ by setting $M(\Tt)\ins Q$, and then letting $\es_+^Q\rest(\delta,\OR^Q]$ be given by recursively restricting the extenders in $\es^M\rest(\delta,\OR^Q]$ to the relevant segment of the model under construction. This makes sense
up to a certain ordinal height, at which we get a premouse $Q$, which is the correct Q-structure.

In non-tame mice, things are more intricate.
The full results for tame mice established in \cite{sile} have not been generalized to the non-tame level; however, one can still prove significant self-iterability for non-tame mice, using an adaptation of P-construction,
developed by Steel, Neeman and Closson;
see \cite[\S7]{dmatm}, \cite{closson}.
The methods are important in Steel's analysis of derived models of (pre)mice in \cite{dmatm}, and Closson's analogous analysis in \cite{closson}.
(As we will describe shortly, we work basically more generally than \cite{dmatm} and \cite{closson},
though in one minor regard, less generally. However, the more general nature only requires some small tweaks to the analysis.)
We refer to the adaptation of P-construction (for arbitrary short extender mice) as \emph{black-hole-construction}, denoted  $\BH$, as it is
  the inverse of \emph{$\star$-translation} (``star-translation''). Thus, black-hole-construction
is formed by restricting all extenders from $\es^M$ above $\delta$, like P-construction, but there are also other steps involved, which are related to coring.
In fact, in some steps -- such as when $\delta$ is the projectum of the current model and $\delta$ is the critical point of the core embedding --
we do take the core.\footnote{In our context, this will correspond to $\delta$ being a local measurable Woodin cardinal -- something which was assumed to not be the case in \cite{dmatm} and \cite{closson}, but which we allow.}
In other kinds of steps,
one does not core, but cuts the current model $N$ down to a certain segment of itself,
and adds an extender $E$ with $\crit(E)<\delta$, derived from a core embedding of $N$.

Although black-hole-construction is what is used to build the desired Q-structure,
we will primarily focus on its inverse operation of $\star$-translation. We will start with the given Q-structure $Q=Q(\Tt,b)$ in $V$,
apply the $\star$-translation to produce $Q^\star$, and show that $Q^\star\pins M$.
That is our main task.
The black-hole-construction of $Q^\star$ is just $Q$, and because $M$ can compute its black-hole-constructions,
it follows that we have $Q\in M$, and $M$ can identify it correctly (and hence use it to  compute $b$).

The basic defintions with which we work, including in particular $\star$-translation and $\bh$-construction, are laid out in  \cite[\S8]{odle_v2}, and we assume the reader is familiar with that material.
(\S8 of \cite{odle_v2} is independent of the rest of that paper, and can be read in isolation.) As alluded to above, we work basically more generally than in \cite{dmatm} and \cite{closson}, except in one minor regard.
That is, firstly, in \cite{dmatm} and \cite{closson}, it is roughly assumed that we are working below local measurable Woodin cardinals.
More precisely, it is assumed that there is no extender $E$ with $\crit(E)=\delta$ and either (i) $E\in\es_+^Q$, or (ii) there is $F$ such that $F\in\es_+^Q$, $\crit(F)<\delta<\lh(F)$ and
$E\in\es_+(\Ult(Q,F))$.
Secondly, in \cite{dmatm} and \cite{closson}, it is not actually assumed that $\Tt$ is definable from parameters over $M|\delta$ and $M|\delta$ is extender algebra generic over $M(\Tt)$, but something a little more general -- that there is $\xi\leq\delta$
and $g\sub\Coll(\om,\delta)$
which is $M^\Tt_b$-generic,
and such that $\Phi_\delta(\Tt)$ is an element of the least admissible set over $(M(\Tt),g)$, where $\Phi_\delta(\Tt)$
is a natural phalanx associated to $\Tt$ (see \cite[\S7]{dmatm}). On the other hand, what we  assume in this regard is that there is  $P\pins M$ such that $\Tt$ is $P$-optimal for $M$; see \cite[Definition 8.2]{odle_v2}. In particular, $M|\delta\ins P\pins M$,
$\rho_\om^P\leq\delta$,
$\Tt$ is definable from parameters over $P$,
and the natural theory determining $P$ is
extender algebra generic over $M(\Tt)$.
So although we start with a precise correspondence between $\Tt$ and some segment $P$ of $M$ (whereas this is more like ``modulo the next admissible'' in \cite{dmatm} and \cite{closson}),
we do not demand that $P=M|\delta$.

\subsection{Conventions and Notation}\label{subsec:terminology}

For a summary of terminology  see
\cite{fsfni_online_first}.  We just mention a few non-standard and key points here.
We deal with premice $M$ with Mitchell-Steel indexing and fine structure,
except that we allow superstrong extenders on the extender sequence $\es_+^M$
and use the modifications to the fine structure explained in
\cite[\S5]{V=HODX_pub}.

An \emph{$\om$-premouse}
 is a sound premouse $N$ with $\rho_\om^N=\om$;
an \emph{$\om$-mouse} is an
$(\om,\om_1+1)$-iterable $\om$-premouse.
If $N$ is an $\om$-mouse, we write $\Sigma_N$ for the unique
$(\om,\om_1+1)$-strategy for $N$.

For a limit length iteration tree $\Tt$
on an $\om$-premouse and  a $\Tt$-cofinal branch $b$,
$Q(\Tt,b)$ denotes
the Q-structure $Q\ins M^\Tt_b$ for $\delta(\Tt)$,
if this exists, and otherwise
$Q=M^\Tt_b$.

\section{Fine structural correspondence between $R$ and $R^\star$}\label{sec:*-translation}

\begin{dfn}
Let $Q$ be $\star$-valid.
If there is a $Q$-$\delta$-measure,
let $q$ be such that $\rho_{q+1}^Q\leq\delta<\rho_q^Q$,
and otherwise let $q=0$.
We say that $Q$ is \emph{$\star$-iterable}
iff all putative $q$-maximal trees on $\Phi(\Tt)\conc(Q,q)$
with $\crit(E^\Tt_\alpha)\leq\delta$ for all $\alpha+1<\lh(\Tt)$ have finite length and wellfounded models.
\end{dfn}
\begin{rem}\label{rem:correction}
Of course the $\star$-iterability
of $Q$ depends on $\Tt$, although this dependence is not made explicit in the terminology.
We will only be considering forming $Q^\star$ when $Q$ is $\star$-iterable.
Note that  if $\Phi(\Tt)\conc(Q,q)$ is $(q,\om+1)$-iterable then $Q$ is $\star$-iterable, and if $Q$ is $\star$-iterable then $Q^\star$ is defined along
a wellfounded recursion.

In the case that $Q=Q(\Tt,b)$ where $b=\Sigma_N(\Tt)$,
$Q^\star$ will be a well-defined premouse,
with fine structure corresponding to that of $Q$ above $\delta$. In this section we will go through various fine structural calculations which will help us demonstrate this. But in this section there will be no actual iterability assumptions other than $\star$-iterability,
along with first-order assumptions which are consequences of iterability and soundness.
In the end we will prove that $Q^\star$ is well-behaved for the (correct) $Q$ just mentioned, as is $R^\star$ for every $\star$-iterable $R\ins Q$; this will be by induction on $R$, using the various lemmas and iterability.

The fine structural calculations
are essentially as
in \cite{closson}, although we organize things a little differently.
However, there are some small issues ignored in \cite{closson} which we will clarify here.
We will point out these details out when we come to them.

Aside from this, we must also incorporate
the more general features of out setup
mentioned in \S\ref{sec:intro},
but this only requires minor adaptations
of \cite{closson}.

Our (main) task will involve an analysis of the correspondence of fine structure and definability between
$R$ and $R^\star$, for $\star$-suitable segments $R\ins Q$. Toward this, we will consider a sequence of structures intermediate between $R$ and $R^\star$,
and establish a fine structural correspondence between each consecutive pair in the sequence (modulo certain assumptions). We first describe this sequence.
\end{rem}

\begin{dfn}\label{dfn:R_i}
Let $Q$ be $\star$-suitable.
Say that $Q$ is \emph{$\star$-expanding}
iff either $Q$ is active with $\crit(F^Q)<\delta$ or there is a $Q$-$\delta$-measure.  Say that $Q$ is \emph{admissibly buffered}
if for every $\star$-expanding (hence $\star$-suitable) segment
$N\ins Q$, there is an admissible $N'\ins Q$
such that $N\pins N'$.

Now suppose that $Q$ is $\star$-iterable.
Let $m<\om$ and $(\left<Q_i\right>_{i\leq m},\left<N_i,\alpha_i\right>_{i< m})$
be as follows:
\begin{enumerate}
\item $Q_0=Q$.
\item If $i<n$ then $Q_i$ is not admissibly buffered,
then $N_i\ins Q_i$ be the largest $\star$-expanding segment of $Q_i$, then $\alpha_i\in\OR$ and $Q_i=\J_{\alpha_i}(N_i)$ (possibly $\alpha_i=0$), and either:
\begin{enumerate}[label=(\roman*)]
\item there is an $N_i$-$\delta$-measure
and $Q_{i+1}=\Ult_n(N_i,\mu_\delta^{N_i})$
where $n$ is such that $\rho_{n+1}^{N_i}\leq\delta<\rho_n^{N_i}$, or
\item otherwise, and  $Q_{i+1}=\Ult_{n_\kappa}(P_\kappa,F^{N_i})$ where $\kappa=\crit(F^{N_i})$.\footnote{Recall that if there is an $N_i$-$\delta$-measure and $N_i$ is active type 3 then $\delta<\nu(F^{N_i})$. Thus, if $N_i$ is active type 3 and $\delta=\nu(F^{N_i})$,
then $Q_{i+1}=\Ult_{n_\kappa}(P_\kappa,F^{N_i})$
where $\kappa=\crit(F^{N_i})$,
and there could be a $Q_{i+1}$-$\delta$-measure.}
\end{enumerate}
\item $Q_m$ is admissibly buffered.
\end{enumerate}
Let $R_0=Q_0$ and for $i\in(0,m]$,
let
\[ R_i=\J_{\alpha_{i-1}+\alpha_{i-2}+\ldots+\alpha_0}(Q_i)=\J_{\alpha_0}(\J_{\alpha_1}(\ldots\J_{\alpha_{i-1}}(Q_i)\ldots)).\qedhere\]
\end{dfn}

Note that such a sequence exists (and is uniquely determined), by the $\star$-iterability of $Q$.
Although $R_i$ need not be a premouse
(in fact it is a premouse iff $R_i=Q_i$,
since $Q_i$ is not sound),
we will still define $(R_i)^\star=\J_{\alpha_{i-1}+\ldots+\alpha_0}((Q_i)^\star)$.
Note that $Q^\star=(R_i)^\star$ for all $i\leq m$.

\begin{dfn}
For a type 3 premouse $S$ such that $\Ult(S,F^S)$ is wellfounded,
define $S_J$ to be the structure equivalent to $S$, but with its active extender indexed according to $\lambda$-indexing. (So $(S_J)^\passive$ is a Mitchell-Steel indexed premouse, but $\OR^{S_J}=j(\kappa)^{+\Ult(S_J,F^{S_J})}$ where $j:S_J\to \Ult(S_J,F^{S_J})$ is the ultrapower map and $\kappa=\crit(F^J)$,
and $F^{S_J}$ is coded via the standard amenable predicate used for $\lambda$-indexing.)

For $i\leq m$, let $(Q_i)'=(Q_i)_J$ if  $Q_i$ is type 3 with $\crit(F^S)<\delta$, and $(Q_i)'=Q_i$ otherwise,
and let $(R_i)'=\J_{\alpha_{i-1}+\ldots+\alpha_0}((Q_i)')$.
\end{dfn}

 Note that with $Q_m$ as defined above, we have $(Q_m)'=Q_m$.

\subsection*{Summary of fine structural correspondence}

We will analyse the correspondence of fine structure between the following pairs of models:

\begin{enumerate}
\item\label{item:summary_Q_i_to_Q_i'} $Q_i$ and $(Q_i)'$ (see \S\ref{sec:between_R_and_R_J}),
\item\label{item:summary_Q_i'_to_Q_i'[t]} $(Q_i)'$ and $(Q_i)'[t]$\footnote{
The same analysis gives a correspondence between $R$ and $R[t]$, for type 3 $R$
with $\crit(F^R)<\delta$),
but  the correspondence between  $R_J$ and $R_J[t]$ will be the important one in this case.} (see \S\ref{sec:between_R_and_R[t]}),
\item $(N_i)'[t]$ and $(Q_{i+1})'[t]$, for $i<m$ (see \S\ref{sec:R[t]_and_S^R[t]}),
\item $(Q_m)'[t]$ and $(Q_m)^\star$ (recall that $(Q_m)'=Q_m$) (see \S\ref{sec:R[t]_and_R^*_admissibly_buffered}),
\item\label{item:summary_R_i'[t]_R_i+1'[t]} $(R_i)'[t]$ and $(R_{i+1})'[t]$, for $i<m$ (see \S\ref{sec:append_ordinals}),
\item\label{item:summary_R_m'[t]_R_m^star}$(R_m)'[t]$ and $(R_m)^\star$
(recall that $(R_m)'=R_m$ and $(R_m)^\star=Q^\star$) (see \S\ref{sec:append_ordinals}).
\end{enumerate}

Chaining  together items \ref{item:summary_Q_i_to_Q_i'}, \ref{item:summary_Q_i'_to_Q_i'[t]}, \ref{item:summary_R_i'[t]_R_i+1'[t]} for $i<m$, and \ref{item:summary_R_m'[t]_R_m^star}, we will obtain
the desired fine structural correspondence between $Q=Q_0=R_0$ and $Q^\star$.

\subsection{Between $R$ and $R_J$, for type 3 $R$}\label{sec:between_R_and_R_J}

Let $R$ be a type 3 premouse such that $\Ult(R,F^R)$ is wellfounded.

\begin{prop}\label{prop:R_J_1-sound}
$\rho_1^{R_J}=\nu(F^R)$, $p_1^{R_J}=\emptyset$, and $R_J$ is $1$-sound.
\end{prop}
\begin{proof}
Clearly $R_J=\Hull_1^{R_J}(\nu(F^R))$. So it suffices to see that $\rho_1^{R_J}\geq\nu(F^R)$.
Let $\alpha<\nu(F^R)$
with $\kappa^{+R}<\alpha$.
Let
$H=\Hull_1^{R_J}(\alpha)$ and $U_\alpha=\Ult(R|\kappa^{+R},F^R\rest\alpha)$ and
$j_\alpha:R|\kappa^{+R}\to U_\alpha$ the ultrapower map,
and $E_\alpha$ the set of retrictions $j_\alpha\rest R|\xi$
for $\xi<\kappa^{+R}$.
Let $R_{J,\alpha}=(U_\alpha,E_\alpha)$.
Let $U=\Ult(R|\kappa^{+R},F^R)$, as above.
Let $\pi_\alpha:U_\alpha\to U$ be the factor map. Then note that $\rg(\pi_\alpha)=H$ and $\pi_\alpha:R_{J,\alpha}\to
R_J$ is $\Sigma_1$-elementary,
so $R_{J,\alpha}$ is the transitive collapse of $H$.
But since $F^R\rest\alpha\in R\sub R_J$, $R_{J,\alpha}\in R_J$, which suffices.
\end{proof}

In the situation above,
also note that (with $j:R|\kappa^{+R}\to U$ the ultrapower map),
\begin{equation}\label{eqn:char_H=rg(pi_alpha)}\begin{split} H=\rg(\pi_\alpha)&=\big\{[b,f]^{R|\kappa^{+R}}_{F^R}\bigm|b\in[\alpha]^{<\om}\wedge f\in R|\kappa^{+R}\big\}\\
&=\big\{j(f)(b)\bigm|b\in[\alpha]^{<\om}\wedge f\in R|\kappa^{+R}\}.\end{split}\end{equation}

\begin{prop}
For each $\kappa<\nu(F^R)$
and $f:[\kappa]^{<\om}\to R|\nu(F^R)$,
we have
\[\begin{split} f\in R^\sq\iff& f\text{ is }\bfrSigma_1^{R_J}\text{ and bounded in }R|\nu(F^R)\\
\iff& f\in R_J\text{ and bounded in }R|\nu(F^R).\end{split}\]
\end{prop}
\begin{proof}
Since $R|\nu(F^R)=R_J|\nu(F^R)$,
this follows immediately from the fact that $\rho_1^{R_J}=\nu(F^R)$ (Proposition \ref{prop:R_J_1-sound}).
\end{proof}
\begin{prop}
Let $E$ be an $R_J$-extender
with $\crit(E)<\nu(F^R)=\rho_1^{R_J}$. Let $U_0=\Ult_0(R_J,E)$
and $U_1=\Ult_1(R_J,E)$
and $\sigma:U_0\to U_1$ the factor map, and $i_0:R_J\to U_0$ be the ultrapower map. Let $U=\Ult_0(R^\sq,E)$ and $i:R^\sq\to U$ the ultrapower map. Then:
\begin{enumerate}[label=--]
\item $U_0=U_1$ and $\sigma=\id$,
\item $\nu(F^{U_0})=\sup  i_0``\nu(F^R)$,
\item $F^U=F^{U_0}\rest\nu(F^{U_0})$,
\item $i=i_0\rest (R|\nu(F^R))$.
\end{enumerate}
\end{prop}
\begin{proof}
Let $\kappa=\crit(E)<\nu(F^R)$.
Let $\vec{x}\in R_J$ and $g:[\kappa]^{<\om}\to\nu(F^R)$ be $\rSigma_1^{R_J}(\vec{x})$. Let $a\in[\nu(E)]^{<\om}$. Suppose $[a,g]^{R_J,1}_E$
is a generator of $F^{U_1}$.
It is enough to see that there is some $X\in E_a$
such that $g\rest X$ is bounded in $\nu(F^R)$ (and hence by Proposition \ref{prop:R_J_1-sound}, $g\rest X\in R|\nu(F^R)$).
But
 we can fix $\alpha$
such that $\vec{x}\in\rg(\pi_\alpha)$, with notation as in the proof of Proposition \ref{prop:R_J_1-sound}
and line (\ref{eqn:char_H=rg(pi_alpha)}).
But then by line (\ref{eqn:char_H=rg(pi_alpha)}),
if $g(u)$ is an $F^R$-generator
then $g(u)<\alpha$.
But then if $[a,g]^{R_J,1}_{E}$ is an $F^R$-generator,
there must be $X\in E_a$ such that $g(u)<\alpha$ for all $u\in X$,
which suffices.
\end{proof}
\begin{prop}
$\rSigma_1^R$ is recursively equivalent to $\rSigma_2^{R_J}$, for parameters in $R|\nu(F^R)$.
That is, there are recursive functions $\varphi\mapsto\psi_\varphi$ and $\varphi\mapsto\varrho_\varphi$
such that:
\begin{enumerate}
\item for each $\rSigma_1$ formula $\varphi$, $\psi_\varphi$
is $\rSigma_2$, and for all $\vec{x}\in R|\nu(F^R)$, we have
\[ R^\sq\sats\varphi(\vec{x})\iff R_J\sats\psi_\varphi(\vec{x}),\]
\item for each $\rSigma_2$ formula $\varphi$, $\varrho_\varphi$ is $\rSigma_1$, and  for all $\vec{x}\in R|\nu(F^R)$, we have
\[ R_J\sats\varphi(\vec{x})\iff R^\sq\sats\varrho_\varphi(\vec{x}).\]
\end{enumerate}
\end{prop}
\begin{proof}
Since the active extender of $R^\sq$
is coded as $\{F^R\rest\alpha\bigm|\alpha<\nu(F^R)\}$, the proof of Proposition \ref{prop:R_J_1-sound} shows how to recursively reduce $\rSigma_2^{R_J}(\{\vec{x}\})$ to $\rSigma_1^{R^\sq}(\{\vec{x}\})$ (for $\vec{x}\in R|\nu(F^R)$). Conversely, note that
we can express the quantifier ``there is $E$
such that $E=F^R\rest\alpha$
for some $\alpha<\rho_1^{R_J}=\nu(F^R)$''
with $\rSigma_2^{R_J}$,
since given $\alpha\in[\kappa^{+R},\nu(F^R))$, $\Th_1^{R_J}(\alpha)$
 easily codes $F^R\rest\alpha$.
\end{proof}

It immediately follows that:
\begin{prop}
For each $n<\om$,
$R$ is $n$-sound iff $R_J$ is $(n+1)$-sound. Suppose $R$ is $n$-sound
and $\om<\rho_n^R$ \tu{(}so $R_J$ is $(n+1)$-sound\tu{)}. Then:
\begin{enumerate}
\item $\om<\rho_n^R=\rho_{n+1}^{R_J}$,
\item $\rSigma_{n+1}^R$ is recursively equivalent to $\rSigma_{n+2}^{R_J}$, for parameters in $R|\nu(F^R)$,
\item $\rho_{k}^{R}=\rho_{k+1}^{R_J}$ and $p_k^R=p_{k+1}^{R_J}$  for each $k\leq n+1$,
\item $R$ is $(n+1)$-solid iff $R_J$ is $(n+2)$-solid,
\item $R$ is $(n+1)$-sound iff $R_J$ is $(n+2)$-sound,
\item $\Hull_{n+1}^{R}(X)=\Hull_{n+2}^{R_J}(X)\cap(R|\nu(F^R))$ for all $X\sub R|\nu(F^R)$,
\item For each $\kappa<\rho_{n+1}^{R}$
and $f:[\kappa]^{<\om}\to R|\nu(F^R)$,
$f$ is $\bfrSigma_{n+1}^{R}$
iff $f$ is $\bfrSigma_{n+2}^{R_J}$.
Therefore if $E$ is an $R$-extender
with $\crit(E)<\rho_{n+1}^R$,
and $\Ult_{n+1}(R,E)$
and $\Ult_{n+2}(R_J,E)$ are wellfounded, then
\[ \Ult_{n+1}(R,E)_J=\Ult_{n+2}(R_J,E)\]
and $i\sub j$
where $i:R^\sq\to(\Ult_{n+1}(R,E))^\sq$
and $j:R_J\to\Ult_{n+2}(R_J,E)$
are the ultrapower maps.
\end{enumerate}

\end{prop}

\begin{dfn}
A \emph{recalibrated-premouse}
(\emph{recal-premouse})
is a structure $R$
where for some premouse $\widetilde{R}$, either:
\begin{enumerate}
\item $\widetilde{R}$ is non-type 3 and $R=\widetilde{R}$, or
\item $\widetilde{R}$ is type 3 and $R=\widetilde{R}_J$.\qedhere
\end{enumerate}
\end{dfn}

\subsection{Between $R$ and $R[t]$ (or $R_J$ and $R_J[t]$)}\label{sec:between_R_and_R[t]}

The translation in this section is just  the usual translation of fine structure between a premouse and a generic extension thereof. (But note that $\delta$ need not be a cutpoint here.)
In this section, fix $R$ which is either $\star$-valid,
or $R=S_J$ for some $\star$-valid $S$.

\begin{dfn}
Let $\OR^R=\delta+\xi$ and
$t=t_P$.
Then  $R[t]$ denotes the structure
with universe $\univ{\J_\xi^{\es^R}(R|\delta,t)}$, and if $\xi>0$, with constants $R|\delta$ and $t$, and predicates $\es^R$ and $F^R$,  and if $\xi=0$, with universe $\univ{R|\delta}$  and predicates $R|\delta$ and $t$.
If $\xi>0$ then $R[t]$ is amenable
with $t\in R[t]$.
\end{dfn}

\begin{prop}
$R$ and $R[t]$ have the same cardinals and cofinalities $\geq\delta$; in particular,  $\delta$ is regular in $R[t]$.
\end{prop}
\begin{proof}
Because $\card^R(\BB)=\delta$
and $R\sats$ ``$\BB$ has the $\delta$-cc''.
\end{proof}
\begin{prop}\label{prop:inter_Sigma_1_R,R[t]} We have:
\begin{enumerate}
\item
$R\sub R[t]$ and $\univ{R}$,
$\es^R$ and $F^R$ are each $\Delta_1^{R[t]}$.
\item There is an $\rSigma_1$ formula $\psi$ of the ``$R[t]$''-language such that for all
$\rSigma_1$ formulas $\varphi$ of the premouse language and all $v\in R[t]$,
\[ v\in R\wedge R\sats\varphi(v)\iff R[t]\sats\psi(\varphi,v).\]
\item There is an $\rSigma_1$ formula $\varrho$ of the premouse language such that for all $\rSigma_1$ formulas $\varphi$ of the ``$R[t]$''-language and all $\BB$-names $\tau\in R$, we have
\[ R[t]\sats\varphi(\tau_{G_t})\iff\exists p\in G_t\ \big[R\sats\varrho(\alpha_{\Tt},\delta,p,\varphi,\tau)\big].\]
\end{enumerate}
Moreover, the formulas $\psi,\varrho$
can be chosen uniformly in $R$ of the same (premouse) type.
\end{prop}
\begin{proof}
The formula $\varrho$ should assert that there is some $\BB$-name $\tau'$ such that $p$ forces $\varphi'(\tau',\tau)$, where $\varphi(v)=\exists w\varphi'(w,v)$ and $\varphi'$ is $\Sigma_0$.\end{proof}

\begin{dfn}\label{dfn:strongly_forces}
Say $p$ \emph{strongly forces $\varphi(\tau)$}
iff $\varrho(\delta,p,\varphi,\tau)$,
where $\varrho$ is as above.
\end{dfn}

As a corollary of the preceding two propositions, we have:
\begin{prop}\label{prop:matching_1-hulls}
Let $\delta+1\sub X\sub\OR^R$.
Then
$\univ{\Hull_1^R(X)}=\univ{\Hull_1^{R[t]}}\cap R$.
\end{prop}

The fine structure of $R[t]$
is defined above $(R|\delta,t,\alpha_{\Tt})$; so we put all ordinals $\xi<\delta$ into all hulls, etc.
For $\xi\in[\delta,\OR^R]$,
$R[t]|\xi=(R|\xi)[t]$
and $R[t]||\xi=(R||\xi)[t]$
(so $R[t]||\delta$ still has its predicates there).
We always have $\rho_0^{R[t]|\xi}=\xi$ by definition. For $\xi$ such that $R|\xi$ is type 3, we will compare the $d$-fine structure of $R|\xi$ with the fine structure of $R[t]|\xi$;
we use the ``type 2-like'' amenable code for $F^{R|\xi}$ and $F^{R[t]|\xi}$ in this case.
Having $\rho_n^{R[t]|\xi}=\delta$
just means the same thing (by definition) as $\rho_n^{R[t]|\xi}\leq\delta$.

\begin{prop}\label{prop:R,R[t]_fs_match_Sigma_1}
Suppose that $R$ is $1$-universal and $1$-solid and
$\delta\leq\rho_1^R$. Then:
\begin{enumerate}[label=--]
\item $\rho_1^{R[t]}=\rho_1^R$,
\item if $\delta\notin p_1^R$
then $p_1^{R[t]}=p_1^{R}$,
\item if $\delta\in p_1^R$ then
$\rho_1^R=\rho_1^{R[t]}=\delta$
and  $p_1^{R[t]}=p_1^R\cut\{\delta\}$,
\item $R[t]$ is $1$-universal and $1$-solid, and
\item if $R$ is $1$-sound then so is $R[t]$.
\end{enumerate}
\end{prop}
\begin{proof}
If $\delta=\OR^R$ then this is immediate, so assume $\delta<\OR^R$.
Let $t^R=\Th_1^R(\rho_1^R\cup\{p_1^R\})$ and $t^{R[t]}=\Th_1^{R[t]}(\rho_1^R\cup\{p_1^R\})$.
We claim that $t^{R[t]}\notin R[t]$;
since $G_t\in R_t$ and
by Proposition \ref{prop:inter_Sigma_1_R,R[t]}, it suffices to see that $t^R\notin R[t]$.
If $\rho_1^R=\OR^R$ this is trivial, so suppose $\rho_1^R<\OR^R$,
so $\delta\leq\rho_1^R<\OR^R$.
But then if $t^R\in R[t]$
then
by the $1$-universality of $R$,
there is a surjection $\pi:\rho_1^R\to R\cap\pow(\rho_1^R)$ with $\pi\in R[t]$. But note that this yields a surjection $\pi':\rho_1^R\to R[t]\cap\pow(\rho_1^R)$ with $\pi'\in R[t]$, which is impossible.

So $\rho_1^{R[t]}\leq\rho_1^R$.
Suppose $\rho_1^{R[t]}<\rho_1^R$.
So $\delta\leq\rho_1^{R[t]}<\rho_1^R$. Let $u=\Th_1^{R}(\rho_1^{R[t]}\cup\{p_1^{R[t]},\delta\})$. So $u\in R\sub R[t]$.
But  by Proposition \ref{prop:inter_Sigma_1_R,R[t]},
it follows that $\Th_1^{R[t]}(\rho_1^{R[t]}\cup\{p_1^{R[t]}\})\in R[t]$, a contradiction.

So $\rho_1^{R[t]}=\rho_1^R$
and $p_1^{R[t]}\leq p_1^R$.
Suppose $\delta\notin p_1^R$
but $p_1^{R[t]}< p_1^R$
(and in particular, $p_1^R\neq\emptyset$).
Since $\rho_1^R\geq\delta$, therefore
$p_1^R\cap(\delta+1)=\emptyset$.
But then using $1$-solidity in $R$
and as in the previous paragraph, we get a contradiction.
Finally suppose $\delta\in p_1^R$.
So $\rho_1^{R[t]}=\rho_1^R=\delta$
and $\delta=\min(p_1^R)$.
Then $t^{R[t]}$ is equivalent to $(t^{R[t]})'=\Th_1^{R[t]}(\delta\cup\{p_1^R\cut(\delta+1)\})$,
since the language for $R[t]$
can refer to the parameter $\delta$
via a constant. So $(t^{R[t]})'\notin R[t]$, so $p_1^{R[t]}\leq p_1^R\cut(\delta+1)$. Now it follows as before, using
$1$-solidity, that $p_1^{R[t]}=p_1^R\cut(\delta+1)=p_1^R\cut\{\delta\}$.

To see that $R[t]$ is $1$-solid,
note that the relevant solidity witnesses can be computed from the corresponding ones of $R$, together with $G_t$.

For $1$-universality,
let $\rho=\rho_1^R=\rho_1^{R[t]}$
and $p=p_1^R$; so $p_1^{R[t]}=p\cut\{\delta\}$.
Let $C=\cHull_1^R(\rho\cup\{p\})$ and let $\pi:C\to R$ be the uncollapse map.
Let $C'=\cHull_1^{R[t]}(\rho\cup\{p\cut\{\delta\}\})$
and $\pi':C'\to R[t]$ the uncollapse.
Suppose $\delta\in\rg(\pi)$.
Then by Proposition \ref{prop:matching_1-hulls},
$\rg(\pi)=\rg(\pi')\cap R$,
so $C\sub C'$ and $\pi\sub\pi'$,
and by $\Sigma_1$-elementarity,
$C'=C[t]$.
We have $R||\rho^{+R}=C||\rho^{+C}$,
and note then that $R[t]||\rho^{+R[t]}=(C')||\rho^{+C'}$,
as desired. Now suppose instead that $\delta\notin \rg(\pi)$.
We do have $\delta\in C'$,
so $\rg(\pi)\psub\rg(\pi')$.
Then $\delta=\rho_1^R=\rho_1^{R[t]}$
and $\delta\notin p_1^R$.
By the $1$-universality of $R$, $R||\delta^{+R}=C||\delta^{+C}$,
and note then that since $\rg(\pi)\cup\{\delta\}\sub\rg(\pi')$,
$\delta^{+C}=\delta^{+C'}\sub\rg(\pi')$, which gives $1$-universality for $R[t]$.

Finally if $R$ is $1$-sound
then by the preceding parts
and Claim \ref{prop:matching_1-hulls},
so is $R[t]$.
\end{proof}

\begin{prop}\label{prop:R,R[t]_fs_match_Sigma_1_rho_1<delta}
Suppose that $\rho_1^R<\delta$ and $R$ is $1$-solid above $\delta$ and has the $1$-hull property at $\delta$. Then:
\begin{enumerate}[label=--]
\item $\rho_1^{R[t]}=\delta$,
\item $p_1^{R[t]}=p_1^R\cut(\delta+1)$,
\item $R[t]$ is $1$-universal
and $1$-solid,
\item if $R$ is $\delta$-sound
then $R$ is $1$-sound.
\end{enumerate}
\end{prop}
\begin{proof}
This is much like the proof of Proposition \ref{prop:R,R[t]_fs_match_Sigma_1}, but using the $1$-hull property at $\delta$
in the role that $1$-universality played in the proof of Proposition \ref{prop:R,R[t]_fs_match_Sigma_1}.
\end{proof}

We now want to adapt Propositions \ref{prop:inter_Sigma_1_R,R[t]}--\ref{prop:R,R[t]_fs_match_Sigma_1_rho_1<delta} to the case of $\rSigma_{n+1}^R$ and $\rSigma_{n+1}^{R[t]}$  replacing $\rSigma_1^R$ and $\rSigma_{n+1}^R$,
and $\rho_{n+1}$ replacing $\rho_1$, etc, where $n>0$,
assuming that $\rho_n^R>\delta$
and $R$ is $n$-sound.

\begin{prop}\label{prop:inter_Sigma_1_R,R[t]_n+1}
Let $n<\om$. Then there is an $\rSigma_{n+1}$ formula $\psi_{n+1}$ of the ``$R[t]$''-language and an $\rSigma_{n+1}$ formula $\varrho_{n+1}$ of the premouse language, chosen uniformly in the type of $R$ and the integer $n$, such that
if $R$ is $n$-sound with
$\delta<\rho_n^R$, then:
\begin{enumerate}
\item\label{item:formula_conversions_n+1} We have:
\begin{enumerate}[label=\tu{(}\alph*\tu{)}]
\item\label{item:R_truth_in_R[t]_n+1} For all
$\rSigma_{n+1}$ formulas $\varphi$ of the premouse language and all $v\in R[t]$,
we have
\[ v\in R\wedge R\sats\varphi(v)\iff R[t]\sats\psi_{n+1}(\varphi,v).\]
\item\label{item:R[t]_truth_forced_in_R_n+1}  For all $\rSigma_{n+1}$ formulas $\varphi$ of the ``$R[t]$''-language and all $\BB$-names $\tau\in R$, we have
\[ R[t]\sats\varphi(\tau_{G_t})\iff\exists p\in G_t\ \big[R\sats\varrho_{n+1}(\delta,p,\varphi,\tau)\big].\]
\end{enumerate}
\item\label{item:matching_1-hulls_n+1}
Let $\delta+1\sub X\sub\OR^R$.
Then
$\univ{\Hull_{\rSigma_{n+1}}^R(X)}=\univ{\Hull_{\rSigma_{n+1}}^{R[t]}}\cap R$.
\item\label{item:R,R[t],rho,p_n+1_match}
Suppose that $R$ is $(n+1)$-universal and $(n+1)$-solid and
$\delta\leq\rho_{n+1}^R$. Then:
\begin{enumerate}[label=\tu{(}\alph*\tu{)}]
\item $\rho_{n+1}^{R[t]}=\rho_{n+1}^R$,
\item\label{item:R,R[t]_p_n+1_match_no_delta} if $\delta\notin p_{n+1}^R$
then $p_{n+1}^{R[t]}=p_{n+1}^{R}$,
\item\label{item:R,R[t]_p_n+1_match_with_delta} if $\delta\in p_{n+1}^R$ then
$\rho_{n+1}^R=\rho_{n+1}^{R[t]}=\delta$
and  $p_{n+1}^{R[t]}=p_{n+1}^R\cut\{\delta\}$,
\item $R[t]$ is $(n+1)$-universal and $(n+1)$-solid, and
\item if $R$ is $(n+1)$-sound then so is $R[t]$.
\end{enumerate}
\item Suppose  $\rho_{n+1}^R<\delta$ and $R$ is $(n+1)$-solid above $\delta$ and has the $(n+1)$-hull property at $\delta$. Then:
\begin{enumerate}[label=\tu{(}\alph*\tu{)}]
\item $\rho_{n+1}^{R[t]}=\delta$,
\item $p_{n+1}^{R[t]}=p_{n+1}^R\cut(\delta+1)$,
\item $R[t]$ is $(n+1)$-universal
and $(n+1)$-solid,
\item if $R$ is $\delta$-sound
then $R$ is $(n+1)$-sound.
\end{enumerate}
\end{enumerate}
\end{prop}
\begin{proof}
The proof is by induction on $n$. We have already established it for $n=0$, so assume $n\geq 1$.

Part \ref{item:formula_conversions_n+1}\ref{item:R_truth_in_R[t]_n+1}: The main witness to an $\rSigma_{n+1}^R$ assertion is an element of $T_n^R$.
But since $\delta<\rho_n^R=\rho_n^{R[t]}$
 and by part \ref{item:formula_conversions_n+1}\ref{item:R_truth_in_R[t]_n+1} for $n-1$,
 we can compute elements of $T_n^R$
 from corresponding elements of $T_n^{R[t]}$, which yields
an appropriate formula $\psi_{n+1}$.
(Recall that the $\Sigma_1$ fragment of an $\rSigma_{n+1}$ formula is already dealt with by Proposition \ref{prop:inter_Sigma_1_R,R[t]}.)

Part \ref{item:formula_conversions_n+1}\ref{item:R[t]_truth_forced_in_R_n+1}:
This is much like part \ref{item:formula_conversions_n+1}\ref{item:R_truth_in_R[t]_n+1}, applying the inductive hypothesis
to part \ref{item:formula_conversions_n+1}\ref{item:R[t]_truth_forced_in_R_n+1}
for $n-1$ to get canonical names
in $R$
for elements of $T_n^{R[t]}$,
and using Proposition \ref{prop:inter_Sigma_1_R,R[t]}
for the final $\rSigma_1$ assertion.

Part \ref{item:matching_1-hulls_n+1}:
This is an immediate consequence of part \ref{item:formula_conversions_n+1}.

Part \ref{item:R,R[t],rho,p_n+1_match}: Let $\vec{q}=\vec{p}_n^R$.
We have $\vec{q}=\vec{p}_n^{R[t]}$ by induction (and since $\delta<\rho_n^R$, so $\delta\notin p_n^R$).
Let $t^R=\Th_{{n+1}}^R(\rho_{n+1}^R\cup\{p_{n+1}^R,\vec{q}\})$ and $t^{R[t]}=\Th_{{n+1}}^{R[t]}(\rho_{n+1}^R\cup\{\vec{q},p_{n+1}^R\})$.
We claim that $t^{R[t]}\notin R[t]$;
since $G_t\in R_t$ and
by part \ref{item:formula_conversions_n+1}, it suffices to see that $t^R\notin R[t]$.
But this follows from the $(n+1)$-universality of $R$ like in the proof of Proposition \ref{prop:R,R[t]_fs_match_Sigma_1}.
And the fact that $\rho_{n+1}^R\leq\rho_{n+1}^{R[t]}$
is also as in the proof of that claim.
So $\rho_{n+1}^{R[t]}=\rho_{n+1}^{R}$.

The rest of the proof is also as in the proof of Propositions \ref{prop:R,R[t]_fs_match_Sigma_1}
and \ref{prop:R,R[t]_fs_match_Sigma_1_rho_1<delta},
 using that $\vec{p}_n^R=\vec{p}_n^{R[t]}$ when verifying the correspondence of $p_{n+1}^R$ and $p_{n+1}^{R[t]}$.
\end{proof}

\begin{prop}Let $n<\om$ and suppose that $R$ is $n$-sound and $\delta<\rho_n^R$.
Let $\mu<\rho_n^R$ and $f:[\mu]^{<\om}\to R$. Let $E$ be an extender which is close to $R$ and $a\in[\nu(E)]^{<\om}$. Let $E^+$ be the standard fattening of $E$ to an $R[t]$-extender. Then:
\begin{enumerate}
\item There is $X\in (E^+)_a$ with $f\rest X\in R[t]$ iff there is $X\in E_a$ with $f\rest X\in R$.
\item There is $X\in (E^+)_a$ such that $f\rest X$ is $\bfrSigma_n^{R[t]}$ iff there is $X\in E_a$ such that $f\rest X$ is $\bfrSigma_n^{R}$.
\end{enumerate}
\end{prop}

\subsection{Between $R[t]$ and $S^R[t]$, etc ($R$ not admissibly buffered)}
\label{sec:R[t]_and_S^R[t]}

\subsubsection{When there is an $R$-$\delta$-measure}
\label{sec:exists_R-delta-measure}
Suppose there is an $R$-$\delta$-measure.
Since $R$ is $\star$-suitable,
we can fix $r<\om$ such that $\rho_{r+1}^R\leq\delta<\rho_r^R$, and $R$ is $\delta$-sound.
Thus, $S=\Ult_r(R,\mu_\delta^R)$ is $(\delta+1)$-sound,
and $\rho_{r+1}^{S}\leq\delta<\rho_r^S$.
Letting $j:R\to S$ be the ultrapower map, we have $\pvec_{r+1}^S\cut\delta=j(\pvec_{r+1}^R\cut\delta)$. Fix $\alpha_\delta<\delta$ and  $\rSigma_{r+1}$ terms $t_\delta$ and $t_{\delta^+}$
such that $\delta=t^{R}(\pvec_{r+1}^R\cut\delta,\alpha_\delta)$
and if $\delta^{+R}<\rho_0^R$
then
 $\delta^{+R}=t_{\delta^+}^R(\pvec_{r+1}^R\cut\delta,\alpha_\delta)$.

\begin{prop}\label{prop:R-delta-measure_rSigma_r+1_equiv_over_mice}
$\rSigma_{n+1}(\{\pvec_{r+1}^R\cut\delta,\alpha_\delta\})$
statements about parameters in $R|\delta$ are recursively equivalent to
$\rSigma_{n+1}(\{\pvec_{r+1}^{S}\cut\delta,\delta,\alpha_\delta\})$ statements
about those parameters. That is, there is an $\rSigma_{r+1}$ formula $\psi_{r+1}$ such that for all $\rSigma_{r+1}$ formulas $\varphi$ and all $v\in R|\delta$, we have
\[ R\sats\varphi(\pvec_{r+1}^R\cut\delta,\alpha_\delta,v)\iff S\sats\psi_{r+1}(\varphi,\pvec_{r+1}^S\cut\delta,\alpha_\delta,\delta,v),\]
and there is an $\rSigma_{r+1}$ formula $\tau_{r+1}$ such that for all $\rSigma_{r+1}$ formulas $\varphi$ and all $v\in R|\delta$, we have
\[ S\sats\varphi(\pvec_{r+1}^S\cut\delta,\alpha_\delta,\delta,v)\iff R\sats\tau_{r+1}(\varphi,\pvec_{r+1}^R\cut\delta,\alpha_\delta,v).\]
\end{prop}
\begin{proof}
The formula $\psi_{r+1}$ need not actually refer to $\delta$; it can just asserts ``$\varphi(\pvec_{r+1}^S\cut\delta,\alpha_\delta,v)$'',
and since $j(\pvec_{r+1}^R\cut\delta)=\pvec_{r+1}^S\cut\delta$,
 $j(\alpha_\delta)=\alpha_\delta$
and $j(v)=v$, this works.
For $\tau_{r+1}$, we use the statement ``there is $X\in\mu_\delta^R$
and there is $\beta<\rho_r^R$ such that
for all $\delta'\in X$, the theory $\Th_{\rSigma_r}^R(\beta\cup\{\pvec_{r+1}^R\cut\delta\})$ codes a witness to $\varphi(\pvec_{r+1}^R\cut\delta,\alpha_\delta,\delta',v)$''.
Now if $\delta^{+R}<\rho_0^R$ then
we can indeed refer to $\mu_\delta^R$ as needed, since $\delta^{+R}=t_{\delta^+}^R(\pvec_{r+1}^R\cut\delta,\alpha_\delta)$.
Suppose $\delta^{+R}=\rho_0^R$.
Then either $R$ is active type 2 with largest cardinal $\delta$,
or $R$ is active type 3 with largest cardinal $\delta^{+R}$.
In either case, it is easy to express the quantifier ``there is $X\in\mu^R_\delta$'' in a $\Sigma_1$ fashion, which suffices.
\end{proof}

\begin{prop}\label{prop:exists_R-delta-measure_R[t]_S[t]_rSigma_r+1_equiv}
$\rSigma_{r+1}^{R[t]}(\{\pvec_{r+1}^{R[t]},\alpha_\delta\})$ statements about parameters in $R|\delta$ are recursively equivalent to $\rSigma_{r+1}^{S[t]}(\{\pvec_{r+1}^{S[t]},\alpha_\delta\})$ statements about those parameters. That is, there is an $\rSigma_{r+1}$ formula $\psi_{r+1}$ such that for all $\rSigma_{r+1}$ formulas $\varphi$ and all $v\in R[t]|\delta$, we have
\[ R[t]\sats\varphi(\pvec_{r+1}^{R[t]},\alpha_\delta,v)\iff S[t]\sats\psi_{r+1}(\varphi,\pvec_{r+1}^{S[t]},\alpha_\delta,v),\]
and there is an $\rSigma_{r+1}$ formula $\tau_{r+1}$ such that for all $\rSigma_{r+1}$ formulas $\varphi$ and all $v\in R[t]|\delta$, we have
\[ S[t]\sats\varphi(\pvec_{r+1}^{S[t]},\alpha_\delta,v)\iff R[t]\sats\tau_{r+1}(\varphi,\pvec_{r+1}^{R[t]},\alpha_\delta,v).\]
\end{prop}
\begin{proof}
Both $R[t]$ and $S[t]$ can refer directly to the parameter $\delta$ via their language.
So the proposition is just a consequence of Proposition \ref{prop:R-delta-measure_rSigma_r+1_equiv_over_mice}
combined with Proposition \ref{prop:inter_Sigma_1_R,R[t]_n+1} (including the uniformity mentioned in its statement).
\end{proof}

Since $\rho_{r+1}^R\leq\delta$ and likewise for $R[t]$, $S$ and $S[t]$, and $R$ is $\delta$-sound and $S$ is $(\delta+1)$-sound, and hence $R[t]$, $S[t]$ are sound,
the facts above suffice for our purposes in this case.
(We won't ever take ultrapowers of $R$ or $R[t]$ via extenders $E$ with $\crit(E)>\delta$,
and so we don't need to consider functions mapping into $R$ or $R[t]$.)

\subsubsection{When $R$ is type 2, $\crit(F^R)<\delta$ and $R$ has no $\delta$-measure}\label{sec:type_2_translate_step}

Fix a type 2 $\star$-suitable  $R$,
with $\kappa=\crit(F^R)<\delta$,
and suppose there is no $R$-$\delta$-measure.
If $M_\kappa$ is type 3 with $\crit(F^{M_\kappa})<\kappa$ then let
 $(P,k)=((M_\kappa)_J,n_\kappa+1)$,
and otherwise let $(P,k)=(M_\kappa,n_\kappa)$.
Let $S=\Ult_k(P,F^R)$
and $j:P\to S$ the ultrapower map.
So $\delta^{+S}<\OR^S$, and there is no $S$-total $E\in\es_+^S$ with $\crit(E)=\delta$.
Moreover, if $U=\Ult_{n_\kappa}(M_\kappa,F^R)$ is type 3 with $\crit(F^U)<\delta$,
then $S=U_J$; otherwise $S=U$.

In this section we will translate between the fine structure of $R[t]$
and that of $S[t]$.

\begin{prop}\label{prop:tp_2_overlap_no_delta-meas_R[t]_S[t]}  Set
$z = (\pvec_k^{S[t]},p_{k+1}^{S[t]}\cut j(\kappa),\kappa,s_0,\alpha_0)$
where
 $s_0\in[\OR]^{<\om}$ is the least $s$ such that
$F_\downarrow^R=i_{F^R}(f)(s)$
for some $f\in R$,
 $f_0$ is the least function
 witnessing the choice of $s_0$,
and $\alpha_0$ be the position of $f_0$ in $R$-order. \tu{(}So $\alpha_0<\kappa^{+R}<\delta$.\tu{)} There is an $\rSigma_{k+1}$  formula $\psi$
and an $\rSigma_1$ formula $\tau$ such that:
\begin{enumerate}
\item\label{item:tp_2_R[t]_to_S[t]_R[t]_sub_S[t]} $R[t]\sub S[t]$
and $\univ{R}$, $\es^{R}$
are $\Sigma_1^{S[t]}(\{\nu(F^R)\})$,
and $\OR^R=\nu(F^R)^{+S[t]}$.
\item\label{item:truth_conversion_one-step_star_active_overlapping}
$\rSigma_{k+1}^{S[t]}(\{z,\alpha_\Tt\})$
statements about parameters in $R[t]$
are recursively equvalent to $\rSigma_1^{R[t]}(\{\alpha_\Tt\})$ statements about those parameters, as witnessed by $\psi,\tau$. In fact, there is some $m<\om$ such that
for all $\rSigma_1$ formulas $\varphi$
and $v\in R[t]$, we have
\[ R[t]\sats\varphi(v,\alpha_\Tt)\iff S[t]\sats\psi(\varphi,v,z,\alpha_\Tt) \]
and for all $\rSigma_{k+1}$ formulas $\varphi$ and $v\in R[t]$, we have
\[ S[t]\sats\varphi(v,z,\alpha_\Tt)\iff R[t]\sats\tau(m,\varphi,v,\alpha_\Tt).\]
The choice of $\psi,\tau$ are moreover uniform in $R$. \tu{(}We could have alternatively replaced $m$ with the parameter $\alpha_0$,
making things fully uniform modulo that extra parameter.\tu{)}

\item\label{item:type_2_R_and_S^R_correspond} With $z$ as in part \ref{item:truth_conversion_one-step_star_active_overlapping} and $X\sub R[t]$, we have
\[ R[t]\cap\Hull^{S[t]}_{\rSigma_{k+1}}(X\cup\{z,\alpha_\Tt\})=\Hull_{\rSigma_1}^{R[t]}(X\cup\{\alpha_\Tt\}).\]
\item\label{item:type_2_R_and_S^R_functions_into_R[t]_and_S[t]} For $\kappa<\OR^{R[t]}$
and $f:[\kappa]^{<\om}\to R[t]$, we have that
$f$ is $\bfrSigma_k^{S[t]}$ iff  $f\in R[t]$.
\end{enumerate}
\end{prop}
Before commencing with the proof, let us point out that Proposition \ref{prop:tp_3_overlap_no_delta-meas_R_J[t]_S[t]} below
is a variant of \ref{prop:tp_2_overlap_no_delta-meas_R[t]_S[t]} in which the type 2 $R$
above is replaced by $R_J$, for a type 3 premouse $R$. The statement and proof of \ref{prop:tp_3_overlap_no_delta-meas_R_J[t]_S[t]} is just a simplification of \ref{prop:tp_2_overlap_no_delta-meas_R[t]_S[t]}, so while we give its full statement, we will omit its proof completely.
\begin{proof}
Let us first give the point how how we can compose the formula $\tau(m,\varphi,v,\alpha_\Tt)$ (in the parameter $\alpha_\Tt$).
Working in $R[t]$,
we can first identify, in a $\Sigma_1^{R[t]}$ fashion, the appropriate representative of the parameter $z$.
For using $F^{R}$, we can identify $\kappa=\crit(F^{R})$,
and from $\kappa$ and $(t,\alpha_\Tt)$ (recall that $t$ is supplied by the predicates of $R[t]$) we can identify $P=M_\kappa$ and $k=n_\kappa$,
and hence a coded version of $S=\Ult_k(P,F^{R})$.
We also have a coded version of $j:P\to S$.
We know that $\pvec_k^{S[t]}=\pvec_k^S=j(\pvec_k^P)$
and $p_{k+1}^{S[t]}\cut j(\kappa)=p_{k+1}^{S}\cut j(\kappa)=j(p_{k+1}^P\cut\kappa)$,
yielding codes for these parameters.
And of course $\kappa$
is represented by $[\{\kappa\},\id]$.

Note also that from any parameter $v\in R[t]$ we can identify, in a $\Sigma_1^{R[t]}$ fashion, a
$\BB_\delta$-name $\dot{v}$
and a standard representative $[a_{\dot{v}},f_{\dot{v}}]$
for $\dot{v}$ (recall we have direct access to $G_t$ in $R'[t]$). In particular, for  $p\in\BB_\delta$,
we have a standard representative $[a_p,f_p]$ for $p$,
and $[a_\delta,f_\delta]$ for $\delta$.

To identify $(s,\alpha_0)$,
we define the sequence $(t_0,g_0),(t_1,g_1),\ldots$,
where $g_0$ is the least function such that for some $t\in[\nu(F^R)]^{<\om}$,
we have $F_\downarrow^R=[t,g_0]$,
and $t_0$ is the least witness, and then $g_1$ is the least function such that for some $t\in[\nu(F^R)]^{<\om}$
with $t<t_0$, we have $[t,g_1]=F_\downarrow^R$,
and $t_1$ is the least witness, etc. Letting $m<\om$ be largest such that $t_m$ is defined, note that $s=t_m$, $f_0=g_m$
and $\alpha_0=\alpha_{g_m}$ is the position of $g_m$ in the order of construction of $R$. The sequence $\left<(t_n,g_n,\alpha_{g_n})\right>_{n\leq m}$ is, moreover, $\rSigma_1^{R[t]}(\{F^R_\downarrow\})$,
and we have $F^R_\downarrow$ available via the language for premice. Given $n\leq m$,
let $z_{n}$ be the version of $z$ where we use $t_n,\alpha_{g_n}$ in place of $s,\alpha_0=\alpha_{f_0}$.

Now since $\delta<\rho_k^S$,
by \S\ref{sec:between_R_and_R[t]},
for $\rSigma_{k+1}$ formulas $\varphi$,
we have $S[t]\sats\varphi(v,z,\alpha_\Tt)$
iff there is $p\in G_t$
such that $p$ strongly forces $\varphi(\dot{v},\check{z},\check{\alpha}_\Tt)$.
So
we take $\tau(n,\varphi,v,\alpha_\Tt)$
to be the formula asserting
\begin{enumerate}[label=--]
\item ``$(t_n,g_n)$ is defined'' and
\item ``there is some $p\in G_t$
such that
\begin{equation}\label{eqn:Ult(P)_thinks_p_strongly_forces_varphi} \Ult_k(P,F^R)\sats\text{`} p\text{ strongly forces }\varphi(\dot{v},\check{z}_{n},\check{\alpha}_\Tt)\text{ with respect to }\BB_\delta\text{' }\text{''},\end{equation}
\end{enumerate}
where we use our code for the ultrapower and our representatives to refer to these objects appropriately. (We will just apply it with $n=m$.)
This statement is $\rSigma_1^{R[t]}$
because we have already identified the relevant parameters,
and given $p\in\BB_\delta^R$, we
 we have identified some $a\in[\nu(F^R)]^{<\om}$
and $f_p,f_{\dot{v}},f_{\check{z}_n},f_{\check{\alpha}_\Tt},f_{\BB_\delta}$ (for $p\in\BB_\delta^R$),
such that $p=[a,f_p]$, $\dot{v}=[a,f_{\dot{v}}]$, etc,
and since the statement modelled by $\Ult_k(P,F^R)$ in   line (\ref{eqn:Ult(P)_thinks_p_strongly_forces_varphi})
is  $\rSigma_{k+1}$, it is equivalent to the existence
of a set $X\in (F^R)_a$
such that $P\sats$ ``there is $\beta<\rho_k^P$
such that for all $u\in X$,
the theory $\Th^P_{\rSigma_k}(\beta\cup\{\pvec_k^P,f(u),q\})$ codes a witness to the statement
\[ \text{`}f_p(u)\text{ strongly forces }\varphi(f_{\dot{v}}(u),f_{\check{z}_n}(u),f_{\check{\alpha}_\Tt}(u))\text{ with respect to }f_{\BB_\delta}(u)\text{' } \text{''}.\]

Conversely, we want to compose $\psi(\varphi,v,z,\alpha_\Tt)$.
For this, we have already computed $\univ{R[t]}$ appropriately,
and $R^\passive[t]=S[t]||\OR^{R}$
and $\OR^{R}=\nu(F^R)^{+S[t]}$.
For this, we need to identify
the elements of $F^{R}$ appropriately,
and also identify $F_\downarrow^R$ (since this is identified via a constant interpreted in $R$).
We will recover the elements of $F^R$ by considering
\begin{equation}\label{eqn:range_of_j_Hull} \Hull_{k+1}^{S}(\kappa\cup\{\pvec_{k+1}^{S}\cut j(\kappa)\})=\rg(j) \end{equation}
where $j:P\to S$ is the ultrapower map. That is, since we have $\kappa$ available (in $z$),
we can identify $P,k$,
and hence identify the least $\alpha<\kappa$
and $\rSigma_{k+1}$ term
$u$ such that $\kappa=u^P(\alpha,\pvec_{k+1}^P\cut\kappa)$.
Thus, $j(\kappa)=u^S(\alpha,\pvec_{k+1}^S\cut\OR^R)$.
Now since $\rSigma_{k+1}^S$
is uniformly $\rSigma_{k+1}^{S[t]}$,
and we have the parameter $z$ available,
we have that, given $E\in R$, $E$ is a fragment of $F^R$ represented in the amenable code for $F^R$ iff
 there are $\lambda,\xi,f$ in the hull in line (\ref{eqn:range_of_j_Hull})
such that $\lambda=u^S(\alpha,\pvec_{k+1}^S\cut\OR^R)$
and $\lambda<\xi$
and $\rho_\om^{S|\xi}=\lambda$
and $f:\lambda\to\pow([\lambda]^{<\om})\cap S|\xi$ is the canonical enumeration
definable over $S|\xi$,
and letting $\bar{\xi}$
be such that $\rho_\om^{S|\bar{\xi}}=\kappa$
and there is an elementary embedding $\bar{j}:S|\bar{\xi}\to S|\xi$ (note this requires that $j(\bar{\xi})=\xi$
and $\bar{j}\sub j$),
then
$E$ is the partial extender measuring
$\pow([\kappa]^{<\om})\cap S|\bar{\xi}$
derived from $\bar{j}$,
of length $\nu(F^R)$.
(Note that $\max(s)+1=\nu(F^R)$, and we have $s$ available in $z$, so can refer to $\nu(F^R)$.)
Now using such fragments $E$ and the parameters $(\alpha_0,s)$, we can recover $F_\downarrow^R$
(we recover $f_0$ from $\alpha_0$, then letting $E$ be as above with $\bar{\xi}$ large enough, $i_E(f_0)(s)=F_\downarrow^R$). Finally, we then have that $\varphi(v)$ (which may refer to $F_\downarrow^R$ via the premouse language) holds in $R[t]$
iff there is some $E$ as above which witnesses  $\varphi(v)$, which leads to an appropriate formulation of $\psi(\varphi,v,z,\alpha_\Tt)$.

Part \ref{item:type_2_R_and_S^R_correspond} is an easy consequence of the earlier parts.

Part \ref{item:type_2_R_and_S^R_functions_into_R[t]_and_S[t]} holds because $R[t]^{\passive}=S[t]|\nu(F^R)^{+S[t]}$ and $\OR^{R[t]}<\rho_k^{S[t]}$.
\end{proof}

\newcommand{\dfs}{\mathrm{d}}
\begin{rem}\label{rem:tp_2_R_charac_t^R}
Suppose $R$ is type 2, $\crit(F^R)<\delta\leq\rho_1^R$, and $1$-sound,
hence Dodd-sound.
By \cite{fsfni_online_first},
then $\rho_1^R=\tau^R$
and letting $s_0$ be least such that $F_\downarrow^R=i_{F^R}(f)(s_0,p_1^R)$ for some $f\in R$, we have $t^R=p_1^R\cup (s_0\cut\tau^R)$.
\end{rem}

\begin{dfn}
Suppose $R$ is $\star$-suitable, type 2, $\crit(F^R)<\delta$, and $R$ has no $\delta$-measure.
We say that $R[t]$ is \emph{Dodd-sound $>\delta$} iff either
\begin{enumerate}[label=--]
\item $\rho_1^R\geq\delta$
and $R$ is Dodd-sound, or
\item  $F^R$ is generated by $\rho_1^{R[t]}\cup\{\delta,s\}$ for some $s\in[\OR]^{<\om}$,
(so $t^R\cut(\delta+1)$ is least such),
and the Dodd-solidity witnesses for $t^R\cut(\delta+1)$ are in $R[t]$ (not necessarily in $R$!); that is,
\[ F^R\rest(\alpha\cup\{t^R\cut(\alpha+1)\})\in R[t] \]
for each $\alpha\in t^R\cut(\delta+1)$.\qedhere
\end{enumerate}
\end{dfn}

So if $R$ is Dodd-sound then $R[t]$ is Dodd-sound $>\delta$, but the converse need not hold.

\begin{prop}\label{prop:tp_2_no_delta-meas_Dodd-soundness}
Suppose $R$ is $\star$-suitable, type 2, $\crit(F^R)<\delta$, and $R$ has no $\delta$-measure. Suppose $R[t]$ is $1$-sound and $>\delta$-Dodd-sound.
Then:
\begin{enumerate}
\item\label{item:R_tp_2_no_delta-meas_S[t]_sound} $S[t]$ is $(k+1)$-sound,
\item\label{item:R_tp_2_no_delta-meas_rho_k+1^S[t]} $\rho_{k+1}^{S[t]}=\rho_1^{R[t]}$
and
\item\label{item:R_tp_2_no_delta-meas_p_k+1^S[t]} $p_{k+1}^{S[t]}=j(p_{k+1}^{P}\cut\kappa)\conc (t^R\cut(\delta+1))$.\footnote{The presence of the fragment $t^R\cut(\delta+1)$  of the Dodd parameter of $F^R$ within $p_{k+1}^{S[t]}$ is ignored in
 \cite[Theorem 1.2.9(d''), with $j=1$]{closson}. Adopting the notation used at that point of \cite{closson}, assuming
$\mathcal{P}$ is $1$-sound and Dodd-sound, it should be
$p_{n+1}(\mathcal{P}[g]^*)=j(p_{n+1}^R\cut\kappa)\conc q$
where $j:R\to\Ult_n(A,F^\mathcal{P})$ is the ultrapower map
for the relevant $A$
and $\kappa=\crit(j)$, and $q=t^\mathcal{P}\cut\delta$
where $t^{\mathcal{P}}$ is the Dodd parameter of $\mathcal{P}$.}
\end{enumerate}
Suppose further that
$P=M_\kappa$ is active type 2 and $\crit(F^P)<\kappa=\crit(F^R)$. Then:
\begin{enumerate}[resume*]
 \item\label{item:R_tp_2_no_delta-meas_P_tp_2_S_tp_2} $S$ is active type 2, $\crit(F^S)<\delta$ and $S[t]$ is $>\delta$-Dodd-sound,
and
\item\label{item:R_tp_2_no_delta-meas_cases_with_rho}
either:
\begin{enumerate}[resume*]
\item $\kappa<\rho_1^P$
and $j(\kappa)<\rho_1^S=\rho_1^{S[t]}$
and $S$ is Dodd-sound, or
\item $\rho_1^P\leq\kappa$
and $\rho_1^{S[t]}=\max(\rho_1^R,\delta)$
and $F^S$ is generated by \[ \{\delta\}\cup \rho_1^{S[t]}\cup j(t^P)\cup (t^R\cut(\delta+1)) \]
and the Dodd-solidity witnesses for $j(t^P)\cup (t^R\cut(\delta+1))$ are in $S[t]$.
\end{enumerate}
\end{enumerate}
\end{prop}
\begin{proof}
Parts \ref{item:R_tp_2_no_delta-meas_S[t]_sound}--\ref{item:R_tp_2_no_delta-meas_p_k+1^S[t]}:
We have $S=\Hull_{k+1}^S(\tau^R\cup\{j(\pvec_{k+1}^P\cut\kappa),t^R\})$,
since this hull contains every element of $\rg(j)$,
and every element of $\tau^R\cup\{t^R\}$, and hence every element of $\OR^R$. Therefore it suffices
to see that $\rho_{k+1}^{S[t]}=\max(\delta,\tau^R)$ and $S[t]$ has the relevant solidity witnesses corresponding to $t^R$.
But it has the corresponding fragments of $F^R$ witnessing Dodd-solidity, since they are in $R\sub S$,
and it has $P$,
and from these it computes the relevant solidity witnesses.
Similarly, if $\delta<\tau^R$ then since $S$ has the fragments of $F^R$ corresponding to each $\alpha<\tau^R$, it has the relevant theories to therefore establish that $\rho_{k+1}^{S[t]}=\tau^R$.

Parts \ref{item:R_tp_2_no_delta-meas_P_tp_2_S_tp_2}, \ref{item:R_tp_2_no_delta-meas_cases_with_rho}: If $\kappa<\rho_1^P$ this is immediate.
So suppose $\rho_1^P\leq\kappa$. So $k=0$ and we already know that $\rho_1^{S[t]}=\rho_1^{R[t]}=\max(\rho_1^R,\delta)$. Since $P$ is $\kappa$-sound and $\kappa$-Dodd-sound, $F^P$ is
generated by $\kappa\cup t^P$.
And $F^R$ is generated by $\{\delta\}\cup\rho_1^{R[t]}\cup(t^R\cut(\delta+1))$, so
 $F^S$ is  generated by $\{\delta\}\cup\rho_1^{S[t]}\cup j(t^P)\cup t^R\cut(\delta+1)$.
And
 the $F^P$-Dodd-solidity witnesses for $t^P\cut\kappa$ are in $P$,
which yield the $F^S$-Dodd-solidity witnesses for $j(t^P)\cut\kappa$ in $S$, hence in $S[t]$, and since the $F^R$-Dodd-solidity witnesses for $t^R\cut(\delta+1)$ are in $R[t]$, these are also in $S[t]$,
and since $P$ is also in $S[t]$,
these can be used there to recover
the $F^S$-Dodd-solidity witnesses
for $t^R\cut(\delta+1)$.
\end{proof}

\begin{prop}\label{prop:tp_2_1-sound_R[t]_translate_rSigma_1_about_p_1}
Suppose  $\delta\leq\rho_1^R$ and $R[t]$ is $1$-sound and $>\delta$-Dodd-sound \tu{(}so $S[t]$ is $(k+1)$-sound\tu{)}. Let $s_0,\alpha_0$ be as in Proposition \ref{prop:tp_2_overlap_no_delta-meas_R[t]_S[t]}
and $s'_0=s_0\cap\tau^R$ \tu{(}so $\alpha_0<\delta$ and $s'_0\in[\tau^R]^{<\om}$\tu{)}.
\begin{enumerate}
\item $\rho_{k+1}^{S[t]}=\rho_1^{R[t]}=\rho_1^R=\tau^R$ and $S[t]$ is $(k+1)$-sound.
\item If $S$ is type 2 then $S[t]$  is $>\delta$-Dodd-sound.
\item\label{item:tp_2_rSigma_1^R[t]_and_rSigma_k+1^S[t]_in_nice_params} $\rSigma_{k+1}^{S[t]}(\{\pvec_{k+1}^{S[t]},s'_0,\alpha_0,\alpha_\Tt\})$ statements about parameters in $R[t]$
are recursively equivalent to
$\rSigma_1^{R[t]}(\{p_1^{R[t]},\alpha_\Tt\})$ statements about those parameters.
That is, there is an $\rSigma_{k+1}$ formula
$\psi_{k+1}$  such that for all $\rSigma_1$ formulas $\varphi$ and $v\in R[t]$, we have
\[ R[t]\models\varphi(v,p_1^{R[t]},\alpha_\Tt)\iff S[t]\models\psi_{k+1}(\varphi,v,\pvec_{k+1}^{S[t]},s'_0,\alpha_0,\alpha_\Tt),\]
and there is an $\rSigma_1$
formula $\tau_1$ such that for all $\rSigma_{k+1}$ formulas $\varphi$ and $v\in R[t]$, we have
\[ S[t]\models\varphi(v,\pvec_{k+1}^{S[t]},s_0',\alpha_0,\alpha_\Tt)\iff R[t]\models\tau_{n+1}(\varphi,v,p_1^{R[t]},\alpha_\Tt).\]
\item\label{item:tp_2_F_downarrow_def_at_rSigma_k+2^S[t]} $\{F_\downarrow^R\}$ is $\rSigma_{k+2}^{S[t]}(\{\pvec_{k+1}^{S[t]},\kappa\})$.
\item\label{item:tp_2_rSigma_2^R[t]_and_rSigma_k+2^S[t]} $\rSigma_{k+2}^{S[t]}(\{\pvec_{k+1}^{S[t]},\kappa,\alpha_\Tt\})$ statements about parameters in $R[t]$
are recursively equivalent to $\rSigma_2^{R[t]}(\{p_1^{R[t]},\alpha_\Tt\})$ statements about parameters in $R[t]$.
 That is, there is an $\rSigma_{k+2}$ formula $\psi_{k+2}$ such that for all $v\in R[t]$ and $\rSigma_2$ formulas $\varphi$, we have
 \[ R[t]\sats\varphi(v,p_1^{R[t]},\alpha_\Tt)\iff S[t]\sats\psi_{k+2}(\varphi,v,\pvec_{k+1}^{S[t]},\kappa,\alpha_\Tt),\]
 and there is an $\rSigma_2$ formula
 $\tau_2$ such that for all $v\in R[t]$
 and $\rSigma_{k+2}$ formulas $\varphi$, we have
 \[ S[t]\sats\varphi(v,\pvec_{k+1}^{R[t]},\kappa,\alpha_\Tt)\iff R[t]\sats\tau_2(\varphi,v,p_1^{R[t]},\alpha_\Tt).\qedhere\]
\end{enumerate}
\end{prop}
\begin{proof}
Part \ref{item:tp_2_rSigma_1^R[t]_and_rSigma_k+1^S[t]_in_nice_params}: This is almost by Propositions \ref{prop:tp_2_overlap_no_delta-meas_R[t]_S[t]} and \ref{prop:tp_2_no_delta-meas_Dodd-soundness}, except that the parameters don't quite match up. Let us just examine that aspect.

First consider formulating $\psi_{k+1}$. For this we basically use the method of proof of Proposition \ref{prop:tp_2_overlap_no_delta-meas_R[t]_S[t]}, but we need to identify $p_1^{R[t]}$ and $F_\downarrow^R$ in an
 $\rSigma_{k+1}(\{\pvec_{k+1}^{S[t]},s_0',\alpha_0\})$ fashion. (Note that we can identify $\kappa$ from $\alpha_0$, since $\kappa\leq\alpha_0<\kappa^{+R}$.) But
 by
 Proposition \ref{prop:tp_2_no_delta-meas_Dodd-soundness} part \ref{item:R_tp_2_no_delta-meas_p_k+1^S[t]},
we have $p_{k+1}^{S[t]}=j(p_{k+1}^{P}\cut\kappa)\conc (t^R\cut(\delta+1))$.
By Remark \ref{rem:tp_2_R_charac_t^R},
$t^R=p_1^R\cup s'_0$ and
$F_\downarrow^R=i_{F^R}(f_{\alpha_0}^R)(s_0)$. So from $p_{k+1}^{S[t]}$ (and since the signature of $S[t]$ provides the point $\delta$) we can recover $p_1^R$ and hence $p_1^{R[t]}\sub p_1^R$, and from $p_{k+1}^{S[t]}$ and $s_0'$ (and possibly the point $\delta$)
we can recover the full $s_0$,
and then using $\alpha_0$
and the method of proof of Proposition \ref{prop:tp_2_overlap_no_delta-meas_R[t]_S[t]}, hence recover $F_\downarrow^R$.

In the converse direction,
just use the methods of the proof of Proposition \ref{prop:tp_2_overlap_no_delta-meas_R[t]_S[t]} to identify $\alpha_0$ and $s'_0$.

Part \ref{item:tp_2_F_downarrow_def_at_rSigma_k+2^S[t]}: By the proof of Proposition \ref{prop:tp_2_overlap_no_delta-meas_R[t]_S[t]}, we can identify the  fragments of $F^R$ in the amenable extender predicate of $R$
in an $\rSigma_{k+1}^{S[t]}(\{\pvec_{k+1}^{S[t]},\kappa\})$ fashion
(recall that the top generator $\gamma$ of $F^R$ is in $t^R$, hence either $=\delta$ or is in $p_{k+1}^{S[t]}$).
But then  $F_\downarrow^R$ is the unique extender in $R[t]$ satisfying the right $\rPi_{k+1}^{S[t]}(\{\pvec_{k+1}^{S[t]},\kappa\})$ statement
(again, we have access to the parameter $\gamma$ here), so it can be identified appropriately.

Part \ref{item:tp_2_rSigma_2^R[t]_and_rSigma_k+2^S[t]}: To formulate $\psi_{k+2}$: Using part of the proof of part \ref{item:tp_2_rSigma_1^R[t]_and_rSigma_k+1^S[t]_in_nice_params}, and using part \ref{item:tp_2_F_downarrow_def_at_rSigma_k+2^S[t]}, we can identify $p_1^{R[t]}$ and $F_\downarrow^R$ appropriately. But then using the methods of the proof of Proposition \ref{prop:tp_2_overlap_no_delta-meas_R[t]_S[t]}, we can translate between the relevant $\bfrSigma_1^{R[t]}$ statements
and $\bfrSigma_{k+1}^{S[t]}$ statements,
in order to convert between theories witnessing the relevant
$\bfrSigma_2^{R[t]}$ statements
and the relevant $\bfrSigma_{k+2}^{S[t]}$ statements. The formulation of $\tau_2$ is similar.
\end{proof}

\begin{prop}\label{prop:tp_2_overlap_no_delta-meas_R[t]_S[t]_n} Let $0<n<\om$. Suppose that $R[t]$
is $n$-sound and $>\delta$-Dodd-sound and $\delta<\rho_n^{R[t]}$.
Then:
\begin{enumerate}
\item $\rho_i^{R[t]}=\rho_{k+i}^{S[t]}$ and $S[t]$ is $(k+i)$-sound,
for all $i\in[1,n]$.
\item $p_i^{R[t]}=p_{k+i}^{S[t]}$ for all $i\in[2,n]$.\footnote{We needn't have $p_1^{R[t]}=p_{k+1}^{S[t]}$,
but see Proposition \ref{prop:tp_2_1-sound_R[t]_translate_rSigma_1_about_p_1} in this regard.}
\item\label{item:truth_conversion_one-step_star_active_overlapping_n}
$\rSigma_{k+n+1}^{S[t]}(\{\pvec_{k+1}^{S[t]},\kappa,\alpha_\Tt\})$
statements about parameters in $R[t]$
are recursively equivalent to $\rSigma_{n+1}^{R[t]}(\{p_1^{R[t]},\alpha_\Tt\})$ statements about those parameters.
That is, there is an $\rSigma_{k+n+1}$ formula $\psi_{k+n+1}$ such that
for all $\rSigma_{n+1}$ formulas $\varphi$
and all $v\in R[t]$, we have
\[ R[t]\sats\varphi(v,p_1^{R[t]},\alpha_\Tt)\iff S[t]\sats\psi_{k+n+1}(\varphi,v,\pvec_{k+1}^{S[t]},\kappa,\alpha_\Tt) \]
and there is an $\rSigma_{n+1}$ formula $\tau_{n+1}$ such that for all $\rSigma_{k+n+1}$ formulas $\varphi$ and all $v\in R[t]$, we have
\[ S[t]\sats\varphi(v,\pvec_{k+1}^{S[t]},\kappa,\alpha_\Tt)\iff R[t]\sats\tau_{n+1}(\varphi,v,p_1^{R[t]},\alpha_\Tt).\]
\tu{(}Note here that the standard parameters to which we refer remain just $p_1^{R[t]}$ and $\pvec_{k+1}^{S[t]}$,
independent of $n$.\tu{)}
\item
For  $X\sub R[t]$, we have
\[ R[t]\cap \Hull^{S[t]}_{\rSigma_{k+n+1}}(X\cup\{\pvec_{k+1}^{S[t]},\kappa,\alpha_\Tt\})=\Hull_{\rSigma_{n+1}}^{R[t]}(X\cup\{p_1^{R[t]},\alpha_\Tt\}).\]
\item For $\mu<\rho_n^{R[t]}=\rho_{n+k}^{S[t]}$
and $f:[\mu]^{<\om}\to R[t]$, we have
\[ f\text{ is }\bfrSigma_{k+n}^{S[t]} \iff f\text{ is }\bfrSigma_n^{R[t]}.\qedhere\]
\end{enumerate}
\end{prop}
\begin{proof}
Using Proposition \ref{prop:tp_2_1-sound_R[t]_translate_rSigma_1_about_p_1} at the base, this is a straightforward induction on $n$.
\end{proof}
\subsubsection{When $R$ is type 3,
$\crit(F^R)<\delta$ and there is no $R$-$\delta$-measure}\label{sec:type_3_translate_step}

Fix an active $\star$-suitable  $R$ of type 3,
with $\kappa=\crit(F^R)<\delta$.
(Note that in this case we first convert from $R$ to $R_J$,
using \S\ref{sec:between_R_and_R_J}.
And we translated between $R_J$ and $R_J[t]$ in \S\ref{sec:between_R_and_R[t]}.) In this section we translate between $R_J[t]$ and $S[t]$ where $S$ is defined like at the start of \S\ref{sec:type_2_translate_step}, but now using $R_J$. Let $\kappa=\crit(F^R)$, $k$, $P$, $j:P\to S$, etc, also be like there.

\begin{prop}\label{prop:tp_3_overlap_no_delta-meas_R_J[t]_S[t]} There is an $\rSigma_{k+1}$  formula $\psi$
and an $\rSigma_1$ formula $\tau$ such that:
\begin{enumerate}
\item\label{item:tp_3_R_J[t]_to_S[t]_R_J[t]_sub_S[t]} $R_J[t]\sub S[t]$
and $\univ{R_J}$, $\es^{R_J}$
are $\Sigma_1^{S[t]}(\{\lambda(F^{R_J})\})$,
and $\OR^{R_J}=\lambda(F^{R_J})^{+S[t]}$ \tu{(}possibly $\lambda(F^{R_J})$ is the largest cardinal of $S[t]$,
in which case $R_J[t]$ and $S[t]$ have the same universe
and $P$, $S$ are active type 2\tu{)}.
\item\label{item:tp_3_truth_conversion_one-step_star_active_overlapping} Set
$z = (\pvec_k^{S[t]},p_{k+1}^{S[t]}\cut j(\kappa),\kappa)$.
Then there is an $\rSigma_{k+1}$ formula $\psi$ and an $\rSigma_1$ formula $\tau$ such that
$\rSigma_{k+1}^{S[t]}(\{z,\alpha_\Tt\})$
statements about parameters in $R_J[t]$
are recursively equvalent to $\rSigma_1^{R_J[t]}(\{\alpha_\Tt\})$ statements about those parameters, as witnessed by $\psi,\tau$. In fact,
for all $\rSigma_1$ formulas $\varphi$
and $v\in R_J[t]$, we have
\[ R_J[t]\sats\varphi(v,\alpha_\Tt)\iff S[t]\sats\psi(\varphi,v,z,\alpha_\Tt) \]
and for all $\rSigma_{k+1}$ formulas $\varphi$ and $v\in R_J[t]$, we have
\[ S[t]\sats\varphi(v,z,\alpha_\Tt)\iff R_J[t]\sats\tau(\varphi,v,\alpha_\Tt),\]
The choice of $\psi,\tau$ are moreover uniform in $R$.

\item\label{item:type_3_R_J_and_S^R_correspond} With $z$ as in part \ref{item:truth_conversion_one-step_star_active_overlapping} and
 $X\sub R$, we have
\[ R[t]\cap \Hull^{S[t]}_{\rSigma_{k+1}}(X\cup\{z,\alpha_\Tt\})=\Hull_{\rSigma_1}^{R[t]}(X\cup\{\alpha_\Tt\}).\]
\item\label{item:type_3_R_J_and_S_functions_into_R_J[t]_and_S[t]} Given $\mu<\OR^{R}$
and $f:[\mu]^{<\om}\to R[t]$,
$f$ is $\rSigma_k^{S[t]}$ iff $f\in R[t]$.
\end{enumerate}
\end{prop}
\begin{proof}
Just simplify  the proof of Proposition \ref{prop:tp_2_overlap_no_delta-meas_R[t]_S[t]}. (One minor new detail regarding part \ref{item:type_3_R_J_and_S_functions_into_R_J[t]_and_S[t]} is that it can be that $\OR^{R[t]}=\rho_k^{S[t]}<\OR^{S[t]}$.
But the same conclusion still holds,
since $\OR^{R[t]}$ is a regular cardinal of $S[t]$.)
\end{proof}

The following proposition is proved like \ref{prop:tp_2_no_delta-meas_Dodd-soundness}; recall that
by \S\ref{sec:between_R_and_R_J}, $R_J$ is $1$-sound with $\rho_1^{R_J}=\nu(F^{R_J})$
and $p_1^{R_J}=\emptyset$.

\begin{prop}\label{prop:tp_3_overlap_no_delta-meas_R_J[t]_S[t]_getting_extra_soundness}
Suppose $R$ is $\star$-suitable, type 3, $\crit(F^R)<\delta$, and $R_J$ has no $\delta$-measure.
Then:
\begin{enumerate}
\item\label{item:R_tp_3_no_delta-meas_S[t]_sound} $S[t]$ is $(k+1)$-sound,
\item\label{item:R_tp_3_no_delta-meas_rho_k+1^S[t]} $\rho_{k+1}^{S[t]}=\rho_1^{R_J[t]}$
and
\item\label{item:R_tp_3_no_delta-meas_p_k+1^S[t]} $p_{k+1}^{S[t]}=j(p_{k+1}^{P}\cut\kappa)$.
\end{enumerate}
Suppose further that
$P=M_\kappa$ is active type 2 and $\crit(F^P)<\kappa=\crit(F^R)$. Then:
\begin{enumerate}[resume*]
 \item\label{item:R_tp_3_no_delta-meas_P_tp_2_S_tp_2} $S$ is active type 2, $\crit(F^S)<\delta$ and $S[t]$ is $>\delta$-Dodd-sound,
and
\item\label{item:R_tp_3_no_delta-meas_cases_with_rho}
either:
\begin{enumerate}[resume*]
\item $\kappa<\rho_1^P$
and $j(\kappa)<\rho_1^S=\rho_1^{S[t]}$
and $S$ is Dodd-sound, or
\item $\rho_1^P\leq\kappa$
and $\rho_1^{S[t]}=\max(\rho_1^{R_J},\delta)$
and $F^S$ is generated by \[ \{\delta\}\cup \rho_1^{S[t]}\cup j(t^P) \]
and the Dodd-solidity witnesses for $j(t^P)$ are in $S[t]$.
\end{enumerate}
\end{enumerate}
\end{prop}

\begin{prop}\label{prop:tp_3_overlap_no_delta-meas_R[t]_S[t]_n} Suppose $R$ is $\star$-suitable, type 3, $\crit(F^R)<\delta$, and $R_J$ has no $\delta$-measure.
 Let $0<n<\om$. Suppose that $R_J[t]$
is $n$-sound and $\delta<\rho_n^{R_J[t]}$.
Then:
\begin{enumerate}
\item $\rho_i^{R_J[t]}=\rho_{k+i}^{S[t]}$ and $S[t]$ is $(k+i)$-sound,
for all $i\in[1,n]$.
\item $p_i^{R_J[t]}=p_{k+i}^{S[t]}$ for all $i\in[2,n]$ and $p_1^{R_J[t]}=\emptyset$.\footnote{We needn't have $p_1^{R_J[t]}=p_{k+1}^{S[t]}$,
but see Proposition
\ref{prop:tp_3_overlap_no_delta-meas_R_J[t]_S[t]_getting_extra_soundness}
in this regard.}
\item
$\rSigma_{k+n+1}^{S[t]}(\{\pvec_{k+1}^{S[t]},\kappa,\alpha_\Tt\})$
statements about parameters in $R_J[t]$
are recursively equivalent to $\rSigma_{n+1}^{R_J[t]}(\{\alpha_\Tt\})$ statements about those parameters.
That is, there is an $\rSigma_{k+n+1}$ formula $\psi_{k+n+1}$ such that
for all $\rSigma_{n+1}$ formulas $\varphi$
and all $v\in R[t]$, we have
\[ R_J[t]\sats\varphi(v,\alpha_\Tt)\iff S[t]\sats\psi_{k+n+1}(\varphi,v,\pvec_{k+1}^{S[t]},\kappa,\alpha_\Tt) \]
and there is an $\rSigma_{n+1}$ formula $\tau_{n+1}$ such that for all $\rSigma_{k+n+1}$ formulas $\varphi$ and all $v\in R_J[t]$, we have
\[ S[t]\sats\varphi(v,\pvec_{k+1}^{S[t]},\kappa,\alpha_\Tt)\iff R_J[t]\sats\tau_{n+1}(\varphi,v,\alpha_\Tt).\]
\tu{(}Note here that $p_1^{R_J[t]}=\emptyset$ and the standard parameters to which we refer remain just $\pvec_{k+1}^{S[t]}$,
independent of $n$.\tu{)}
\item\label{item:type_2_R_and_S^R_correspond_n}
For  $X\sub R_J[t]$, we have
\[ R_J[t]\cap \Hull^{S[t]}_{\rSigma_{k+n+1}}(X\cup\{\pvec_{k+1}^{S[t]},\kappa,\alpha_\Tt\})=\Hull_{\rSigma_{n+1}}^{R_J[t]}(X\cup\{\alpha_\Tt\}).\]
\item\label{item:type_2_R_and_S^R_functions_into_R[t]_and_S[t]_n} For $\mu<\rho_n^{R_J[t]}=\rho_{n+k}^{S[t]}$
and $f:[\mu]^{<\om}\to R_J[t]$, we have
\[ f\text{ is }\bfrSigma_{k+n}^{S[t]} \iff f\text{ is }\bfrSigma_n^{R_J[t]}.\qedhere\]
\end{enumerate}
\end{prop}
\begin{proof}
By induction on $n$, using Propositions \ref{prop:tp_3_overlap_no_delta-meas_R_J[t]_S[t]}
and \ref{prop:tp_3_overlap_no_delta-meas_R_J[t]_S[t]_getting_extra_soundness} at the base.
\end{proof}

\subsection{Between $R[t]$ and $R^\star$ when $R$ is admissibly buffered}\label{sec:R[t]_and_R^*_admissibly_buffered}

We now compute the final step of the translation of fine structure
between $R$ and $R^\star$ -- the step between $R_m[t]=(R_m)'[t]$ and $(R_m)^\star$,
or resetting notation, the step from $R[t]$ to $R^\star$ when $R$ is admissibly buffered. For this step, we will rely on some of the standard first order consequences of the iterability of $R^\star$, in the case that $R[t]$ projects to $\delta$. This has the implication that
our proof that $R^\star$ is a premouse
will rely on the iterability of its proper segments,
and also the full translation of fine structure between $R$ and $R^\star$
will rely on the iterability of $R^\star$.

\begin{rem}\label{rem:R[t]_not_rSigma_1_in_R^star}
Before we begin, we will indicate why  some care is required in the calculations to come.
It turns out that $R$ can fail to be $\rSigma_1^{R^\star}$, and even fail to be $\rSigma_1^{R^\star}(\{x\})$
for ``small'' parameters $x$, and even subsets of $\delta$ can easily be ``small'' here. Moreover,
\cite[Theorem 1.2.9(a)]{closson}
is not correct in general in the case that $j=1$, and (in the notation there) the $*$-type of $\mathcal{P}$ is I or II. For suppose that  $\gamma>\delta$ is an $R$-successor-cardinal and there is $E\in\es^R$ with \[ \crit(E)<\delta<\gamma<\gamma^{+R}<\gamma^{++R}<\lh(E)<\OR^R.\]
Now $R$ is $\rSigma_1^R$,
but we claim that $R$ is not $\rSigma_1^{R^\star}(\{x\})$
for any $x\in R^\star|\gamma$. For let $x\in R^\star|\gamma$
and $\varphi$ be $\rSigma_1$ and suppose that for each $N\in R^\star$,
we have
\[ N\pins R\iff R^\star\sats\varphi(x,N).\]
Then $R^\star\sats\varphi(x,R|\gamma^{+R})$.
Therefore by condensation,
$R^\star|\gamma^{++R}\sats\varphi(x,R|\gamma^{+R})$.
Let $U=\Ult(R,E)$.
We have $\gamma^{++R^\star}=\gamma^{++R}$
and \[ R^\star|\gamma^{++R^\star}=(R|\gamma^{++R})^\star=(U|\gamma^{++U})^\star\sats\varphi(x,R|\gamma^{+R}),\]
and since $R|\gamma^{++R}=U|\gamma^{++U}$, therefore
\[ (U|\gamma^{++U})^\star\sats\varphi(x,U|\gamma^{+U}).\]
Now
let $F$ be the trivial completion of $E\rest\gamma$,
so $F\in\es^R$. Let $\xi=\lh(F)$. Let $U_F=\Ult(R,F)$.
Let $\pi:U\to U_F$
be the factor map,
so $\crit(\pi)=\xi=\gamma^{+U_F}$,
 $\pi(\gamma^{+U_F})=\gamma^{+U}$
 and $\pi(\gamma^{++U_F})=\gamma^{++U}$,
 so by elementarity,
 \[ (U_F|\gamma^{++U_F})^\star\sats\varphi(x,U_F|\xi).\]
But $(U_F|\gamma^{++U_F})^\star\ins(R|\xi)^\star\ins R^\star$
(using here that  $\gamma$ is an $R$-successor-cardinal, so $i_F(\kappa)>\gamma$
where $\kappa=\crit(F)$,
so $U_F$ agrees sufficiently with $\Ult_{n_\kappa}(P_\kappa,F)$).
And since $\varphi$ is $\rSigma_1$,
therefore $R^\star\sats\varphi(x,U_F|\xi)$, so we have $U_F|\xi\ins R$.
But this is false, since $R|\xi$ is active with $F$, whereas $\xi=\gamma^{+U_F}$.
\end{rem}

\begin{prop}\label{prop:adm_buff_basics} Assume $R$ is $\star$-iterable and admissibly buffered and $R^\star$ is a well-defined premouse. Then:
\begin{enumerate}
\item\label{item:adm_buff_R[t]_R^star_same_universe} $R[t]$ and $R^\star$ have the same universe.
\item\label{item:adm_buff_R^star_Delta_1_in_R[t]} $R^\star$ is $\Delta_1^{R[t]}(\{\alpha_\Tt\})$
and $\{P\}$ is $\Delta_1^{R[t]}(\emptyset)$.
\item\label{item:adm_buff_rSigma_1^R^star_is_rSigma_1^R[t]} $\rSigma_1^{R^\star}(\{\alpha_\Tt,P\})$ recursively reduces to $\rSigma_1^{R[t]}(\{\alpha_\Tt\})$.\footnote{Note that this reduction is only claimed in the one direction.}
That is, there is an $\rSigma_1$ formula $\tau_1$ such that for all $\rSigma_1$ formulas $\varphi$ and all $v\in R[t]$ \tu{(}equivalently, $v\in R^\star$\tu{)}, we have
\[ R^\star\sats\varphi(v,\alpha_\Tt,P)\iff R[t]\sats\tau_1(\varphi,v,\alpha_\Tt).\]
\item\label{item:adm_buff_R[t]_Delta_2_in_R^star}
$R[t]$ is $\rDelta_2^{R^\star}(\{\alpha_\Tt,P\})$.
\end{enumerate}
\end{prop}
\begin{proof}
Part \ref{item:adm_buff_R[t]_R^star_same_universe}: We have $R^\star\sub R[t]$ because working in $R[t]$,
we can compute the star-translation, remaining inside $R[t]$, since $R$ is admissibly buffered,
and by \cite[Theorem 10]{conjecture_mouse_order_for_weasels}, $S[t]$ is admissible for every admissible $S$ with $R|\delta\pins S\ins R$. Conversely, we have $R[t]\sub R^\star$ because $t\in R^\star$ and since $R$ can be recovered working in $R^\star$, by computing the $\Moon$-construction. Since $R$ is admissibly buffered,
$R^\star$ is sufficiently closed that it computes the entire $\Moon$-construction.

In the case that $R$ is passive, part \ref{item:adm_buff_R^star_Delta_1_in_R[t]} follows from a little examination of the preceding paragraph,
and since $P$ is determined by the theory $t$,
and $\Tt$ determined by $(P,\alpha_\Tt)$.
In the case that $R$ is active type 2 (hence with $\delta<\crit(F^R)$, these observations
suffice to define $(R^\star)^\passive$.
But (the amenable code of) $F^{R^\star}$
just consists of those pairs $(N,G)$
such that  $N\pins R^\star$,
$G$ measures the appropriate collection of subsets of its critical point, has support $\nu(F^R)=\nu(F^{R^\star})$, and is induced by some element of the amenable code of $F^{R[t]}$.
Since $(R^\star)^\passive$ is $\rSigma_1^{R[t]}(\{\alpha_\Tt\})$, so is this amenable code. In the case that $R$ is active type 3, it is rather similar,
but we work literally with $R^\sq$, $(R[t])^\sq$ and $(R^\star)^\sq$, hence working with the corresponding notion of amenable code.

Part \ref{item:adm_buff_rSigma_1^R^star_is_rSigma_1^R[t]}: This is a direct consequence of parts \ref{item:adm_buff_R[t]_R^star_same_universe} and \ref{item:adm_buff_R^star_Delta_1_in_R[t]}.

Part \ref{item:adm_buff_R[t]_Delta_2_in_R^star} is likewise, noting that $\rDelta_2$ suffices (whereas by Remark \ref{rem:R[t]_not_rSigma_1_in_R^star}, $\rDelta_1$ does not in general suffice), because given a premouse $N\in R^\star$, we have $N\pins R=(R^\star)^\moon$ iff for some $S\pins R^\star$, we have $N\ins (S')^\moon$ for all $S'$ such that $S\ins S'\pins R^\star$. (In the particular case that $R=\J(N)$ then $R^\star=\J(N^\star)$, and $\{N\}$ is $\rDelta_2^{R^\star}(\{\alpha_\Tt,P\})$,
and hence so is $R$.)
\end{proof}
\begin{prop}\label{prop:adm_buff_Sigma_1_corresp_passive_largest_card} Suppose $R$ is $\star$-iterable, admissibly buffered,
passive, and has a largest cardinal $\gamma$. Suppose $R^\star$ is a well-defined premouse. Then:
\begin{enumerate}
\item\label{item:adm_buff_R[t]_Delta_1_in_R^star} $R[t]$ is $\rDelta_1^{R^\star}(\{P,\alpha_\Tt,\gamma\})$.\footnote{Note that we have the parameter $\gamma$ available here. Cf.~Remark \ref{rem:R[t]_not_rSigma_1_in_R^star}.}
\item\label{item:adm_buff_rSigma_1_equiv}  $\rSigma_1^{R[t]}(\{\alpha_\Tt,\gamma\})$ statements
about parameters
are recursively equivalent
to $\rSigma_1^{R^\star}(\{P,\alpha_\Tt,\gamma\})$ statements about those parameters. That is,
there is an $\rSigma_1$ formula $\psi_1$ such that for all $\rSigma_1$ formulas $\varphi$ and all $v\in R[t]$ \tu{(}equivalently, $v\in R^\star$\tu{)}, we have
\[ R[t]\sats\varphi(v,\alpha_\Tt,\gamma)\iff R^\star\sats\psi_1(\varphi,v,P,\alpha_\Tt,\gamma),\]
and there is an $\rSigma_1$ formula $\tau_1$ such that for all $\rSigma_1$ formulas $\varphi$ and all $v\in R[t]$, we have
\[ R^\star\sats\varphi(v,P,\alpha_\Tt,\gamma)\iff R[t]\sats\tau_1(\varphi,v,\alpha_\Tt,\gamma).\]
\item\label{item:adm_buff_Hull_1_corresp}
For any $X\sub R[t]$,
$\Hull_1^{R[t]}(X\cup\{\alpha_\Tt,\gamma\})$ and $\Hull_1^{R^\star}(X\cup\{\alpha_\Tt,P,\gamma\})$ have identical universes.
\end{enumerate}
Moreover, the definitions and formulas witnessing parts   \ref{item:adm_buff_R[t]_Delta_1_in_R^star} and \ref{item:adm_buff_rSigma_1_equiv} are uniform in $R,P,\alpha_\Tt$.
\end{prop}
\begin{proof}
Part \ref{item:adm_buff_R[t]_Delta_1_in_R^star}:
We have $R|\gamma=(R^\star|\gamma)^{\moon}\ins R=(R^\star)^\moon$,
since $\gamma$ is an $R^\star$-cardinal. So since we have the parameter $\gamma$ available,
we can identify $R|\gamma$.
Beyond $R|\gamma$, we have
$R|\gamma\ins N\pins (R^\star)^\moon$
iff there is some $N'\in R^\star$
and $n<\om$
such that $N\ins N'$,
$\rho_{n+1}^{N'}=\gamma$
and $N'\ins S^{\moon}$
for some $S\pins R^\star$. This yields an $\rSigma_1^{R^\star}(\{P,\alpha_\Tt,\gamma\})$ definition of $R$.

Parts \ref{item:adm_buff_rSigma_1_equiv}  and \ref{item:adm_buff_Hull_1_corresp}
are immediate consequences.
\end{proof}

\begin{prop}\label{item:adm_buff_Sigma_1_corresp_passive_no_largest_card} Suppose $R$ is passive, $\star$-iterable and admissibly buffered, with no largest cardinal,
and $R\sats$ ``all my proper segments satisfy condensation''. Suppose $R^\star$ is a well-defined premouse. Then:
\begin{enumerate}
\item\label{item:adm_buff_R[t]_no_lgst_card_rho_1=OR} $\rho_1^R=\rho_1^{R[t]}=\rho_1^{R^\star}=\OR^R$,
and for each $R$-cardinal $\gamma>\delta$, we have
\[ R|\gamma\preccurlyeq_1 R\text{ and }R[t]|\gamma\preccurlyeq_1 R[t]\text{ and }R^\star|\gamma\preccurlyeq_1 R^\star.\]
\item\label{item:adm_buff_rSigma_2_equiv_no_lgst_card}  $\rSigma_2^{R[t]}(\{\alpha_\Tt\})$ statements
about parameters
are recursively equivalent
to $\rSigma_2^{R^\star}(\{P,\alpha_\Tt\})$ statements about those parameters. That is,
there is an $\rSigma_2$ formula $\psi_2$ such that for all $\rSigma_2$ formulas $\varphi$ and all $v\in R[t]$, we have
\[ R[t]\sats\varphi(v,\alpha_\Tt)\iff R^\star\sats\psi_2(\varphi,v,P,\alpha_\Tt),\]
and there is an $\rSigma_2$ formula $\tau_2$ such that for all $\rSigma_2$ formulas $\varphi$ and all $v\in R[t]$, we have
\[ R^\star\sats\varphi(v,P,\alpha_\Tt)\iff R[t]\sats\tau_2(\varphi,v,\alpha_\Tt).\]
\item\label{item:adm_buff_Hull_2_corresp_no_lgst_card}
For any $X\sub R[t]$,
$\Hull_2^{R[t]}(X\cup\{\alpha_\Tt\})$ and $\Hull_2^{R^\star}(X\cup\{\alpha_\Tt,P\})$ have identical universes.
\end{enumerate}
\end{prop}
\begin{proof}
Part \ref{item:adm_buff_R[t]_no_lgst_card_rho_1=OR} is by condensation and since for each $R$-cardinal $\gamma>\delta$, we have $(R|\gamma)[t]=(R[t])|\gamma$ and $(R|\gamma)^\star=R^\star|\gamma$.
Part \ref{item:adm_buff_rSigma_2_equiv_no_lgst_card} is a consequence of this and Proposition \ref{prop:adm_buff_basics}. In particular, for the formula $\psi_2$,
working in $R^\star$,
one can  identify theories of the form $\Th_{\rSigma_1}^{R[t]}(\gamma)$,
for cardinals $\gamma>\delta$, in an $\rSigma_2(\{P,\alpha_\Tt\})$ fashion, by using part \ref{item:adm_buff_R[t]_no_lgst_card_rho_1=OR}. Part \ref{item:adm_buff_Hull_2_corresp_no_lgst_card} is an immediate corollary
of part \ref{item:adm_buff_rSigma_2_equiv_no_lgst_card}.
\end{proof}
\begin{prop}\label{item:adm_buff_Sigma_1_corresp_active}
Assume $R$ is $\star$-iterable, admissibly buffered, and active, hence with $\delta<\crit(F^R)$. Suppose $R^\star$ is a well-defined premouse. Then:
\begin{enumerate}
\item\label{item:adm_buff_active_R[t]_rSigma_1^R^star} $R[t]$ is $\rDelta_1^{R^\star}(\{P,\alpha_\Tt\})$.
\item\label{item:adm_buff_rSigma_1_equiv_active}  $\rSigma_1^{R[t]}(\{\alpha_\Tt\})$ statements
about parameters
are recursively equivalent
to $\rSigma_1^{R^\star}(\{P,\alpha_\Tt\})$ statements about those parameters. That is,
there is an $\rSigma_1$ formula $\psi_1$ such that for all $\rSigma_1$ formulas $\varphi$ and all $v\in R[t]$, we have
\[ R[t]\sats\varphi(v,\alpha_\Tt)\iff R^\star\sats\psi_1(\varphi,v,P,\alpha_\Tt),\]
and there is an $\rSigma_1$ formula $\tau_1$ such that for all $\rSigma_1$ formulas $\varphi$ and all $v\in R[t]$, we have
\[ R^\star\sats\varphi(v,P,\alpha_\Tt)\iff R[t]\sats\tau_1(\varphi,v,\alpha_\Tt).\]
\item\label{item:adm_buff_Hull_1_corresp_active}
For any $X\sub R[t]$,
$\Hull_1^{R[t]}(X\cup\{\alpha_\Tt\})$ and $\Hull_1^{R^\star}(X\cup\{\alpha_\Tt,P\})$ have identical universes.
\end{enumerate}
\end{prop}
\begin{proof}
Part \ref{item:adm_buff_active_R[t]_rSigma_1^R^star}:
Because $R^\star$ is active, the class of $R^\star$-cardinals is $\rDelta_1^{R^\star}$, and so the  proof of Proposition \ref{prop:adm_buff_Sigma_1_corresp_passive_largest_card}
adapts to give that $(R[t])^\passive$ is $\rDelta_1^{R^\star}(\{P,\alpha_\Tt\})$ in this case.
But then (the amenable code of) $F^{R[t]}$ is also $\rSigma_1^{R^\star}(\{P,\alpha_\Tt\})$,
just like in the converse direction in the proof of Proposition \ref{prop:adm_buff_basics} part \ref{item:adm_buff_R^star_Delta_1_in_R[t]}.

As usual, parts \ref{item:adm_buff_rSigma_1_equiv_active} and \ref{item:adm_buff_Hull_1_corresp_active} are  consequences of part \ref{item:adm_buff_active_R[t]_rSigma_1^R^star} and Proposition \ref{prop:adm_buff_basics}.
\end{proof}
\begin{prop}\label{prop:adm_buff_fs_corresp}
 Let $n<\om$. Suppose that $R[t]$ is  $n$-sound,
 $\delta<\rho_n^{R[t]}$ and $R^\star$ is a well-defined premouse.
 Suppose that $R\sats$ ``All of my proper segments satisfy condensation''. Then:
\begin{enumerate}
\item\label{item:adm_buff_fs_corresp_up_to_n} $R^\star$ is $n$-sound, $\rho_i^{R[t]}=\rho_i^{R^\star}$ and
 $p_i^{R[t]}=p_i^{R^\star}$ for all $i\leq n$.
\item\label{item:adm_buff_rSigma_n+1_equiv} If $R$ is passive with a largest cardinal, assume  $0<n$.\footnote{See Proposition \ref{prop:adm_buff_Sigma_1_corresp_passive_largest_card} for the variant when $R$ is passive with a largest cardinal and $n=0$.} Then \tu{(}in general\tu{)},  $\rSigma_{n+1}^{R[t]}(\{\alpha_\Tt\})$ statements
about parameters
are recursively equivalent
to $\rSigma_{n+1}^{R^\star}(\{P,\alpha_\Tt\})$ statements about those parameters. That is,
there is an $\rSigma_{n+1}$ formula $\psi_{n+1}$ such that for all $\rSigma_{n+1}$ formulas $\varphi$ and all $v\in R[t]$ \tu{(}equivalently, $v\in R^\star$\tu{)}, we have
\[ R[t]\sats\varphi(v,\alpha_\Tt)\iff R^\star\sats\psi_{n+1}(\varphi,v,P,\alpha_\Tt),\]
and there is an $\rSigma_{n+1}$ formula $\tau_1$ such that for all $\rSigma_{n+1}$ formulas $\varphi$ and all $v\in R[t]$, we have
\[ R^\star\sats\varphi(v,P,\alpha_\Tt)\iff R[t]\sats\tau_{n+1}(\varphi,v,\alpha_\Tt).\] \item\label{item:adm_buff_Hull_n+1_corresp}
If $R$ is passive with a largest cardinal, assume  $0<n$.
\footnote{See Proposition \ref{prop:adm_buff_Sigma_1_corresp_passive_largest_card} for the variant when $R$ is passive with a largest cardinal and $n=0$.}
Then \tu{(}in general\tu{)},
for any $X\sub R[t]$,
$\Hull_{n+1}^{R[t]}(X\cup\{\alpha_\Tt\})$ and $\Hull_{n+1}^{R^\star}(X\cup\{\alpha_\Tt,P\})$ have identical universes.
\item\label{item:adm_buff_fs_corresp_rho_n+1>delta} Suppose $\rho_{n+1}^{R[t]}>\delta$. Then:
\begin{enumerate}[label=\tu{(}\alph*\tu{)}]
\item $\rho_{n+1}^{R[t]}=\rho_{n+1}^{R^\star}$,
\item $p_{n+1}^{R[t]}=p_{n+1}^{R^\star}$,
\item\label{item:adm_buff_fs_corresp_rho_n+1>delta_solidity}
$R[t]$ is $(n+1)$-solid iff $R^\star$ is $(n+1)$-solid, and
\item\label{item:adm_buff_fs_corresp_rho_n+1>delta_soundness} $R[t]$ is $(n+1)$-sound iff $R^\star$ is $(n+1)$-sound.
\end{enumerate}
\item\label{item:adm_buff_fs_corresp_rho_n+1=delta_P-premouse} Suppose $\rho_{n+1}^{R[t]}=\delta$. Let $R^\star_P$ be the $P$-premouse whose extender sequence is just $\es^{R^\star}_+\rest(\OR^P,\OR^{R^\star}]$. Then:
\begin{enumerate}[label=\tu{(}\alph*\tu{)}]
\item $\rho_{n+1}^{R^\star_P}=\OR^P$,
\item $p_{n+1}^{R[t]}=p_{n+1}^{R^\star}$, and $p_{n+1}^{R[t]}\cap(\OR^P+1)=\emptyset$,
\item $R[t]$ is $(n+1)$-solid iff $R^\star_P$ is $(n+1)$-solid, and
\item $R[t]$ is $(n+1)$-sound iff $R^\star_P$ is $(n+1)$-sound.
\end{enumerate}
\item\label{item:adm_buff_fs_corresp_rho_n+1=delta=rho_n+1^R^star} Suppose $\rho_{n+1}^{R[t]}=\delta=\rho_{n+1}^{R^\star}$
and $R^\star$ is $(n+1)$-universal and $(n+1)$-solid.
Then:
\begin{enumerate}[label=\tu{(}\alph*\tu{)}]
\item\label{item:adm_buff_fs_corresp_rho_n+1=delta=rho_n+1^R^star_when_P_in_Hull} If $P\in\Hull_{n+1}^{R^\star}(\delta\cup\{\pvec_n^{R^\star},p_{n+1}^{R[t]}\})$ then $p_{n+1}^{R[t]}=p_{n+1}^{R^\star}$.
\item\label{item:adm_buff_fs_corresp_rho_n+1=delta=rho_n+1^R^star_when_P_not_in_Hull} If $P\notin\Hull_{n+1}^{R^\star}(\delta\cup\{\pvec_n^{R^\star},p_{n+1}^{R[t]}\})$
then there is $\zeta\in[\delta,\OR^P]$ such that $p_{n+1}^{R[t]}\cup\{\zeta\}=p_{n+1}^{R^\star}$.
\item $R[t]$ is $(n+1)$-universal and $(n+1)$-solid.
\item $R[t]$ is $(n+1)$-sound iff $R^\star$ is $(n+1)$-sound.
\item\label{item:adm_buff_rho_n+1^R[t]=delta=rho_n+1^R^star_charac_p_n+1} If $R[t]$ is $(n+1)$-sound
then $p_{n+1}^{R^\star}$
is the least $q\in[\OR^{R[t]}]^{<\om}$ such that $p_{n+1}^{R[t]}\ins q$
and $\OR^P\in\Hull_{n+1}^{R^\star}(\delta\cup\{\pvec_n^{R^\star},q\})$.
\end{enumerate}
\end{enumerate}
\end{prop}

\begin{proof}
Parts \ref{item:adm_buff_fs_corresp_up_to_n}
 is trivial when $n=0$, and by induction
(using part \ref{item:adm_buff_fs_corresp_rho_n+1>delta}) when $n>0$.

Part \ref{item:adm_buff_rSigma_n+1_equiv}: This is by
Propositions \ref{item:adm_buff_Sigma_1_corresp_passive_no_largest_card} and \ref{item:adm_buff_Sigma_1_corresp_active} when $n=0$. It is by induction (with part \ref{item:adm_buff_rSigma_n+1_equiv}) when $n>0$,
except in case $n=1$ and $R$ is passive with a largest cardinal $\gamma$. In the latter case, it follows from Proposition \ref{prop:adm_buff_Sigma_1_corresp_passive_largest_card},
since $\{\gamma\}$ is $\rDelta_2^{R[t]}$ and $\rDelta_2^{R^\star}$.

Part \ref{item:adm_buff_Hull_n+1_corresp} is a corollary of part \ref{item:adm_buff_rSigma_n+1_equiv}.

Parts \ref{item:adm_buff_fs_corresp_rho_n+1>delta} and \ref{item:adm_buff_fs_corresp_rho_n+1=delta_P-premouse}:
We have $p_{n+1}^{R[t]}\cap(\OR^P+1)=\emptyset$ since $P$ is encoded into $t$, and the language for $R[t]$ can refer directly to $t$. So
 parts \ref{item:adm_buff_fs_corresp_rho_n+1>delta} and \ref{item:adm_buff_fs_corresp_rho_n+1=delta_P-premouse}
 are easy consequences of parts
\ref{item:adm_buff_rSigma_n+1_equiv} and \ref{item:adm_buff_Hull_n+1_corresp}, except in the case that $n=0$ and $R$ is passive with a largest cardinal $\gamma$. So consider the latter case of part \ref{item:adm_buff_fs_corresp_rho_n+1>delta}.
By Proposition \ref{prop:adm_buff_Sigma_1_corresp_passive_largest_card},
 $\rho\eqdef\rho_1^{R}=\rho_1^{R[t]}=\rho_1^{R^\star}>\delta$. If $\rho_1^R=\OR^R$,
there is nothing to prove, so suppose $\rho_1^R<\OR^R$. Therefore $\delta<\rho_1^R\leq\gamma$.
So $p_1^R=p_1^{R[t]}$. Now we claim that $p_1^{R^\star}=p_1^{R[t]}$.
For $p_1^{R[t]}=p_1^R\not\sub\gamma$, since by condensation, $R|\gamma\preccurlyeq_1 R$. So $p_1^R=(p_1^R\cap\gamma)\cup\{\xi\}$ where $\xi=\max(p_1^R)$. But $R|\xi\preccurlyeq_1 R$, since otherwise $\xi\in\Hull_1^R(\xi)$, contradicting the minimality of $p_1^R$. In particular, either $\xi=\gamma$ or $\gamma$ is the largest cardinal of $R|\xi$. In either case, $\{\gamma\}$ is $\rSigma_1^{R^\star}(\{\xi\})$ (since if $\gamma<\xi$ then $\gamma$ is just the unique ordinal $\gamma'<\xi$ such that for some $m<\om$, $\gamma'=\rho_{m+1}^N$ for some $N\pins R^\star$).
So by Proposition \ref{prop:adm_buff_Sigma_1_corresp_passive_largest_card}, the theory $t=\Th_{1}^{R^\star}(\rho\cup\{p_1^{R[t]}\})$ encodes  the theory $\Th_1^{R[t]}(\rho\cup\{p_1^{R[t]}\})$, so $t\notin R[t]$, so $t\notin R^\star$. This shows that $p_1^{R^\star}\leq p_1^{R[t]}$. But conversely, if $p_1^{R^\star}<p_1^{R[t]}$ then just by the minimality of $p_1^{R[t]}$, we have $t'=\Th_1^{R[t]}(\rho\cup\{p_1^{R^\star}\})\in R[t]$, hence $t'\in R^\star$, but by Proposition \ref{prop:adm_buff_basics} part \ref{item:adm_buff_rSigma_1^R^star_is_rSigma_1^R[t]} (here we don't need even need to consider the parameter $\gamma$), $t''=\Th_1^{R^\star}(\rho\cup\{p_1^{R^\star}\})$ can be computed from $t'$, so $t''\in R^\star$, contradiction. So $p_1^{R^\star}=p_1^{R[t]}$.
Regarding clause \ref{item:adm_buff_fs_corresp_rho_n+1>delta_solidity}
(solidity), we automatically have the $1$-solidity witnesses for $\xi=\max(p_1^{R[t]})=\max(p_1^{R^\star})$, for both models, by the discussion above.
But the remaining $1$-solidity witnesses have access to the parameter $\gamma$, so the equivalence follows
from Proposition \ref{prop:adm_buff_Sigma_1_corresp_passive_largest_card}. Clause \ref{item:adm_buff_fs_corresp_rho_n+1>delta_soundness} (soundness) follows
from clause \ref{item:adm_buff_fs_corresp_rho_n+1>delta_solidity} and Proposition \ref{prop:adm_buff_Sigma_1_corresp_passive_largest_card}. This completes the proof of part \ref{item:adm_buff_fs_corresp_rho_n+1>delta} in the present case. Part \ref{item:adm_buff_fs_corresp_rho_n+1=delta_P-premouse} is very similar. If $\gamma>\delta$ it is essentially the same. If $\gamma=\delta$,
both models can easily define the parameter $\delta$ via their language (along with $P,t$). Using that parameter, we can proceed  as before.

Part \ref{item:adm_buff_fs_corresp_rho_n+1=delta=rho_n+1^R^star}: Here we break the proof into cases:

\begin{case}\label{case:not_n=0_and_R_pv_with_lgst_card} It is not the case that $n=0$ and $R$ is passive with a largest cardinal.

\begin{scase} $P\in\Hull_{n+1}^{R^\star}(\delta\cup\{\pvec_n^{R^\star},p_{n+1}^{R[t]}\})$.

By parts \ref{item:adm_buff_rSigma_n+1_equiv} and \ref{item:adm_buff_Hull_n+1_corresp}, $p_{n+1}^{R[t]}=p_{n+1}^{R^\star}$, and also since $R^\star$ is $(n+1)$-universal and by parts \ref{item:adm_buff_rSigma_n+1_equiv} and \ref{item:adm_buff_Hull_n+1_corresp}, $R[t]$ is $(n+1)$-universal.
The $(n+1)$-solidity of $R[t]$ is likewise, noting that since $\OR^P<\alpha$ for all $\alpha\in p_{n+1}^{R[t]}$, the $(n+1)$-solidity witnesses for $R^\star$ can all refer to the parameter $\OR^P$, and hence suffice to compute the corresponding $(n+1)$-solidity witnesses for $R[t]$.
Again using parts \ref{item:adm_buff_rSigma_n+1_equiv} and \ref{item:adm_buff_Hull_n+1_corresp}, it follows that $R[t]$ is $(n+1)$-sound iff $R^\star_P$ is $(n+1)$-sound.
Finally, since $P\in\Hull_{n+1}^{R^\star}(\delta\cup\{\pvec_n^{R^\star},p_{n+1}^{R[t]}\})$
and $p_{n+1}^{R^\star}=p_{n+1}^{R[t]}$, clause \ref{item:adm_buff_rho_n+1^R[t]=delta=rho_n+1^R^star_charac_p_n+1} is immediate in the present case.
\end{scase}

\begin{scase}
 $ P\notin\Hull_{n+1}^{R^\star}(\delta\cup\{\pvec_n^{R^\star},p_{n+1}^{R[t]}\})$.

We claim that $p_{n+1}^{R^\star}\neq p_{n+1}^{R[t]}$. For suppose  $p_{n+1}^{R^\star}=p_{n+1}^{R[t]}$.
Then by the $(n+1)$-universality of $R^\star$, $\delta\notin\Hull_{n+1}^{R^\star}(\delta\cup\{\pvec_n^{R^\star},p_{n+1}^{R^\star}\})$,
but since $\delta$ is a successor cardinal of $R^\star$ (by our original assumptions on $P$),
therefore $n=0$ and $\Hull_1^{R^\star}(\delta\cup\{p_1^{R^\star}\})=R^\star|\delta$, which is in $R^\star$, a contradiction.

It follows (again using
parts \ref{item:adm_buff_rSigma_n+1_equiv} and \ref{item:adm_buff_Hull_n+1_corresp}) that $p_{n+1}^{R^\star}=p_{n+1}^{R[t]}\cup\{\zeta\}$ for some $\zeta\in[\delta,\OR^P]$.
Now the $(n+1)$-universality of $R^\star$
yields that of $R[t]$, and likewise for $(n+1)$-solidity, since $p_{n+1}^{R[t]}\cap(\OR^P+1)=\emptyset$.
Like before, it follows that $R[t]$ is $(n+1)$-sound iff $R^\star$ is $(n+1)$-sound.
Finally, for clause \ref{item:adm_buff_rho_n+1^R[t]=delta=rho_n+1^R^star_charac_p_n+1}, suppose $R[t]$ is $(n+1)$-sound, and hence so is $R^\star$, as we have seen. Therefore $P\in\Hull_{n+1}^{R^\star}(\delta\cup\{\pvec_{n+1}^{R^\star}\})$. Conversely,
suppose there is $\zeta'<\zeta$ such that
 $P\in\Hull_{n+1}^{R^\star}(\delta\cup\{\pvec_n^{R^\star},p_{n+1}^{R[t]},\zeta'\})$. Then note that since $R[t]$ is $(n+1)$-sound, this hull is just $R^\star$, contradicting the minimality of $p_{n+1}^{R^\star}$.
\end{scase}
\end{case}
\begin{case} $n=0$ and $R$ is passive with a largest cardinal $\gamma$.

If $\gamma>\delta$ then by the argument for part \ref{item:adm_buff_fs_corresp_rho_n+1=delta_P-premouse},
$\gamma$ is easily definable from $\max(p_1^{R[t]})=\max(p_1^{R^\star})$,
 and because of this and by Proposition  \ref{prop:adm_buff_Sigma_1_corresp_passive_largest_card}, we can argue as in Case \ref{case:not_n=0_and_R_pv_with_lgst_card}.
Suppose $\gamma=\delta$. Then either $p_1^{R[t]}=\emptyset$ or  $p_1^{R[t]}=\{\xi\}$ for some $\xi\in(\delta,\OR^R)$ (since $R[t]$ can identify $t$ via its language, and hence define the parameter $\delta$).

Suppose $p_1^{R[t]}=\{\xi\}$. then $\xi>\OR^P$,
since $P$ can be identified from $t$, and so $\delta,P\in\Hull_1^{R^\star}(\delta\cup\{p_1^{R[t]}\})$.
But then much as before, $p_1^{R^\star}\leq p_1^{R[t]}$.
But if $p_1^{R^\star}<p_1^{R[t]}$
then $t=\Th_{1}^{R[t]}(\delta\cup\{p_1^{R^\star}\})\in R[t]$, so $t\in R^\star$, and this gives a contradiction like before. So $p_1^{R[t]}=\{\xi\}=p_1^{R^\star}$,
and the rest is now clear.

Now suppose instead that $p_1^{R[t]}=\emptyset$.
If $P\in\Hull_1^{R^\star}(\delta)$ then also $\delta\in\Hull_1^{R^\star}(\delta)$, and then as before,
it follows that $p_1^{R^\star}=\emptyset$,
and the remaining clauses are clear.
If instead $P\notin\Hull_1^{R^\star}(\delta)$,
then this hull is bounded in $\OR^{R^\star}$,
hence an element of $R^\star$, and so $p_1^{R^\star}\neq\emptyset$,
and it follows that $p_1^{R^\star}=\{\zeta\}$
for some $\zeta\in[\delta,\OR^P]$, etc.\qedhere
\end{case}
\end{proof}

\begin{prop}
 Let $n<\om$. Suppose that $R[t]$ is  $n$-sound,
 $\delta<\rho_n^{R[t]}$ and $R^\star$ is a well-defined premouse.
Let $\mu<\rho_n^{R[t]}$
 and $f:[\mu]^{<\om}\to R[t]$ \tu{(}equivalently, $f:[\mu]^{<\om}\to R^\star$\tu{)}. Then $f$ is $\bfrSigma_n^{R[t]}$ iff $f$ is $\bfrSigma_n^{R^\star}$.
\end{prop}
\begin{proof}
If $n=0$, this is because $R[t]$ and $R^\star$ have the same universe,
and if $n>0$, it is an immediate consequence of
\ref{prop:adm_buff_Sigma_1_corresp_passive_largest_card} and \ref{prop:adm_buff_fs_corresp}
 part \ref{item:adm_buff_rSigma_n+1_equiv}.
\end{proof}

\subsection{When $R^+$ is passive and there is no $R^+$-$\delta$-measure}\label{sec:append_ordinals}

Let $N$ be $\star$-iterable. Suppose that either:
\begin{enumerate}[label=(\roman*)]
\item\label{item:F^R'_strictly_overlaps} $N$ is active type 2 with $\crit(F^{N})<\delta$
and there is no $N$-$\delta$-measure, or
\item\label{item:F^R'_strictly_overlaps_tp_3} $N$ is active type 3 with $\crit(F^{N})<\delta$ and there is no $N$-$\delta$-measure, or
\item\label{item:R'-delta-measure} there is an $N$-$\delta$-measure
and $\rho_\om^{N}=\delta$, or
\item\label{item:R'_adm_buffered} $N=Q_m$ (which is admissibly buffered).
\end{enumerate}
Assume also that we have all of the extra assumptions regarding $N$ and $N^\star$ used in the foregoing sections, such as that $N^\star$ is a well-defined premouse,
is $1$-solid, etc, which guarantee the correspondence of fine structure established there.

Let $0<\xi\in\OR$ be such that
$\J_\xi(N[t])$ is a $t$-premouse and $\rho_1^{\J_\alpha(N[t])}\leq\OR^N$ for all $\alpha\in(0,\xi]$, and if \ref{item:R'_adm_buffered} holds, that $\xi\leq\alpha_{m-1}+\ldots+\alpha_0$ (see Definition \ref{dfn:R_i}).

\begin{prop}\label{prop:R_passive_not_adm_buff_no_R-delta_meas}
Let $R,S$ be such that either:
\begin{enumerate}[label=\tu{(}\alph*\tu{)}]
\item\label{R_passive_not_adm_buff_no_R-delta_meas_tp_2} \ref{item:F^R'_strictly_overlaps} holds, $R=N$ and $S$ is the corresponding structure defined in \S\ref{sec:type_2_translate_step},
\item\label{R_passive_not_adm_buff_no_R-delta_meas_R_to_R_J} $R=N$ is active type 3 with $\crit(F^R)<\delta$
and $S=N_J=R_J\neq R=N$, or
\item\label{R_passive_not_adm_buff_no_R-delta_meas_tp_3} \ref{item:F^R'_strictly_overlaps_tp_3} holds, $R=N_J$ and $S$ is the corresponding structure defined in \S\ref{sec:type_3_translate_step}.
\end{enumerate}
Let $0<\alpha\leq\xi$ and $n<\om$. Suppose that $\J_\alpha(R[t])$ is an $n$-sound $t$-premouse which is $1$-solid. Let $r\in[\OR]^{<\om}$ be least
such that $\OR^R\in\Hull_1^{\J_\alpha(R[t])}(\{r,p_1^{\J_\alpha(S[t])}\})$, and let $r'=r\cut \rho_1^{\J_\alpha(R[t])}$.\footnote{Note $r'\cap(\delta+1)=\emptyset$,
since $\rho_1^{\J_\alpha(R[t])}\geq\delta$ by definition, and the parameter $\delta$ can be defined via the $t$-premouse language.}
Then:
\begin{enumerate}
\item $\J_\alpha(R[t])\sub \J_\alpha(S[t])$,
\item $\pow(\beta)\cap\J_\alpha(R[t])=\pow(\beta)\cap\J_\alpha(S[t])$ for all $\beta\in\OR^{\J_\alpha(R[t])}$.
\item $\J_\alpha(S[t])$ is a $t$-premouse,
\item\label{item:passive_after_overlap_rSigma_1_equiv} $\rSigma_1^{\J_\alpha(R[t])}(\{R[t],\alpha_\Tt\})$ statements
about parameters in $\J_\alpha(R[t])$
are recursively equivalent
to $\rSigma_1^{\J_\alpha(S[t])}(\{\alpha_\Tt\})$ statements about those parameters. That is,
there is an $\rSigma_1$ formula $\psi_1$ such that for all $\rSigma_1$ formulas $\varphi$ and all $v\in \J_\alpha(R[t])$, we have
\[ \J_\alpha(R[t])\sats\varphi(v,R[t],\alpha_\Tt)\iff \J_\alpha(S[t])\sats\psi_1(\varphi,v,\alpha_\Tt),\]
and there is an $\rSigma_1$ formula $\tau_1$ such that for all $\rSigma_1$ formulas $\varphi$ and all $v\in \J_\alpha(R[t])$, we have
\[ \J_\alpha(S[t])\sats\varphi(v,\alpha_\Tt)\iff \J_\alpha(R[t])\sats\tau_1(\varphi,v,R[t],\alpha_\Tt).\]

\item\label{item:passive_after_overlap_rho_1_match}
$\rho_{1}^{\J_\alpha(R[t])}=\rho_{1}^{\J_\alpha(S[t])}$,

\item\label{item:passive_after_overlap_p_1_correspondence}
We have:
\begin{enumerate}[label=--]
\item $p_1^{\J_\alpha(R[t])}=p_1^{\J_\alpha(S[t])}\cup r'$
\item
$\OR^R\in\Hull_1^{\J_\alpha(R[t])}(\eta\cup \{p_1^{\J_\alpha(R[t])}\cut(\eta+1)\})$ for each $\eta\in p_1^{\J_\alpha(S[t])}$,
and
\item $\J_\alpha(S[t])$ is $1$-solid.
\end{enumerate}
\item\label{item:passive_after_overlap_n-soundness} $\J_\alpha(S[t])$ is $n$-sound.
\item\label{item:passive_after_overlap_rSigma_i+1_translation} For all $i\in [1,n]$, $\rSigma_{i+1}^{\J_\alpha(R[t])}(\{\alpha_\Tt\})$
statements about parameters in $\J_\alpha(R[t])$
are recursively equivalent
to $\rSigma_{i+1}^{\J_\alpha(S[t])}(\{\alpha_\Tt\})$ statements about those parameters.
That is, there is an $\rSigma_{i+1}$ formula $\psi_{i+1}$
such that for all $\rSigma_{i+1}$ formulas $\varphi$
and all $v\in\J_\alpha(R[t])$, we have
\[ \J_\alpha(R[t])\sats\varphi(v,\alpha_\Tt)\iff \J_\alpha(S[t])\sats\psi_{i+1}(\varphi,v,\alpha_\Tt),\]
and there is an $\rSigma_{i+1}$
formula $\tau_{i+1}$ such that for all $\rSigma_{i+1}$ formulas $\varphi$
and all $v\in\J_\alpha(R[t])$, we have
\[ \J_\alpha(S[t])\sats\varphi(v,\alpha_\Tt)\iff\J_\alpha(R[t])\sats\tau_{i+1}(\varphi,v,\alpha_\Tt).\]
\item\label{item:passive_after_overlap_rho_i+1_match}
$\rho_{i+1}^{\J_\alpha(S[t])}=\rho_{i+1}^{\J_\alpha(R[t])}$ for all $i\in[1, n]$,

\item\label{item:passive_after_overlap_p_i+1_match} $p_{i+1}^{\J_\alpha(S[t])}=
p_{i+1}^{\J_\alpha(R[t])}$
\item\label{item:passive_after_overlap_i+1-solidity} $\J_\alpha(R[t])$ is $(i+1)$-solid iff $\J_\alpha(S[t])$ is $(i+1)$-solid, for all $i\in[1, n]$.
\item\label{item:passive_after_overlap_function_match} Let $\mu<\rho_n^{\J_\alpha(R[t])}=\rho_n^{\J_\alpha(S[t])}$, and let $f:[\mu]^{<\om}\to R[t]$. Then
\[ f\text{ is }\bfrSigma_n^{\J_\alpha(R[t])}\iff f\text{ is }\bfrSigma_n^{\J_\alpha(S[t])}.\]
\end{enumerate}
\end{prop}

\begin{proof}
We proceed by induction on $(\alpha,n)$,
building on the results of \S\ref{sec:type_2_translate_step}. In particular, since $\J(R)$ is a premouse, $R$ is sound,
so $R[t]$ is a sound $t$-premouse, so $S[t]$ is a sound $t$-premouse, so $\J(S[t])$ is a $t$-premouse (although $\J(S)$ is not a premouse, since $S$ is not sound).

Part \ref{item:passive_after_overlap_rSigma_1_equiv}: Working in $\J_\alpha(S[t])$,
$S[t]$ is the least segment $Q[t]$
such that either $Q$ is not sound (under hypotheses \ref{R_passive_not_adm_buff_no_R-delta_meas_tp_2}
and \ref{R_passive_not_adm_buff_no_R-delta_meas_tp_3}),
or $Q$ is not Mitchell-Steel indexed (under hypothesis \ref{R_passive_not_adm_buff_no_R-delta_meas_R_to_R_J}). This allows us to identify $S[t]$ and hence $S$ in a $\Sigma_1$ fashion, and hence compute $R$ and $R[t]$. By induction, we also have $\J_\alpha(R[t])\sub \J_\alpha(S[t])$, and the $\J_\alpha(R[t])$ hierarchy above $R[t]$ is $\rSigma_1^{\J_\alpha(S[t])}$, which suffices to formulate $\psi_1$.

To formulate $\tau_1$, since we are working from the parameter $R[t]$,
it is easy to compute a coded version $S[t]$ and the $\J_\alpha(S[t])$ hierarchy, embedding the elements of $\J_\alpha(R[t])$ into these codes faithfully. This yields $\tau_1$.

Part \ref{item:passive_after_overlap_rho_1_match}: This is an immediate consequence of part \ref{item:passive_after_overlap_rSigma_1_equiv}.

Part \ref{item:passive_after_overlap_p_1_correspondence}: Let $\rho=\rho_1^{\J_\alpha(R[t])}=\rho_1^{\J_\alpha(S[t])}$. Since $\{R[t]\}$ is $\rSigma_1^{S[t]}(\emptyset)$, \[ \pow(\rho)\cap\J_\alpha(R[t])=\pow(\rho)\cap\J_\alpha(S[t]) \]
and $\J_\alpha(R[t])$ is $\rSigma_1^{\J_\alpha(S[t])}$,
certainly $p_1^{\J_\alpha(S[t])}\leq p_1^{\J_\alpha(R[t])}$. And
 $r'\leq p_1^{\J_\alpha(R[t])}$,
 since otherwise $R[t]\notin\Hull_1^{\J_\alpha(R[t])}(\rho\cup\{p_1^{J_\alpha(R[t])}\})$,
 which implies that this hull is bounded in $\J_\alpha(R[t])$, and hence an element of $\J_\alpha(R[t])$, impossible. Let $r'_0$ be the longest (top) segment of $r'$ such that $r'_0\ins p_1^{\J_\alpha(R[t])}$.
 Note that $r'_0\in\Hull_1^{\J_\alpha(S[t])}(\emptyset)$. So if $r'_0=p_1^{\J_\alpha(R[t])}$ then $p_1^{\J_\alpha(S[t])}=\emptyset$
 and we are done.
 So suppose $r'_0\psub p_1^{\J_\alpha(R[t])}$, and let $\gamma_0=\max(p_1^{\J_\alpha(R[t])}\cut r'_0)$. We claim that $\gamma_0=\max(p_1^{\J_\alpha(S[t])})$.
 For note that either $r'=r'_0$, or $\max(r'\cut r'_0)<\gamma_0$, and therefore $R[t]\in\Hull_1^{\J_\alpha(R[t])}(\gamma_0\cup \{r'_0\})$.
 By $1$-solidity for $\J_\alpha(R[t])$,
 we have $t=\Th_1^{\J_\alpha(R[t])}(\gamma_0\cup\{r'_0\})\in \J_\alpha(R[t])$. But since $R[t]$ is in that hull, it follows that $t'=\Th_1^{\J_\alpha(S[t])}(\gamma_0\cup\{r'_0\})$ can be computed from $t$, so $t'\in \J_\alpha(R[t])\sub\J_\alpha(S[t])$. Since we already know that $p_1^{\J_\alpha(S[t])}\sub\gamma_0+1$, it follows that $\gamma_0=\max(p_1^{\J_\alpha(S[t])})$,
 and that $\J_\alpha(S[t])$ is $1$-solid with respect to this parameter.
 Iterating this argument, but relative to the parameter $\gamma_0$ (and then $\{\gamma_0,\gamma_1\}$, etc), one can easily complete the proof.

 Part \ref{item:passive_after_overlap_n-soundness} is immediate by induction and earlier parts.

 Part \ref{item:passive_after_overlap_rSigma_i+1_translation}:  This follows by induction, and using the fact that $\{R[t]\}$ is $\rSigma_2^{\J_\alpha(R[t])}$.

 Part \ref{item:passive_after_overlap_rho_i+1_match} follows immediately from part \ref{item:passive_after_overlap_rSigma_i+1_translation},
 and parts \ref{item:passive_after_overlap_p_i+1_match}
 and \ref{item:passive_after_overlap_i+1-solidity} then follow from both of those parts together.

 Part \ref{item:passive_after_overlap_function_match} follows from earlier parts.
\end{proof}

\begin{prop}\label{prop:R_passive_not_adm_buff_R-delta-meas}
Suppose that \ref{item:R'-delta-measure} holds,
$R=N$ and $S$ is the corresponding structure defined in \S\ref{sec:exists_R-delta-measure}.
Let $(\alpha,r,r')$ be as in Proposition \ref{prop:R_passive_not_adm_buff_no_R-delta_meas} for $(R,S)$,
and suppose that $\J_\alpha(R[t])$ is a $t$-premouse. Then:
\begin{enumerate}
\item $\rho_\om^{R[t]}=\delta=\rho_\om^{S[t]}$ and $R[t]$, $S[t]$ are sound $t$-premice.
\item $\pow(\delta)\cap\J_\alpha(R[t])=\pow(\delta)\cap\J_\alpha(S[t])$.
\item $\J_\alpha(S[t])$ is a $t$-premouse,
and both $\J_\alpha(R[t])$ and $\J_\alpha(S[t])$
have largest cardinal $\delta$.
\item\label{item:passive_after_delta-meas_rSigma_1_equiv} $\rSigma_1^{\J_\alpha(R[t])}(\{R[t]\})$ statements
about parameters in $\delta$
are recursively equivalent
to $\rSigma_1^{\J_\alpha(S[t])}$ statements about those parameters. That is,
there is an $\rSigma_1$ formula $\psi_1$ such that for all $\rSigma_1$ formulas $\varphi$ and all $v\in\delta$, we have
\[ \J_\alpha(R[t])\sats\varphi(v,R[t])\iff \J_\alpha(S[t])\sats\psi_1(\varphi,v),\]
and there is an $\rSigma_1$ formula $\tau_1$ such that for all $\rSigma_1$ formulas $\varphi$ and all $v\in \delta$, we have
\[ \J_\alpha(S[t])\sats\varphi(v)\iff \J_\alpha(R[t])\sats\tau_1(\varphi,v,R[t]).\]

\item\label{item:passive_after_overlap_rho_1_match_delta-measure}
$\rho_{1}^{\J_\alpha(R[t])}=\rho_{1}^{\J_\alpha(S[t])}=\delta$,

\item\label{item:passive_after_overlap_p_1_correspondence_delta-measure} $p_1^{\J_\alpha(R[t])}=r'$
and $p_1^{\J_\alpha(S[t])}=\emptyset$.
\item\label{item:passive_after_overlap_n-soundness_delta-measure} $\J_\alpha(R[t])$ and $\J_\alpha(S[t])$ are sound $t$-premice.
\end{enumerate}
\end{prop}
\begin{proof}
This is like the proof of Proposition \ref{prop:R_passive_not_adm_buff_no_R-delta_meas},
but in this case, $R[t]$ and $S[t]$
project to $\delta$ and are sound,
by \S\ref{sec:exists_R-delta-measure},
which makes things mostly simpler than
in \ref{prop:R_passive_not_adm_buff_no_R-delta_meas}.
However, one wrinkle is that we needn't have $\J_\alpha(R[t])\sub \J_\alpha(S[t])$,
or even $R\sub\J_\alpha(S[t])$.
(This is because there could be subsets of $\delta^{+R}$
(such as $\mu_\delta^R$) which are in $R$,
and hence further sets produced from this which are in $R$, and require some distance to construct.) But $R[t]$ and $S[t]$
they can both compute each other in the codes, which is enough. The fact that $p_1^{\J_\alpha(R[t])}=r'$
is because if $\Hull_1^{\J_\alpha(R[t])}(\rho_1^{\J_\alpha(R[t])}\cup\{p_1^{\J_\alpha(R[t])}\})$ is cofinal in $\J_\alpha(R[t])$, but then since $\delta$ is the largest cardinal of $\J_\alpha(R[t])$,
that hull is just $\J_\alpha(R[t])$ itself,
but note that if $p_1^{\J_\alpha(R[t])}\neq\emptyset$
then $p_1^{\J_\alpha(R[t])}=\{\zeta\}$ for some $\zeta$,
and $\Hull_1^{\J_\alpha(R[t])}(\zeta)=\J_\alpha(R[t])|\zeta$, but if $\zeta>\OR^R$ then $\J_\zeta(R[t])\preccurlyeq_1 \J_\alpha(R[t])$,
so $\rho_1^{\J_\zeta(R[t])}=\zeta$, a contradiction,
so $\zeta\leq\OR^R$, so $R[t]\notin\Hull_1^{\J_\alpha(R[t])}(\zeta)$.
\end{proof}

\begin{prop}\label{prop:adm_buff_N_J_beta(N_i-1[t])_and_J_alpha(N^star)}
Suppose \ref{item:R'_adm_buffered} holds,
so $N=Q_m$ is admissibly buffered
\tu{(}along with our other assumptions mentioned at the start of \S\ref{sec:append_ordinals}, and in particular, $\alpha\leq\alpha_{m-1}+\ldots+\alpha_0$\tu{)}.
Let $\alpha=\alpha_{m-1}+\alpha_{m-2}+\ldots+\alpha_{i}+\beta$
where $0<\beta\leq\alpha_{i-1}$.
Recall that $Q_{i-1}=\J_{\alpha_{i-1}}(N_{i-1})$.
Let $\gamma$ be the largest cardinal of $\J_\beta(N_{i-1}[t])$.
Let $n<\om$ be such that $\J_\beta(N_{i-1}[t])$ is an $n$-sound $t$-premouse and $\delta<\rho_n^{\J_\beta(N_{i-1}[t])}$.
Let $N^\star_P$ be the $P$-premouse given by $N^\star$ \tu{(}recall $P\ins N^\star$ and $\OR^P$ is a strong cutpoint of $N^\star$; $N^\star_P$ has the same universe as does $N^\star$\tu{)}.
Then:
\begin{enumerate}
\item\label{item:gamma_lgst_card_of_J_beta(N^*)} $\gamma$ is the largest cardinal of $\J_\alpha(N^\star)$.
\item\label{item:pow(gamma)_same_J_beta(N_i-1)_and_J_beta(N^star)} $\pow(\gamma)\cap\J_\beta(N_{i-1}[t])=\pow(\gamma)\cap\J_\alpha(N^\star)$.
\item\label{item:functions_into_gamma_same_J_beta(N_i-1)_J_alpha(N^star)}Suppose $\delta<\gamma$. Then:
\begin{enumerate}[label=--]
\item $\J_\beta(N_{i-1}[t])\sub\J_\alpha(N^\star)$.
\item For all $\mu<\gamma$,
all $\lambda<\beta$, all $\ell<\om$
and all $f:[\mu]^{<\om}\to\Ss_{\ell}(\J_\lambda(N_{i-1}))$,\footnote{These are the kinds of functions relevant to forming $\Ult_0(\J_{\beta}(N_{i-1}[t]),E)$
and ``$i_E(\J_\beta(N_{i-1}[t]))$''
as a submodel of $\Ult_0(\J_\alpha(N^\star),E)$.
There \emph{can} be functions $f\in\J_\alpha(N^\star)$
such that $f:[\mu]^{<\om}\to\J_\beta(N_{i-1}[t])$,
or even just $f:[\mu]^{<\om}\to\OR^{\J_\beta(N_{i-1}[t])}$,
such that $f\notin\J_\beta(N_{i-1}[t])$,
because such $f$ can be unbounded in those ordinals. And it can even be that those ordinals have cofinality $\mu=\crit(E)$ in $\J_{\alpha}(N^\star)$. But one should ignore such functions when computing ``$i_E(\J_\beta(N_{i-1}[t]))$''.}
\[ f\in \J_\beta(N_{i-1}[t])\iff f\in \J_\alpha(N^\star).\]
\end{enumerate}
\item\label{item:gamma_and_gamma^+^N_i-1} $\gamma<\gamma^{+N_{i-1}}\leq\OR^{N_{i-1}}$ and
$\gamma^{+N_{i-1}}<\OR^{N^\star}$,
and
\[ \J_\beta(N_{i-1}[t])=\Hull_1^{\J_\beta(N_{i-1}[t])}(\gamma\cup\{\gamma^{+N_{i-1}}\}).\]

\item\label{item:gamma_and_p_1,delta<gamma} Suppose $\delta<\gamma$. Then $p_1^{\J_\beta(N_{i-1}[t])}\neq\emptyset$,
and $\gamma\leq\xi\leq\gamma^{+N_{i-1}}$ where $\xi=\max(p_1^{\J_\beta(N_{i-1}[t])})$, and either $\xi=\gamma$ or $N_{i-1}||\xi\sats$ ``$\gamma$ is the largest cardinal''.

\item\label{item:gamma_and_p_1,delta=gamma} Suppose $\delta=\gamma$ and $p_1^{\J_\beta(N_{i-1}[t])}\neq\emptyset$.
Then $p_1^{\J_\beta(N_{i-1}[t])}=\{\xi\}$ where $\delta<\xi\leq\delta^{+N_{i-1}}$,
and $N_{i-1}||\xi\sats$ ``$\delta$ is the largest cardinal''.

\item\label{item:from_R_m_to_Q^star_rSigma_1_reduce_right_truth_to_left_truth} $\rSigma_1^{\J_\alpha(N^\star_P)}(\{\alpha_\Tt\})$
statements about parameters
$v\leq\gamma^{+N_{i-1}}$ \tu{(}or if $\delta<\gamma$, $v\in\J_\beta(N_{i-1}[t])$\tu{)}
are recursively reducible to
 $\rSigma_1^{\J_\beta(N_{i-1}[t])}(\{\alpha_\Tt\})$ statements about $v$;
 that is, there is an $\rSigma_1$ formula $\tau_1$ such that for all $\rSigma_1$ formulas $\varphi$ and all $v\leq\gamma^{+N_{i-1}}$ \tu{(}or if $\delta<\gamma$, $v\in\J_\beta(N_{i-1}[t])$\tu{)}, we have
 \[ \J_\alpha(N^\star_P)\sats\varphi(v,\alpha_\Tt)\iff \J_\beta(N_{i-1}[t])\sats\tau_1(\varphi,v,\alpha_\Tt).\]
 \item\label{item:from_R_m_to_Q^star_rSigma_1_equiv} $\rSigma_1^{\J_\beta(N_{i-1}[t])}(\{\gamma,\alpha_\Tt\})$ statements about parameters
$v\leq\gamma^{+N_{i-1}}$ \tu{(}or if $\delta<\gamma$, $v\in\J_\beta(N_{i-1}[t])$\tu{)} are
recursively equivalent to
$\rSigma_1^{\J_\alpha(N^\star_P)}(\{\gamma,\alpha_\Tt\})$ statements about $v$. That is, there is an $\rSigma_1$ formula $\psi_1$
such that for all $\rSigma_1$ formulas $\varphi$ and all $v\leq\gamma^{+N_{i-1}}$ \tu{(}or if $\delta<\gamma$, $v\in\J_\beta(N_{i-1}[t])$\tu{)}, we have
\[ \J_\beta(N_{i-1}[t])\sats\varphi(v,\gamma,\alpha_\Tt)\iff \J_\alpha(N^\star_P)\sats\psi_1(\varphi,v,\gamma,\alpha_\Tt),\]
and there is an $\rSigma_1$ formula $\tau_1$ such that for all $\rSigma_1$ formulas $\varphi$ and all $v\leq\gamma^{+N_{i-1}}$ \tu{(}or if $\delta<\gamma$, $v\in\J_\beta(N_{i-1}[t])$\tu{)}, we have
\[ \J_\alpha(N^\star_P)\sats\varphi(v,\gamma,\alpha_\Tt)\iff \J_\beta(N_{i-1}[t])\sats\tau_1(\varphi,v,\gamma,\alpha_\Tt).\]
\item\label{item:rho_1_J_beta(N_i-1)_andJ_beta(N^star)} $\rho_1^{\J_\beta(N_{i-1}[t])}=\rho_1^{\J_\alpha(N^\star_P)}$
\item\label{item:p_1_J_beta(N_i-1)_andJ_beta(N^star)} $p_1^{\J_\beta(N_{i-1}[t])}=p_1^{\J_\alpha(N^\star_P)}$
\item\label{item:1-solidity_J_beta(N_i-1)_and_J_beta(N^star)} $\J_\beta(N_{i-1}[t])$
is $1$-solid iff $\J_\alpha(N^\star_P)$
is $1$-solid.
\item\label{item:1-soundness_J_beta(N_i-1)_and_J_beta(N^star)} $\J_\beta(N_{i-1}[t])$ is $1$-sound
iff $\J_\alpha(N^\star_P)$ is $1$-sound.
\item\label{item:J_alpha(N^star)_n-sound} $\J_\alpha(N^\star_P)$ is an $n$-sound $P$-premouse.
\item\label{item:rho_ell_J_beta(N_i-1)_J_alpha(N^star)} $\rho_\ell^{\J_\beta(N_{i-1}[t])}=\rho_\ell^{\J_\alpha(N^\star_P)}$ for all $\ell\in(0,n]$  and $p_\ell^{\J_\beta(N_{i-1}[t])}=p_\ell^{\J_\alpha(N^\star_P)}$ for all $\ell\leq n$.
\item\label{item:J_beta(N_i-1[t])_J_alpha(N^star)_rSigma_n_functions} If $0<n$ then for all $\mu<\rho_n^{\J_\beta(N_{i-1}[t])}$
and all $f:[\mu]^{<\om}\to\rho_n^{\J_\beta(N_{i-1}[t])}$,
$f$ is $\bfrSigma_n^{\J_\beta(N_{i-1}[t])}$ iff $f$ is $\bfrSigma_n^{\J_\alpha(N^\star)}$.
\item\label{item:from_R_m_to_Q^star_rSigma_n+1_equiv} If $0<n$ then $\rSigma_{n+1}^{\J_\beta(N_{i-1}[t])}(\{\alpha_\Tt\})$ statements
about parameters $v\leq\gamma^{+N_{i-1}}$ \tu{(}or if $\delta<\gamma$, $v\in\J_\beta(N_{i-1}[t])$\tu{)} are recursively equivalent
to $\rSigma_{n+1}^{\J_\alpha(N^\star_P)}(\{\alpha_\Tt\})$
statements about $v$. That is, there is an $\rSigma_{n+1}$ formula $\psi_{n+1}$
such that for all $\rSigma_{n+1}$ formulas $\varphi$ and all $v\leq\gamma^{+N_{i-1}}$ \tu{(}or if $\delta<\gamma$, $v\in\J_\beta(N_{i-1}[t])$\tu{)}, we have
\[ \J_\beta(N_{i-1}[t])\sats\varphi(v,\alpha_\Tt)\iff\J_\alpha(N^\star_P)\sats\psi_1(\varphi,v,\alpha_\Tt),\]
and there is an $\rSigma_{n+1}$
formula $\tau_{n+1}$ such that for all $\rSigma_{n+1}$ formulas $\varphi$
and all $v\leq\gamma^{+N_{i-1}}$ \tu{(}or if $\delta<\gamma$, $v\in\J_\beta(N_{i-1}[t])$\tu{)},
we have
\[ \J_\alpha(N^\star_P)\sats\varphi(v,\alpha_\Tt)\iff\J_\beta(N_{i-1}[t])\sats\tau_1(\varphi,v,\alpha_\Tt).\]
\item\label{item:rho_n+1_J_beta(N_i-1)_andJ_beta(N^star)} $\rho_{n+1}^{\J_\beta(N_{i-1}[t])}=\rho_{n+1}^{\J_\alpha(N^\star_P)}$
\item\label{item:p_n+1_J_beta(N_i-1)_andJ_beta(N^star)} $p_{n+1}^{\J_\beta(N_{i-1}[t])}=p_{n+1}^{\J_\alpha(N^\star_P)}$
\item\label{item:n+1-solidity_J_beta(N_i-1)_and_J_beta(N^star)} $\J_\beta(N_{i-1}[t])$ is $(n+1)$-solid iff $\J_\alpha(N^\star_P)$ is $(n+1)$-solid
\item\label{item:n+1-soundness_J_beta(N_i-1)_and_J_beta(N^star)} $\J_\beta(N_{i-1}[t])$ is $(n+1)$-sound iff $\J_\alpha(N^\star_P)$ is $(n+1)$-sound.
\end{enumerate}
\end{prop}
\begin{proof}
We proceed by induction on $(\alpha,n)$.
Parts
\ref{item:gamma_lgst_card_of_J_beta(N^*)}, \ref{item:pow(gamma)_same_J_beta(N_i-1)_and_J_beta(N^star)} and \ref{item:functions_into_gamma_same_J_beta(N_i-1)_J_alpha(N^star)} are by induction and the correspondence of fine structure in earlier sections (to translate for example between $N_{i-1}[t]$ and $Q_{i}[t]$, etc).
Regarding part \ref{item:functions_into_gamma_same_J_beta(N_i-1)_J_alpha(N^star)}: We get $\J_\beta(N_{i-1}[t])\sub\J_\alpha(N^\star)$ just by computing $\J_\alpha(N^\star)^{\moon}$
over $\J_\alpha(N^\star)$;
this construction is sufficiently local, since $\delta<\gamma$. And the fact that
$f\in\J_\beta(N_{i-1}[t])$ iff $f\in \J_\alpha(N^\star)$
for all functions $f$ as in the second clause follows from the agreement of the models over $\pow(\gamma)$,
since $\J_\beta(N_{i-1}[t])\sats$ ``every set is the surjective image of $\gamma$''.

Part \ref{item:gamma_and_gamma^+^N_i-1}:
We have $\gamma<\OR^{N_{i-1}}$
since $\rho_\om^{N_{i-1}}<\OR^{N_{i-1}}$ and since there is no $\zeta\leq\beta$
such that $\J_\zeta(N_{i-1})$ is admissible.
We have
 $\J_\beta(N_{i-1})=\Hull_1^{\J_\beta(N_{i-1})}(\OR^{N_{i-1}}+1)$
 (because there is no $\zeta\leq\beta$
 such that $\J_\zeta(N_{i-1})$
 is admissible).
 But there is some $\zeta<\beta$
 such that $\J_\zeta(N_{i-1})$
 projects to $\gamma$,
 and note that for the least such $\zeta$, $\{\zeta,N_{i-1}\}$
 is $\rSigma_1^{\J_\beta(N_{i-1})}(\{\gamma^{+N_{i-1}}\})$,
 and hence that $\J_\beta(N_{i-1})=\Hull_1^{\J_\beta(N_{i-1})}(\gamma\cup\{\gamma^{+N_{i-1}}\})$.

We have $\gamma^{+N_{i-1}}<\OR^{N^\star}$ because
either
 there is no $N_{i-1}$-$\delta$-measure and $\OR^{N_{i-1}}<\OR^{N^\star}$,
or there is an $N_{i-1}$-$\delta$-measure, $\gamma=\delta$ and
$\delta^{+N_{i-1}}<\OR^{N^\star}$.

 Part \ref{item:gamma_and_p_1,delta<gamma}:
 Suppose $\delta<\gamma$. Then $p_1^{\J_\beta(N_{i-1})}\not\sub\gamma$,
by condensation for proper segments of $\J_\beta(N_{i-1})$. So letting $\xi=\max(p_1^{\J_\beta(N_{i-1})})$,  by
part \ref{item:gamma_and_gamma^+^N_i-1}, we have $\gamma\leq\xi\leq\gamma^{+N_{i-1}}$.
(Note also that since $H=\Hull_1^{\J_\beta(N_{i-1})}(\gamma\cup\{\xi\})$ is cofinal in $\J_\beta(N_{i-1})$, in fact $N_{i-1}\in H$, and so  $H=\J_\beta(N_{i-1})$.)

Part \ref{item:gamma_and_p_1,delta=gamma}: Suppose $\delta=\gamma$ and $p_1^{\J_\beta(N_{i-1})}\neq\emptyset$.
So there is $\xi$ such that $p_1^{\J_\beta(N_{i-1})}=\{\xi\}$.
We have $\delta<\xi$, since the parameter $\delta$ can be defined via the $t$-premouse language. (Note
also that $\J_\beta(N_{i-1}[t])=\Hull_1^{\J_\beta(N_{i-1}[t])}(\delta\cup\{\xi\})$, since this hull is cofinal.)

Part \ref{item:from_R_m_to_Q^star_rSigma_1_reduce_right_truth_to_left_truth}:
This is a matter of computing $\J_\beta(N_{i-1})^\star$, in the codes,
over $\J_\beta(N_{i-1}[t])$,
and embedding parameters $\leq\gamma^{+N_{i-1}}$ appropriately into these codes \tu{(}and if $\delta<\gamma$, embedding  $v\in\J_\beta(N_{i-1}[t])$ into the codes\tu{)},
noting that this can be done in an $\rSigma_1(\{\alpha_\Tt\})$ fashion.
(Here is the outline. For simplicity suppose that $\J_\beta(N_{i-1})$ has no largest proper segment. Then working over $\J_\beta(N_{i-1}[t])$, given an $\rSigma_1$ formula $\varphi$ and $v\leq\gamma^{+N_{i-1}}$ \tu{(}or if $\delta<\gamma$, $v\in\J_\beta(N_{i-1}[t])$\tu{)},
$\J_\beta(N^\star)\sats\varphi(v)$ iff there is some $\star$-suitable $R\pins \J_\beta(N_{i-1})$
such that $R^\star\sats\varphi(v)$. So it suffices to define in an $\rSigma_1(\{\alpha_\Tt\})$ fashion, uniformly in such $R$, a coded version of $R^\star$
and codes for the parameter $v$. For this, we follow the process in Definition \ref{dfn:R_i} (but starting with our given $R$, instead of the $Q$ there), coded in a reasonable fashion. The ultrapowers involved in the definition of the $\star$-translation can indeed be coded, since we have access to $t$ and $\alpha_\Tt$ and hence $\Tt$ and $P$ and (hence, recursively, codes for) the  extenders involved. And because we are working in $\J_\beta(N_{i-1})$, we do have enough ordinals above $N_{i-1}$ to compute the resulting coded version of $R^\star$. The required code for parameters $v\leq\gamma^{+N_{i-1}}$,
come from their usual ultrapower representations, iterated through the sequence of ultrapowers formed;  for example  $v$ is represented in the first ultrapower by $[\{v\},\id]$, etc. If $\delta<\gamma$ and $v\in\J_\beta(N_{i-1}[t])$ is more arbitrary,
we can first identify (uniformly in $v$) some $\rSigma_1$ term $t$
and some $\eta<\gamma^{+N_{i-1}}$
such that $v=t^{\J_\beta(N_{i-1}[t])}(\eta,\gamma^{+N_{i-1}})$, and formulate an appropriate code for $v$ from
$t$ and codes for $\eta$ and $\gamma^{+N{i-1}}$. We leave the details to the reader. In case $\J_\beta(N_{i-1})$ has a largest proper segment $R$ (which must be $\star$-suitable), we just use Jensen's $\Ss$-hierarchy to stratify $\rSigma_1^{\J_\beta(N^\star)}$, as usual.)

Part \ref{item:from_R_m_to_Q^star_rSigma_1_equiv}: We discussed how to formulate $\tau_1$ in part \ref{item:from_R_m_to_Q^star_rSigma_1_reduce_right_truth_to_left_truth} (for that direction we don't rely on the parameter $\gamma$). For $\psi_1$ (reducing $\J_\beta(N[t])$ truth to $\J_\beta(N^\star)$ truth),
we just compute $\J_\beta(N^\star)^{\moon}$ directly, working in $\J_\beta(N^\star)$, accpeting only stages extending $(N^\star|\gamma)^\moon$ which project to $\gamma$. (See the proof of Proposition \ref{prop:adm_buff_Sigma_1_corresp_passive_largest_card} for more discussion.)

Part \ref{item:rho_1_J_beta(N_i-1)_andJ_beta(N^star)}: This follows from parts \ref{item:pow(gamma)_same_J_beta(N_i-1)_and_J_beta(N^star)} and \ref{item:from_R_m_to_Q^star_rSigma_1_equiv}.

Part \ref{item:p_1_J_beta(N_i-1)_andJ_beta(N^star)}: Let $\rho=\rho_1^{\J_\beta(N_{i-1}[t])}=\rho_1^{\J_\alpha(N^\star)}$.  Then
 $t=\Th_1^{\J_\beta(N^\star)}(\rho\cup\{p_1^{\J_\beta(N_{i-1})}\})$ computes
 $t'=\Th_1^{\J_\beta(N_{i-1})}(\rho\cup\{p_1^{\J_\beta(N_{i-1})}\})$,
by parts
\ref{item:from_R_m_to_Q^star_rSigma_1_equiv} and
 \ref{item:gamma_and_p_1,delta<gamma}
 (and since if $\gamma=\delta$
 then the parameter $\gamma$ is anyway directly
 definable
 via the $t$-premouse language),
 so by part \ref{item:pow(gamma)_same_J_beta(N_i-1)_and_J_beta(N^star)},
 $t\notin \J_\beta(N^\star)$.
 So $p_1^{\J_\beta(N^\star)}\leq p_1^{\J_\beta(N_{i-1})}$.
 But then by part \ref{item:from_R_m_to_Q^star_rSigma_1_reduce_right_truth_to_left_truth}
  and otherwise similarly,
  it follows that $p_1^{\J_\beta(N^\star)}=p_1^{\J_\beta(N_{i-1})}$.

  Part \ref{item:1-solidity_J_beta(N_i-1)_and_J_beta(N^star)}: Suppose $p_1^{\J_\beta(N_{i-1}[t])}\neq\emptyset$, and let $\xi=\max(p_1^{\J_\beta(N_{i-1}[t])})$.
We have $\xi\notin\Hull_1^{\J_\beta(N_{i-1}[t])}(\xi)$ by the minimality of $\xi$,
  so in fact $\Hull_1^{\J_\beta(N_{i-1})}(\xi)=N_{i-1}||\xi\in J_\beta(N_{i-1})$, giving the required $1$-solidity witness at $\xi$. Likewise for $\J_\beta(N^\star)$.
  For the other solidity witnesses,
  the equivalence  follows from part \ref{item:from_R_m_to_Q^star_rSigma_1_equiv}.

  Part \ref{item:1-soundness_J_beta(N_i-1)_and_J_beta(N^star)}: By parts \ref{item:from_R_m_to_Q^star_rSigma_1_equiv}--\ref{item:1-solidity_J_beta(N_i-1)_and_J_beta(N^star)} and since
  the hulls relevant to $1$-soundness include $\gamma$ and $\alpha_\Tt$.

  Parts \ref{item:J_alpha(N^star)_n-sound}, \ref{item:rho_ell_J_beta(N_i-1)_J_alpha(N^star)} and \ref{item:J_beta(N_i-1[t])_J_alpha(N^star)_rSigma_n_functions} are by induction on $n$;
  part \ref{item:J_beta(N_i-1[t])_J_alpha(N^star)_rSigma_n_functions} uses parts \ref{item:from_R_m_to_Q^star_rSigma_1_equiv} and \ref{item:from_R_m_to_Q^star_rSigma_n+1_equiv} from stage $n-1$,
  noting that it suffices to use parameters within $\rho_n\cup\{\pvec_n\}$ to define such functions $f$).

  Part \ref{item:from_R_m_to_Q^star_rSigma_n+1_equiv} is by part \ref{item:from_R_m_to_Q^star_rSigma_1_equiv} and induction (on $n$), and since  $\{\gamma\}$ is $\rSigma_2$-definable over both models.

Parts \ref{item:rho_n+1_J_beta(N_i-1)_andJ_beta(N^star)}--\ref{item:n+1-soundness_J_beta(N_i-1)_and_J_beta(N^star)} follow from
parts
\ref{item:pow(gamma)_same_J_beta(N_i-1)_and_J_beta(N^star)} and \ref{item:from_R_m_to_Q^star_rSigma_n+1_equiv} in the usual manner;
 by parts
 \ref{item:gamma_and_p_1,delta<gamma} and \ref{item:gamma_and_p_1,delta=gamma},
$\pvec_{n+1}^{\J_\beta(N_{i-1}[t])}\sub\gamma^{+N_{i-1}}+1$ and $\pvec_{n+1}^{\J_\alpha(N^\star)}\sub\gamma^{+N_{i-1}}+1$,
so part \ref{item:from_R_m_to_Q^star_rSigma_n+1_equiv} applies to all relevant parameters.
\end{proof}

With setup as above, we can now deduce the fine structural correspondence between $\J_\alpha(N[t])$ and $\J_\alpha(N^\star)$.

\begin{prop}\label{prop:adm_buff_N_J_alpha(N[t])_and_J_alpha(N^star)}
Assume the hypotheses and notation of Proposition \ref{prop:adm_buff_N_J_beta(N_i-1[t])_and_J_alpha(N^star)}.
Let $r\in[\OR]^{<\om}$
be least such that $\OR^N\in\Hull_1^{\J_\alpha(N^\star_P)}(\{r,p_1^{\J_\alpha(N[t])}\})$, and let $r'=r\cut\rho_1^{N^\star_P}$.
Let $n<\om$ and suppose that $\J_\alpha(N[t])$ is an $n$-sound $t$-premouse.
Then:
\begin{enumerate}
\item\label{item:J_alpha(N[t])_J_alpha(N^star)_same_univ} $\J_\alpha(N[t])$ and $\J_\alpha(N^\star)$ have the same universe, with largest cardinal $\gamma$.
 \item\label{item:J_alpha(N[t])_J_alpha(N^star)_rSigma_1_equiv} $\rSigma_1^{\J_\alpha(N[t])}(\{\OR^N,\alpha_\Tt\})$ statements about parameters
$v\in\J_\alpha(N[t])$ are
recursively equivalent to
$\rSigma_1^{\J_\alpha(N^\star_P)}(\{\OR^N,\alpha_\Tt\})$ statements about $v$. That is, there is an $\rSigma_1$ formula $\psi_1$
such that for all $\rSigma_1$ formulas $\varphi$ and all $v\in\J_\alpha(N[t])$, we have
\[ \J_\alpha(N[t])\sats\varphi(v,\OR^N,\alpha_\Tt)\iff \J_\alpha(N^\star_P)\sats\psi_1(\varphi,v,\OR^N,\alpha_\Tt),\]
and there is an $\rSigma_1$ formula $\tau_1$ such that for all $v\in\J_\alpha(N[t])$, we have
\[ \J_\alpha(N^\star_P)\sats\varphi(v,\OR^N,\alpha_\Tt)\iff \J_\alpha(N[t])\sats\tau_1(\varphi,v,\OR^N,\alpha_\Tt).\]
\item\label{item:rho_1_J_alpha(N[t])_J_alpha(N^star)} $\rho_1^{\J_\alpha(N[t])}=\rho_1^{\J_\alpha(N^\star_P)}$
\item\label{item:p_1_J_alpha(N[t])_J_alpha(N^star)} We have:
\begin{enumerate}[label=--]
\item $p_1^{\J_\alpha(N[t])}\cup r'=p_1^{\J_\alpha(N^\star_P)}=p_1^{\J_\beta(N_{i-1}[t])}$.
\item $\OR^N\in\Hull_1^{\J_\alpha(N^\star_P)}(\eta\cup\{p_1^{\J_\alpha(N^\star_P)}\cut(\eta+1)\})$ for each
$\eta\in p_1^{\J_\alpha(N[t])}$.
\end{enumerate}
\item\label{item:1-solidity_J_alpha(N[t])_J_alpha(N^star)} $\J_\alpha(N[t])$
is $1$-solid iff $\J_\alpha(N^\star_P)$
is $1$-solid.
\item\label{item:1-soundness_J_alpha(N[t])_J_alpha(N^star)_} $\J_\alpha(N[t])$ is $1$-sound
iff $\J_\alpha(N^\star_P)$ is $1$-sound.
\item\label{item:J_alpha(N^star)_n-sound_with_J_alpha(N[t])} $\J_\alpha(N^\star_P)$ is an $n$-sound $t$-premouse.
\item\label{item:J_alpha(N^star)_level_n_matchup} $\rho_\ell^{\J_\alpha(N[t])}=\rho_\ell^{\J_\alpha(N^\star_P)}$ for all $\ell\leq n$, and  $p_\ell^{\J_\alpha(N[t])}=p_\ell^{\J_\alpha(N^\star_P)}$ for all $\ell$ such that $2\leq\ell\leq n$.
\item\label{item:J_alpha(N[t])_J_alpha(N^star)_rSigma_n+1_equiv} Suppose $0<n$. Then $\rSigma_{n+1}^{\J_\alpha(N[t])}(\{\alpha_\Tt\})$ statements about parameters
$v\in\J_\alpha(N[t])$ are
recursively equivalent to
$\rSigma_{n+1}^{\J_\alpha(N^\star_P)}(\{\alpha_\Tt\})$ statements about $v$. That is, there is an $\rSigma_{n+1}$ formula $\psi_{n+1}$
such that for all $\rSigma_{n+1}$ formulas $\varphi$ and all $v\in\J_\alpha(N[t])$, we have
\[ \J_\alpha(N[t])\sats\varphi(v,\alpha_\Tt)\iff \J_\alpha(N^\star_P)\sats\psi_{n+1}(\varphi,v,\alpha_\Tt),\]
and there is an $\rSigma_{n+1}$ formula $\tau_{n+1}$ such that for all $v\in\J_\alpha(N[t])$, we have
\[ \J_\alpha(N^\star_P)\sats\varphi(v,\alpha_\Tt)\iff \J_\alpha(N[t])\sats\tau_{n+1}(\varphi,v,\alpha_\Tt).\]
\item\label{item:J_alpha(N[t])_J_alpha(N^star)_rho_n+1} $\rho_{n+1}^{\J_\alpha(N[t])}=\rho_{n+1}^{\J_\alpha(N^\star_P)}$.
\item\label{item:J_alpha(N[t])_J_alpha(N^star)_p_n+1} If $0<n$ then  $p_{n+1}^{\J_\alpha(N[t])}=p_{n+1}^{\J_\alpha(N^\star_P)}$.
\item\label{item:J_alpha(N[t])_J_alpha(N^star)_n+1-solid} $\J_\alpha(N[t])$ is $(n+1)$-solid iff $\J_\alpha(N^\star_P)$ is $(n+1)$-solid.
\item\label{item:J_alpha(N[t])_J_alpha(N^star)_n+1-sound} $J_\alpha(N[t])$ is $(n+1)$-sound
iff $\J_\alpha(N^\star_P)$ is $(n+1)$-sound.
\item\label{item:J_alpha(N[t])_J_alpha(N^star)_functions} Let $\mu<\rho_n^{\J_\alpha(N[t])}$
and $f:[\mu]^{<\om}\to\J_\alpha(N[t])$.
Then $f$ is $\bfrSigma_n^{\J_\alpha(N[t])}$ iff $f$ is $\bfrSigma_n^{\J_\alpha(N^\star_P)}$.
\end{enumerate}
\end{prop}
\begin{proof}
Part \ref{item:J_alpha(N[t])_J_alpha(N^star)_same_univ} is by induction
and the correspondence between $N[t]$ and $N^\star$.

Part \ref{item:J_alpha(N[t])_J_alpha(N^star)_rSigma_1_equiv}: Note that $\J_\alpha(N[t])$ is $\rSigma_1^{\J_\alpha(N^\star_P)}(\{\OR^N\})$, since $N=(N^\star)^\moon$,
and above $N$ and hence $N[t]$, we just construct through to $\J_\alpha(N[t])$.
(We don't want to just compute $\J_\alpha(N^\star)^\moon$, since this could core down too much.)
Conversely, $\J_\alpha(N^\star_P)$
is just $\rSigma_1^{\J_\alpha(N[t])}(\{\alpha_\Tt\})$ (without needing the parameter $\OR^N$).

Part \ref{item:rho_1_J_alpha(N[t])_J_alpha(N^star)} is an immediate consequence
of parts \ref{item:J_alpha(N[t])_J_alpha(N^star)_same_univ} and \ref{item:J_alpha(N[t])_J_alpha(N^star)_rSigma_1_equiv}.

Part \ref{item:p_1_J_alpha(N[t])_J_alpha(N^star)}: By Proposition \ref{prop:adm_buff_N_J_beta(N_i-1[t])_and_J_alpha(N^star)},
$\rho=\rho_1^{\J_\alpha(N[t])}=\rho_1^{\J_\alpha(N^\star_P)}=\rho_1^{\J_\beta(N_{i-1}[t])}$
and
$p_1^{\J_\alpha(N^\star_P)}=p_1^{\J_\beta(N_{i-1}[t])}$.
Also by Proposition \ref{prop:R_passive_not_adm_buff_no_R-delta_meas}
and its proof, $p_1^{\J_\beta(N_{i-1}[t])}=p_1^{\J_\alpha(N[t])}\cup\{s'\}$,
where $s$ is least such that $\OR^{N_{i-1}}\in\Hull_1^{\J_\beta(N_{i-1}[t])}(\rho\cup\{p_1^{\J_\alpha(N[t])}\})$
and $s'=s\cut\rho$.
But $s'=r'$,
by Proposition \ref{prop:adm_buff_N_J_beta(N_i-1[t])_and_J_alpha(N^star)} and some straightforward calculation.
(Note that $\{\OR^N\}$ is $\rSigma_1^{\J_\alpha(N[t])}$,
since $N$ can be identified via its lack of soundness.
It follows then that $p_1^{\J_\alpha(N[t])}\sub\gamma$.
Suppose $\delta<\gamma$;
then $p=p_1^{\J_\beta(N_{i-1}[t])}=p_1^{\J_\alpha(N^\star_P)}\neq\emptyset$; let $\xi=\max(p)$.
Then $\xi\geq\gamma$ by condensation in $\J_\beta(N_{i-1}[t])$.
Further,  $\xi\notin\Hull_1^{\J_\beta(N_{i-1}[t])}(\xi)$, by the minimality of $\xi$,
so $\xi\leq\max(s')$;
likewise for $\J_\alpha(N^\star_P)$,
so $\xi\leq\max(r')$.
But $H=\Hull_1^{\J_\alpha(N^\star_P)}(\xi+1)$ is cofinal in $\J_\alpha(N^\star_P)$,
and it follows that $\OR^N\in H$.
(For given any $\eta\in[0,\alpha)$,
there is some $\ell\leq m$ such that
$N^\star$ is the unique segment $N'\ins\J_\eta(N^\star)$ such that
there are exactly $\ell$ segments $N''$ with $N'\ins N''\ins\J_\eta(N^\star)$
and $(N'')^\moon$ is unsound.)
Thereore $\max(r')\leq\xi$, and so we have shown that
 $\max(r')=\xi$. Also,
 by Propostion \ref{prop:adm_buff_N_J_beta(N_i-1[t])_and_J_alpha(N^star)},
 and since we can refer to the parameter $\xi$, and hence also $\gamma$, we can convert between $\rSigma_1$ truth over $\J_\beta(N_{i-1}[t])$ and $\J_\alpha(N^\star_P)$ (for parameters in $\J_\beta(N_{i-1}[t])$), in a simple enough fashion,
that one can now obtain $s'=r'$.
(We can pass back and forth between defining the point $\OR^{N_{i-1}}$
over $\J_\beta(N_{i-1}[t])$
and defining the point $\OR^N$
over $\J_\alpha(N^\star_P)$,
since they compute each other via
$\star$-translation and $\Moon$-construction.)

Part \ref{item:1-solidity_J_alpha(N[t])_J_alpha(N^star)}:
This is by a similar argument as in
the proof of Proposition \ref{prop:R_passive_not_adm_buff_no_R-delta_meas}. Since $\{\OR^N\}$
is $\rSigma_1^{\J_\alpha(N[t])}$,
we get that $r'\in\Hull_1^{\J_\alpha(N[t])}(\rho\cup\{p_1^{\J_\alpha(N[t])}\})$
(where $\rho=\rho_1^{\J_\alpha(N[t])}$),
and so if $\J_\alpha(N[t])$ is $1$-solid, its witness for $\eta\in p_1^{\J_\alpha(N[t])}$ encodes a witness for $\J_\alpha(N^\star_P)$
and $\alpha$. We get the witnesses for those $\alpha\in r'$ automatically,
because the corresponding hulls do not include $\OR^N$, and so they are bounded in $\J_\alpha(N[t])$.
Conversely, if $\J_\alpha(N^\star_P)$ is $1$-solid, and $\eta\in p_1^{\J_\alpha(N[t])}$,
then since $\OR^N\in\Hull_1^{\J_\alpha(N^\star_P)}(\eta\cup \{p_1^{\J_\alpha(N^\star_P)}\cut(\eta+1)\})$,
its solidity witness at $\eta$
encodes one for $\J_\alpha(N[t])$
at $\eta$.

Part \ref{item:1-soundness_J_alpha(N[t])_J_alpha(N^star)_}: This follows immediately
from parts \ref{item:J_alpha(N[t])_J_alpha(N^star)_rSigma_1_equiv}, \ref{item:p_1_J_alpha(N[t])_J_alpha(N^star)} and \ref{item:1-solidity_J_alpha(N[t])_J_alpha(N^star)}.

Parts
\ref{item:J_alpha(N^star)_n-sound_with_J_alpha(N[t])} and \ref{item:J_alpha(N^star)_level_n_matchup} are just by induction.

Part \ref{item:J_alpha(N[t])_J_alpha(N^star)_rSigma_n+1_equiv}: If $n=1$,
this follows from parts
\ref{item:rho_1_J_alpha(N[t])_J_alpha(N^star)} and \ref{item:J_alpha(N[t])_J_alpha(N^star)_rSigma_1_equiv}, noting that  $\{\OR^N\}$
is $\rSigma_1^{\J_\alpha(N[t])}$
and $\rSigma_2^{\J_\alpha(N^\star_P)}$,
so we can indeed apply part \ref{item:J_alpha(N[t])_J_alpha(N^star)_rSigma_1_equiv}.
(To see that $\{\OR^N\}$ is $\rSigma_2^{\J_\alpha(N^\star_P)}$,
note that we can identify $\OR^N$ much like we did in the proof of part \ref{item:p_1_J_alpha(N[t])_J_alpha(N^star)} from ordinals above $\OR^N$.)
If $n>1$, it follows directly by induction, using parts \ref{item:J_alpha(N^star)_level_n_matchup} and \ref{item:J_alpha(N[t])_J_alpha(N^star)_rSigma_n+1_equiv}.

Parts
\ref{item:J_alpha(N[t])_J_alpha(N^star)_rho_n+1},  \ref{item:J_alpha(N[t])_J_alpha(N^star)_p_n+1}, \ref{item:J_alpha(N[t])_J_alpha(N^star)_n+1-solid} and \ref{item:J_alpha(N[t])_J_alpha(N^star)_n+1-sound} now follow as usual from part \ref{item:J_alpha(N[t])_J_alpha(N^star)_rSigma_n+1_equiv},
and since the distinction between $p_1^{\J_\alpha(N[t])}$
and $p_1^{\J_\alpha(N^\star_P)}$
does not affect the value of $p_{n+1}$,
since $p_1^{\J_\alpha(N^\star_P)}\in H=\Hull_1^{\J_\alpha(N[t])}(\{p_1^{\J_\alpha(N[t])}\})$,
since $p_1^{\J_\alpha(N[t])},\OR^N\in H$. (Use the usual trick of finding a descending sequence of candidates for $r$ (the original $r$, not $r'$),
and letting $r_\infty$ be the last one found, note that $r'\ins r_\infty$.)

Part \ref{item:J_alpha(N[t])_J_alpha(N^star)_functions} is an immediate consequence of parts \ref{item:J_alpha(N[t])_J_alpha(N^star)_same_univ}, \ref{item:J_alpha(N[t])_J_alpha(N^star)_rSigma_1_equiv} and \ref{item:J_alpha(N[t])_J_alpha(N^star)_rSigma_n+1_equiv} (the latter at the previous induction level).
\end{proof}
\section{Iterability of $R^\star$}\label{sec:iterability_*-trans}

Let $b=\Sigma(\Tt)$ and $Q=Q(\Tt,b)$
and $q<\om$ be such that $\rho_{q+1}^Q\leq\delta<\rho_q^Q$ (if $Q=M^\Tt_b|\delta$, $\star$-translation is not needed, so we can ignore this case).
Let $\Phi(\Tt)\conc ((Q,q),\delta)$
be the phalanx of $\Tt$ followed by $(Q,q)$, with exchange ordinal $\delta$ associated to $\Tt$.
(Degree-maximal trees $\Uu$ on the phalanx have $\deg^\Uu_0=q$ and $\delta<\lh(E^\Uu_0)$.)
Fix a degree-maximal strategy $\Sigma$ for the phalanx.

Say a degree-maximal tree $\Uu$ on $\Phi(\Tt)\conc ((Q,q),\delta)$ is \emph{$\star$-valid}
iff
 for each $\alpha+1<\lh(\Uu)$,
it is not the case that $\delta^{+M^\Uu_\alpha|\lh(E^\Uu_\alpha)}<\lh(E^\Uu_\alpha)$
and $M^\Uu_\alpha||\lh(E^\Uu_\alpha)\sats$ ``there is a $\delta$-measure'',

We will show that
$\star$-valid trees $\Uu$ on $\Phi(\Tt)\conc Q$ via $\Sigma$
correspond precisely
to (putative) $\deg^\star(Q)$-maximal above-$\delta$ trees $\Uu'$ on $Q^\star$
via the (putative) strategy $\Sigma'$, where:
\begin{enumerate}
\item The extenders used in $\Uu'$
are precisely those used in $\Uu$
which have critical point $>\delta$.
\item  $\Uu'$ is padded,
setting $E^{\Uu'}_\alpha=\emptyset$
just when $\crit(E^\Uu_\alpha)\leq\delta$. (Note then that for every limit $\eta<\lh(\Uu)$,
$\eta$ is a limit of non-padded stages of $\Uu'$, since $\Uu$ can only use finitely many extenders in a row all with critical points $\leq\delta$.)
\item Write $\nu^\Uu_\alpha$ and $\nu^{\Uu'}_\alpha$ for the exchange ordinals associated to the nodes of $\Uu,\Uu'$. For $\Uu$, these are $\nu^\Uu_\alpha=\nu(E^\Uu_\alpha)$ as usual. For $\Uu'$, they are just $\nu^{\Uu'}_\alpha=\nu^\Uu_\alpha$ (so
$\nu^{\Uu'}_\alpha=\nu(E^{\Uu'}_\alpha)$ when $E^{\Uu'}_\alpha\neq\emptyset$).
Then $\pred^\Uu$ and $\pred^{\Uu'}$ are determined as usual using the exchange ordinals $\nu^\Uu_\alpha=\nu^{\Uu'}_\alpha$.
\item
If $E^{\Uu'}_\alpha=\emptyset$ then (for bookkeeping purposes),
set $M^{\Uu'}_{\alpha+1}=R^\star$
and $\deg^{\Uu'}_{\alpha+1}=\deg^\star(R)$, where:
\begin{enumerate}[label=--]
\item if $\crit(E^\Uu_\alpha)<\delta$
then $R=\mathrm{exit}^\Uu_\alpha$, and
\item if $\crit(E^\Uu_\alpha)=\delta$ then  $R\ins M^\Uu_\alpha$ is least such that  $\mathrm{exit}^\Uu_\alpha\ins R$ and either $R=M^\Uu_\alpha$ or $\rho_\om^R=\delta$.
 \end{enumerate}
 (We will see that this is indeed a segment of $M^{\Uu'}_\alpha$,
 and does not restrict the possibilities for $\Uu'$.)
\item For each limit $\eta<\lh(\Uu)$,
$[0,\eta)^{\Uu'}$ agrees with $[0,\eta)^{\Uu}$ on a tail.
\end{enumerate}

This determines $\Uu'$ from $\Uu$,
and vice versa. We also have:
\begin{enumerate}
\item $\pred^{\Uu'}_{\alpha+1}=\pred^\Uu_{\alpha+1}$ for all $\alpha+1<\lh(\Uu)$
such that $\crit(E^\Uu_\alpha)>\delta$ (here $E^{\Uu'}_\alpha=E^\Uu_\alpha$).

\item Let $\eta<\lh(\Uu)$.
Let the \emph{padding support $P^{\Uu'}_\eta$ of $\eta$} be the set of all $\alpha<^{\Uu'}\eta$ such that $E^{\Uu'}_\alpha=\emptyset$.
 Then $P^{\Uu'}_\eta$ can be characterized recursively as follows: let $\gamma=\mathrm{root}^\Uu(\eta)$. If $\gamma=\lh(\Tt)$ (that is, $M^\Uu_\eta$ is above $Q$) then $P^{\Uu'}_\eta=\emptyset$. Suppose $\gamma<\lh(\Tt)$ and let $\alpha+1=\mathrm{succ}^\Uu(\gamma,\eta)$. Then $P^{\Uu'}_\eta=\{\alpha\}\cup P^{\Uu'}_\alpha$. It follows that $P^{\Uu'}_\eta$ is finite.
 Further, letting $\left<\alpha_i\right>_{i\leq n}$ be the increasing enumeration of $P^{\Uu'}_\eta\cup\{\eta\}$, then $\vec{E}^{\Uu'}_{0\eta}$
 (the sequence of extenders used along $[0,\eta]^{\Uu'}$) is just
 \[ \vec{E}^{\Uu}_{0\alpha_0}\conc\vec{E}^{\Uu}_{\alpha_0+1,\alpha_1}\conc\ldots\conc\vec{E}^{\Uu}_{\alpha_{n-1}+1,\eta}.\]
\item For every $\alpha<\lh(\Uu)$,
$M^{\Uu'}_\alpha=(M^\Uu_\alpha)^\star$
and $\deg^{\Uu'}_\alpha=\deg^{\star}(M^\Uu_\alpha)$.
\item\label{item:it_star_M^*_alpha+1} For every $\alpha+1<\lh(\Uu)$
with $\delta<\crit(E^\Uu_\alpha)$,
letting $\kappa=\crit(E^\Uu_\alpha)$
 and $\beta=\pred^\Uu(\alpha+1)$, we have:
\begin{enumerate}[label=--]
\item $M^{*\Uu'}_\beta=(M^\Uu_\beta)^\star$,
\item  $\pow(\kappa)\cap (M^{\Uu'}_\beta[t])=\pow(\kappa)\cap M^{\Uu}_\beta$,
so $\kappa^{+M^{\Uu'}_\beta}=\kappa^{+ M^\Uu_\beta}$,
 \item  $\alpha+1\in\mathscr{D}^{\Uu'}$
 iff $\alpha+1\in\mathscr{D}^{\Uu}$,
\item $M^{*\Uu'}_{\alpha+1}=(M^{*\Uu}_{\alpha+1})^\star$
and $\deg^{\Uu'}_{\alpha+1}=
\deg^{\star}(M^{*\Uu}_{\alpha+1},\deg^\Uu_{\alpha+1})$.\footnote{Define $\deg^\star(M^{*\Uu}_{\alpha+1})$
has the maximal degree $n$ such that
$(M^{\Uu}_{\alpha+1})^\star$ is $n$-sound and $\delta<\rho_n((M^\Uu_{\alpha+1})^\star)$,
and if $M^\Uu_{\alpha+1}$ is $j$-sound
and $\delta\leq\mu<\rho_j(M^\Uu_{\alpha+1})$,
then
 $\deg^\star(M^{*\Uu}_{\alpha+1},\mu)$
 is the maximal degree $n$
 such that $(M^{\Uu}_{\alpha+1})^\star$ is $n$-sound and
 $\mu<\rho_n((M^\Uu_{\alpha+1})^\star)$.}
 \end{enumerate}
 \end{enumerate}

We will prove by induction on $\lh(\Uu)=\lh(\Uu')$ that for $\Uu,\Uu'$ built as above  (with $\Uu'$ a putative tree), we have:
\begin{enumerate}[label=--]
\item $\Uu'$ is in fact an iteration tree (has well-defined and wellfounded models) and the above properties hold,
\item $\Uu'$ can be freely extended
(as a putative $\deg^\star(Q)$-maximal tree),
along with a corresponding extension of $\Uu$ (as a degree-maximal tree),
\item $\Uu$ can be freely extended (as a degree-maximal tree), along with a corresponding extension of $\Uu'$
(as a putative $\deg^\star(Q)$-maximal tree).
\end{enumerate}

It will follow that $\Sigma'$ is
a $\deg^\star(Q)$-maximal above-$\delta$ strategy for $Q^\star$, as desired.

If $\Uu,\Uu'$ are trivial (use no extenders), then the desired properties are immediate. Now consider a successor step. Say we have $\lh(\Uu)=\lh(\Uu')=\alpha+1$, and $\Uu'$ has wellfounded last model,
and the properties listed above hold for the pair. So $M^{\Uu'}_{\alpha}=(M^\Uu_\alpha)^\star$, and considering the correspondence of the extender sequences between these models, it follows that we can indeed extend $\Uu,\Uu'$ appropriately. At this stage,
we extend only by choosing $E\in\es(M^\Uu_\alpha)$ with $\lh(E^\Uu_\beta)\leq\lh(E)$ for all $\beta<\alpha$. If there is $\beta<\alpha$ such that $E^{\Uu'}_\gamma=\emptyset$ for all $\gamma\in[\beta,\alpha)$,
and we want to extend $\Uu'$
with $E\in\es^{\Uu'}_\alpha$
where $\lh(E)<\lh(E^\Uu_\gamma)$ for some $\gamma\in[\beta,\alpha)$,
then we shouldn't be extending $\Uu'$;
we should be extending some proper segment of $\Uu'$ instead. Thus,
given $E\in\es_+(M^\Uu_\alpha)$
with $\lh(E^\Uu_\beta)\leq\lh(E)$
for all $\beta<\alpha$,
if  $\crit(E)>\delta$
then $E\in\es_+((M^{\Uu}_\alpha)^\star)=\es_+(M^{\Uu'}_\alpha)$, so we can extend both $\Uu$ and $\Uu'$ with $E$.
Conversely, given $E\in\es_+(M^{\Uu'}_\alpha)$ with $\lh(E^\Uu_\beta)<\lh(E)$ for all $\beta<\alpha$,
and letting $(F_0,\ldots,F_{m-1})$
be the $\star$-support of $E$ for $M^\Uu_\alpha$, then $E^\Uu_{\alpha+i}=F_i$ for $i<m$, and $E^\Uu_{\alpha+m}=E$ (so $E^{\Uu'}_{\alpha+i}=\emptyset$
for $i<m$ and $E^{\Uu'}_{\alpha+m}=E$).

Suppose we extend $\Uu$ with $E^\Uu_\alpha=E$
having $\crit(E)\leq\delta$.
Then note that $(M^\Uu_{\alpha+1})^\star\ins (M^\Uu_\alpha)^\star=M^{\Uu'}_\alpha$,
and letting $n=\deg^{\star}(M^\Uu_{\alpha+1})$,
we have $\rho_{n+1}((M^\Uu_{\alpha+1})^\star)\leq\nu^\Uu_\alpha<\rho_n((M^\Uu_{\alpha+1})^\star)$, so setting $M^{\Uu'}_{\alpha+1}=(M^\Uu_{\alpha+1})^\star$ and $\deg^{\Uu'}_{\alpha+1}=n$ is merely a bookkeeping device; it does not artificially restrict the possibilities for $\Uu'$ (given that we have already imposed the restrictions given by setting $E^\Uu_\alpha=E$).

Now suppose instead that $E^\Uu_\alpha=E$, with $\crit(E)>\delta$. So $E^{\Uu'}_\alpha=E$. Let $\beta=\pred^\Uu(\alpha+1)=\pred^{\Uu'}(\alpha+1)$. We have $M^{\Uu'}_\beta=(M^\Uu_\beta)^\star$
and $\deg^{\Uu'}_\beta=\deg^\star(M^\Uu_\beta)$.
By the level-by-level fine structural correspondence between $M^\Uu_\beta$ and $(M^\Uu_\beta)^\star$ (including the correspondence of cardinals, etc),
it follows that $M^{*\Uu'}_{\alpha+1}=(M^{*\Uu}_{\alpha+1})^\star$
and $\deg^\star(M^{*\Uu}_{\alpha+1})$.
Let $\mu=\crit(E)$ and $a\in[\nu(E)]^{<\om}$. Let $n=\deg^\Uu_{\alpha+1}$ and $n+k=\deg^{\Uu'}_{\alpha+1}$.
By \S\ref{sec:*-translation},
for each $f:[\mu]^{<\om}\to M^{*\Uu}_{\alpha+1}$,
\[ f\text{ is }\bfrSigma_n^{M^{*\Uu}_{\alpha+1}[t]}\iff f\text{ is }\bfrSigma_{n+k}^{M^{*\Uu'}_{\alpha+1}}.\]
But
\[ f\text{ is }\bfrSigma_n^{M^{*\Uu}_{\alpha+1}[t]}\implies
\exists X\in E_a\ \big[f\rest X\text{ is }\bfrSigma_n^{M^{*\Uu}_{\alpha+1}}\big],\]
and conversely,
\[ f\text{ is }\bfrSigma_n^{M^{*\Uu}_{\alpha+1}}\implies
 f\text{ is }\bfrSigma_n^{M^{*\Uu}_{\alpha+1}}.\]

It follows that the ultrapowers
\[ \Ult_n(M^{*\Uu}_{\alpha+1},E),\]
\[ \Ult_n(M^{*\Uu}_{\alpha+1}[t],E),\]
\[ \Ult_{n+k}((M^{*\Uu}_{\alpha+1})^\star,E) \]
all act on $M^{*\Uu}_{\alpha+1}\sub M^{*\Uu}_{\alpha+1}[t]\sub (M^{*\Uu}_{\alpha+1})^\star$
in the same fashion, and so it is straightforward to see that
\[ (\Ult_n(M^{*\Uu}_{\alpha+1},E))^\star=\Ult_{n+k}(M^{*\Uu}_{\alpha+1},E)=\Ult_{n+k}(M^{*\Uu'}_{\alpha+1},E)=M^{\Uu'}_{\alpha+1} \]
and
\[ \deg^\star(M^{\Uu}_{\alpha+1})=\deg^{\Uu'}_{\alpha+1}.\]

Because of this, it follows that $M^{\Uu'}_{\alpha+1}$ is wellfounded. That is, since $(\Tt\conc b)\conc (\Uu\conc\left<E\right>)$ is via $\Sigma$,
we can extend $\Uu\conc\left<E\right>$ and its initial segments,
with extenders involved in buildng $(M^\Uu_{\alpha+1})^\star$, showing
that $M^{\Uu'}_{\alpha+1}=(M^\Uu_{\alpha+1})^\star$ is wellfounded.

Propagating the inductive hypotheses through limit stages $\eta$ is straightforward, with a minor adaptation of the foregoing calculations to see that $(M^\Uu_\eta)^\star=M^{\Uu'}_\eta$.
(Note that, letting $\gamma=\mathrm{root}^\Uu(\eta)$ and $\alpha+1=\mathrm{succ}^\Uu(\gamma,\eta)$, then for every $\beta+1\in(\alpha+1,\eta]^\Uu=(\alpha+1,\eta]^{\Uu'}$,
we have $\delta<\crit(E^\Uu_\eta)$,
so $E^{\Uu'}_\beta=E^\Uu_\beta\neq\emptyset$.
So there there are no drops in model along $(\alpha+1,\eta)^{\Uu'}$ which are simply induced by padding.
(It can be that $\crit(E^\Uu_\gamma)\leq\delta$ where $\gamma\in[\alpha+1,\eta)^\Uu$, but then $\gamma+1\notin[\alpha+1,\eta)^\Uu$.)
Therefore, by property \ref{item:it_star_M^*_alpha+1},
$\mathscr{D}^{\Uu'}\cap(\alpha+1,\eta)^{\Uu'}=\mathscr{D}^{\Uu}\cap(\alpha+1,\eta)^{\Uu}$,
and in particular, this set is finite.
Now moving above all drops along these branches, it is easy to adapt the calculations from the successor step to stage $\eta$.)

\section{Self-iterability}

Let $R$ be the stack of all $\om$-mice (that is, requiring $(\om,\om_1+1)$-iterability).
Then $R\sats$ ``I am $(\om,\om_1)$-iterable'', and computes the correct restriction of the correct strategies. This is via genericity inflation of trees, using $\star$-translation. We just need to see that this never runs through $\OR^R$ many stages. If $\OR^R=\om_1^V$
then it can't, because in that case we just get a contradiction in $V$ by using $(\om,\om_1+1)$-iterability there. So suppose $\OR^R<\om_1^V$.
Let $R^+$ be the stack of $(\om,\om_1+1)$-iterable mice $S$
such that $R\pins S$
and $\rho_\om^S\leq\OR^R$.
Then in fact there is no such $S$
with $\rho_\om^S=\om$ (since this would contradict the definition of $R$), so those $S$ have $\rho_\om^S=\OR^R$, and $S\sats\ZFC^-$. But now the failure of termination of the process in $R$
implies that the resulting Q-structure $Q$ has $Q^\star\pins R^+$,
so  $Q\in R^+$, which leads to a contradiction in $R^+$.

Suppose $\OR^R<\om_1^V$.
Let $R^+$ be as above.
Then in $R^+\sats$ ``$R$ is $(\om,\om_1+1)$-iterable'';
this follows by relativizing the preceding proof; i.e. $R^+\sats$ ``I am $\OR$-iterable''.

\section{$\star$-translation for Varsovian models (incomplete)}

NOTE: At this stage, this section gives the main ideas for what is needed for the adaptation of $\star$-translation to the form needed for the papers \cite{vm2} and \cite{vmom} (and also rather much further),  but is at this point lacking a careful discussion of the $\star$-translation details. This will be added in due course, but what is missing is essentially very similar to the material in \S\ref{sec:*-translation} and \S\ref{sec:iterability_*-trans}.

In this section we will develop a variant of the $\star$-translation of the previous section, one which is which is owed by the papers \cite{vm2} and \cite{vmom}. A big difference here is that the background model $M$ can have extenders overlapping $\delta=\delta(\Tt)$. However,
there will be strong restrictions on such overlaps and how they relate to corresponding overlaps of $\delta$ in the Q-structure $Q$ (or variant thereof), which will be important in handling them.
In forming the $\star$-translation,
when we encounter extenders $E$ overlapping $\delta$ on the sequence of our $\star$-suitable model $Q$, we will not take the ultrapower by $E$, unlike in usual $\star$-translation. Instead,
we will (attempt to) use them to determine an overlapping extender $F$ to put on the sequence of the $\star$-translation. This $F$ might not agree with $E$ on the ordinals, but $F\rest\OR$
will still be easily determined by $E\rest\OR$, because of other details of our circumstances to be explained below.

Another difference is  that if we have $E$ as above
and $F\rest\OR$ determined from $E$ as (not yet explained in detail) above, it does not seem immediately obvious to the author
that we always get an extender $F$ extending $F\rest\OR$ which, when added to the sequence of the $\star$-translation, actually yields a premouse. This complicates the proof used for regular $\star$-translation.
For in that proof, we first showed that $Q^\star$ was a premouse,
with fine structure corresponding tightly to that of $Q$,
and then we showed that $Q^\star$ is appropriately iterable,
and from there, a standard comparison argument shows that $Q^\star$ must line up with the relevant   mice on the ``$\star$'' side.
But here it seems we can't directly show that $Q^\star$ is even a premouse, if we encounter a problematic $E$ and $F\rest\OR$  as just mentioned, with $E\in\es^Q$. In order to get around this, we will now instead
do everything at once.
In order to show that $Q^\star$ does line up with the relevant mice $N$ on the ``$\star$'' side,
we will argue directly by comparison, without knowing in advance that $Q^\star$ is even a premouse. If it has any initial segments which are not premice, then since $N$ is a premouse,
and $R\ins Q^\star$ is the least non-premouse, then
either $N\pins R\ins Q^\star$
(so we are done) or there is a disagreement between $Q^\star$ and $N$, and the least such disagreement is between $R$ and $N$. The first step of comparison will remove this least disagreement. (More precisely,
in the tree $\Uu$ on $\Phi(\Tt)\conc(Q,q)$,
there will in general  be a finite sequence $E^\Uu_0,\ldots,E^\Uu_{n-1}$ of extenders overlapping $\delta$
which are used first in $\Uu$, in order to reveal the least disagreement, and we set $E^\Vv_0=\ldots=E^\Vv_{n-1}$,
where $\Vv$ is the tree on $N$.
Then $E^\Uu_n,E^\Vv_n$ are
the extenders indexed at the
 least disagreement. (Maybe $E^\Uu_n=\emptyset$ or $E^\Vv_n=\emptyset$.
 If $E^\Uu_n\neq\emptyset$,
 we specify here that $E^\Uu_n$ is not a ``$\star$-overlapping type'', meaning that it does not get iterated away when forming the $\star$-translation. (Recall that ``$\star$-overlapping type'' is no longer the same thing as overlapping $\delta$.)
 It could be that $\lh(E^\Vv_n)=\lh(E^\Uu_{n-1})$,
 but then $E^\Uu_n=\emptyset$.\footnote{We are now working in Jensen indexing, another difference with the earlier section. This is because the Varsovian models papers work in Jensen indexing.}  Otherwise,
 $\lh(E^\Uu_{n-1})<\lh(E^\Uu_n)$ or $\lh(E^\Uu_{n-1})<\lh(E^\Vv_n)$, whichever is defined,
 and whichever of these might be non-empty (with at least one non-empty). For ``strongly overlapping'' extenders (those that overlap but are not of $\star$-overlapping type),
 we determine ``difference''
 modulo the appropriate conversion between $E$ and $F$ mentioned earlier. Then one can show that the comparison must terminate.
 But then we use the fact that the $\star$-translation of the final model of $\Uu$ is well-behaved up through the final model of $\Vv$
 (is a premouse, and has appropriate fine structure)
 to now reach the usual contradiction via fine structure etc.

 One more, minor, difference is that the base mouse will be presented with $\lambda$-indexing, so everything will be in $\lambda$-indexing.
 Thus, some of the details of $\star$-translation need some slight modification.

Now let us discuss more detail.
 Let $M$ be a  proper class mouse, with $\lambda$-indexing.
Assume that $M$ has  no strong cardinal $\kappa$ which is a limit of strong cardinals.
We assume that $M$ has
  the right sort of fine structure that allows us to analyse the kind of comparison argument which arises in the Varsovian model context \cite{vm2}, \cite{vmom};
  we will indeed be using precisely one of those kind of comparison arguments, just a little more abstractly, and with a fully developed $\star$-translation.\footnote{The details of that $\star$-translation are yet to be written down, but they are very close to those in the earlier section.} Assume that $M$ has only set-many measurable cardinals, and we have some ordinal $\lambda^M$ above all measurables of $M$ such that $M|\lambda^M\sub\Hull^M(\Gamma)$ for any proper class $\Gamma$ of ordinals.

 Let $\vec{\kappa}^M=\left<\kappa^M_\alpha\right>_{\alpha<\theta}$ enumerate the strong cardinals of $M$ below some point.
 Let $W$ be a proper class iterate of $M$, via its correct strategy, and $i:M\to W$ the iteration map,
 and $\kappa^W_\alpha=i(\kappa^M_\alpha)$. Let $\vec{j}=\left<j_\alpha\right>_{\alpha<\theta}$ be such that $j_\alpha:(\kappa_\alpha^M)^{+M}\to(\kappa_\alpha^W)^{+W}$ is a cofinal map for each $\alpha<\theta$. Note
 here that $j_\alpha$ need not agree with $i$, and moreover, $j_\alpha$ need not agree with $j_\beta$ for $\alpha<\beta$. In the Varsovian model papers \cite{vm2}, \cite{vmom},
 the relevant maps $j_\alpha$ arise as restrictions of iteration acting on a direct limit ``at'' $(\kappa_\alpha^M)^{+M}$ (for example,
 if $\alpha=0$ then we have the first direct limit $M_\infty$,
 and $j_0$ is an iteration map on $M_\infty$). We will not spell out just yet exactly what properties we need of the $j_\alpha$'s, but will just point out what they need to do when we get there. (They will play the usual role that they played in the corresponding arguments in \cite{vm2} and \cite{vmom}.)

 Let $\sup_{\alpha<\theta}(\kappa^W_\alpha)^{+W}\leq\lambda<\eta$,
 with $\zeta$ a regular cardinal in $W$,
 and such that $W$ has no Woodin cardinals $\delta$ such that $\lambda<\delta<\zeta$.
 Suppose that
 for all $(\kappa,\delta',E)$ such that
  $E\in\es^{W}$ and $E$ is $W$-total
  and $\crit(E)=\kappa<\zeta$
  and $\lambda<\lh(E)$
  and $\kappa<\delta'<\lh(E)$
 and $\delta'$ is Woodin in $W|\lh(E)$, we have that $\kappa=\kappa^W_\alpha$ for some $\alpha<\theta$.

 Suppose that $W|\zeta\in M$.
 Let $\Tt$ be a correct  normal tree on $W$, of limit length, which is above $\lambda$ and based on $W|\zeta$. Let $\delta=\delta(\Tt)$, and suppose also that $W|\zeta\in M|\delta$ and $\Tt\in M$.
 Suppose that there is $\eta<\delta$ such that $W|\zeta\in M|\eta$ and $\eta$ is a strong $\vec{\kappa}^M$-cutpoint of $M$, meaning that for each $E\in\es^M$,
 if $\crit(E)\leq\eta<\lh(E)$ then $E=\kappa_\alpha^M$ for some $\alpha<\theta$.
 Suppose that $\vec{j}\in M|\delta$. Suppose that $M|\delta\sats$ ``there is a largest cardinal, and it is $\leq\eta$''. Suppose that $\Tt'$,
 which is just $\Tt$ but as a tree on $W'=W|\zeta$, is definable from parameters over $M|\delta$,
 and $M|\delta$ is extender algebra generic over $M(\Tt)$.

 Let $Q$ be a premouse such that $M(\Tt)\ins Q$.
 We attempt to define the $\star$-translation $Q^\star$ of $Q$ (which depends also on $\Tt$,
 $\vec{\kappa}^M$, $\vec{\kappa}^W$, and $\vec{j}$) just like in \cite[\S8]{odle_v2}, except that we work with $\lambda$-indexing (as $Q$ is $\lambda$-indexed) and the following difference.
 For extenders $E\in\es_+^Q$ with $\crit(E)=\kappa_\alpha^W<\delta<\lh(E)$ for some $\alpha<\theta$,
 we say that $E$ is \emph{strongly overlapping $\delta$},
 and here we do not proceed by forming an ultrapower by $E$;
 instead we set
 $(Q|\lh(E))^\star$ to be the unique active premouse $N$ such that $N^{\passive}=(Q||\lh(E))^\star$
 and $\crit(F^N)=\kappa_\alpha^M$
 and $F^N\rest(\kappa_\alpha^M)^{+M}=E\com j_\alpha$, if this does determine a premouse; otherwise $(Q|\lh(E))^\star$ is simply $((Q||\lh(E))^\star,E\com j_\alpha)$ (which is not a premouse).

So, we compare the attempted $\star$-translation versus $M$, above $\delta$,
comparing extenders for their action on the ordinals, up to equivalence modulo composition with $j_\alpha$, in case we have strongly overlapping extenders on both sides.
Given this, suppose that the comparison ends up non-dropping on both sides, hence producing  proper class models $N_W$ and $N_M$ (recall that these are distinct, because we did not ``fully'' compare; they just ``agree'' ``above $\delta$''). Let  $\Gamma$ be a proper class of  ordinals which is fixed pointwise by both iteration maps
$j_W:W\to N_W$ and $j_M:M\to N_M$.
 We want to show that the main branches do not use extenders overlapping $\delta$; given this, the usual fine structural calculations coming from $\star$-translation yield a contradiction.

 Let $E$ be the first extender used along the $W$ side, and $F$ the first used along the $M$ side.
Let $\kappa_E=\crit(E)$ and $\kappa_F=\crit(F)$.
Note then that, by our anti-large cardinal assumption (that $M$ has no strong limit of strongs),
 $j_W(\kappa_E)$ is the least strong cardinal of $N_W$ which is $>\delta$, and likewise $j_M(\kappa_F)$ is the least strong cardinal of $N_M$ which is $>\delta$. But since $N_W$ agrees with $N_M$ above $\delta$ regarding extenders with critical point $>\delta$, it follows that $j_W(\kappa_E)=j_M(\kappa_F)$.

 Let $\alpha$ be  such that $\kappa_E=\kappa_\alpha^W$
  and $\beta$ such that $\kappa_F=\kappa_\beta^M$.
  We claim that $\alpha=\beta$.
  For $\alpha$ is the ordertype of the strongs of $N_W$ which are $<\delta$,
  and this is encoded into $\es^{N_W}\rest(\delta,j_W(\kappa_E))$.
  To see this, let $\gamma^W\in(\delta,j_W(\kappa_E))$
  be a cutpoint relative to
  $\vec{\kappa}^W\rest\alpha$;
  that is, the only extenders in $\es^{N_W}$ overlapping $\gamma^W$,
  have critical point among $\vec{\kappa}^W\rest\alpha$.
  (Such $\gamma$'s are cofinal in $j_W(\kappa_E)$.) Let $\gamma^M$
  be defined analogously for $M$ and $\beta$. Let $\gamma=\max(\gamma^W,\gamma^M)$. Then if $E\in\es^{N_W}$
  and $\crit(E)<\delta<\gamma$,
  then $E$ overlaps $\gamma^W$, so
$\crit(E)=\kappa^W_\xi$
for some $\xi<\alpha$;
similarly, letting $E'\in\es^{N_M}$ with $\lh(E')=E$,
we have $\crit(E')<\gamma^M<\lh(E')$,
so $\crit(E')=\kappa^M_{\xi'}$
for some ${\xi'}<\beta$.
But $\xi=\xi'$.
On the other hand, for each $\xi<\alpha$,
we can find cofinally many $E\in\es^{N_W}\rest j_W(\kappa_E)$
with $\crit(E)=\kappa_\xi^W$,
and likewise for $M$.
It follows that $\alpha=\beta$.

Now we have $i_E^W(\kappa_E)\leq j_W(\kappa_E)$ and $i_F^M(\kappa_F)\leq j_M(\kappa_F)$.
If $i_E^W(\kappa_E)=i_F^M(\kappa_F)$ then
 we used $E$ and $F$ in the comparison
 at the same stage, so it must have been that $E\com j_\alpha\neq F$.
 Now this is precisely where the properties of our maps need to provide a contradiction: in the present circumstances, we should know that we must have $E\com j_\alpha=F$.
 (That is, we have trees $\Tt$ on $M$,
 with last model $W$, followed by a normal continuation $\Uu$ (arising from comparison and resulting from the $\star$-translation iterability proof), and a normal tree $\Vv$ on $M$,
 with $\Tt$ within $M$, and extender algebra genericity as we started with,
 etc, and the final models with the agreement arising from comparison up to the indices of $E$ and $F$,
 which are \emph{both} indexed at the same point; under those circumstances, we assume that in fact $E\com j_\alpha=F\rest\OR$. This uniquely determines $F$ from $E$, or $E$ from $F$. This is exactly what results from the circumstances in  \cite{vm2} and \cite{vmom}.)

So either $\lambda(E)<\lambda(F)$
or vice versa.
Suppose $\lambda(E)<\lambda(F)$; otherwise it is symmetric. We claim that $F\rest \lambda(E)$ is a whole proper segment of $F$
which is not in $N_M$, a contradiction.
For
\[ \OR\cap\Hull^{N_W}(\Gamma\cup\lambda(E))=\OR\cap\Hull^{N_M}(\Gamma\cup\lambda(E)), \]
since $\delta<\lambda(E)$
and $N_M$ is generic over $N_W$
for the extender algebra at $\delta$.
But because $j_W(\kappa_E)=i_{UN_W}(\lambda(E))=j_M(\kappa_E)\geq\lambda(F)$
where $U=\Ult(W,E)$
and $i_{UN_W}:U\to N_W$
is the tail of the iteration map,
these hulls contain no ordinals in the interval $[\lambda(E),\lambda(F))$,
which implies that $F\rest\lambda(E)$
is whole.

\section*{Acknowledgements}

Funded by the Austrian Science Fund (FWF) [10.55776/Y1498].

\bibliographystyle{plain}
\bibliography{../bibliography/bibliography}

\end{document}